\documentclass[12pt,fleqn]{article}

\usepackage{csquotes}
\usepackage[utf8]{inputenc}
\usepackage[T1]{fontenc}
\usepackage{mathtools}
\usepackage{amsfonts,amsmath, amsthm,amssymb}
\usepackage{appendix}
\usepackage{bbding}

 \usepackage{fouriernc}
\usepackage{mathrsfs} 
\usepackage{stmaryrd}
\usepackage{stmaryrd}
\usepackage{bbding}
 \usepackage{latexsym}
\usepackage{esint}
\usepackage[margin=17mm]{geometry} 
 \numberwithin{equation}{section}
\usepackage{color}
 \definecolor{db}{rgb}{0.0,0.0,0.8} 
\definecolor{dg}{rgb}{0.0,0.55,0.14}
\definecolor{dr}{rgb}{0.5,0,0.07}
\def\be{\begin{equation}}
\def\ee{\end{equation}}

\usepackage{hyperref}
\hypersetup{frenchlinks=true,colorlinks=true,linkcolor=db,citecolor=dg,
 bookmarks=true}

 \usepackage{enumerate}
\newcounter{step}

\def\ben{\begin{enumerate}}
\def\bena{\begin{enumerate}[a)]}
\def\een{\end{enumerate}}
\def\bit{\begin{itemize}}
\def\iit{\end{itemize}}
\def\dist{\operatorname{dist}}
\def\Lip{\operatorname{Lip}}

\def\tr{\operatorname{tr}}

\DeclareMathAlphabet{\mathonebb}{U}{bbold}{m}{n}

\def\R{{\mathbb R}}
\def\N{{\mathbb N}}

\def\Z{{\mathbb Z}}

\def\fo{\forall\, }

\def\va{\varphi}
\def\d{\displaystyle}
\def\im{\imath}
\def\ve{\varepsilon}
\def\p{\partial}
\def\l{\label}
\def\O{\Omega}

\def\na{\nabla}

\def\so{{\mathbb S}^1}

\theoremstyle{definition}
\newtheorem{theo}{Theorem}
\newtheorem{prop}{Proposition}[section]
\newtheorem{coro}[prop]{Corollary}
\newtheorem{lemm}[prop]{Lemma}
\theoremstyle{definition}

\theoremstyle{definition}
\newtheorem{rema}[prop]{Remark}
\theoremstyle{definition}

\def\beq{\begin{equation}}
\def\eeq{\end{equation}}
\def\be*{\begin{equation*}}
\def\ee*{\end{equation*}}
\def\epr{\end{proof}}
\def\bpr{\begin{proof}}

\def\wsp{W^{s,p}}
\def\O{\Omega}
\def\s1{{\mathbb S}^1}
\def\wsps1{\wsp (\O\, ; \s1)}
\def\R{{\mathbb R}}
\def\calr{\mathscr{R}}
\def\cinf{C^\infty(\overline\O\, ; \s1)}
\def\cinn{C^\infty(\overline\O\, ; N)}
\def\p{\partial}
\def\d{\displaystyle}
\def\wspn{W^{s, p}(\Omega\, ; N)}
\def\Z{\mathbb Z}

\date{\today }
\title{Density in $\wsp (\O ; N) $}

\author{Ha\"\i m Brezis$^{(1), (2)}$, Petru Mironescu$^{(3)}$\Envelope }

\begin{document}
\maketitle

\begin{abstract}
Let $\O$ be a smooth bounded domain in $\R^n$, $0<s<\infty$ and $1\le p<\infty$. We prove that $C^\infty(\overline\O\, ; \s1)$ is dense in $\wsps1$ except when $1\le sp<2$ and $n\ge 2$. The main ingredient is a new approximation method for $\wsp$-maps when $s<1$. With $0<s<1$, $1\le p<\infty$ and $sp<n$, $\O$ a ball, and $N$ a general compact connected manifold, we prove that $C^\infty(\overline\O\, ; N)$ is dense in $W^{s,p}(\O\, ; N)$   if and only if $\pi_{[sp]}(N)=0$. This supplements analogous results obtained by Bethuel when $s=1$, and by Bousquet, Ponce and Van Schaftingen when $s=2,3,\ldots$ [General domains $\O$ have been treated by Hang and Lin when  $s=1$; our approach allows to extend their result to $s<1$.] The case where $s>1$, $s\not\in\N$, is still open.
\end{abstract}

\section{Introduction}
\l{sec1}

Let $\O$ be a smooth bounded domain in $\R^n$, $n\ge 2$. [The questions we will consider are already interesting when $\O$ is a cube or a ball.] The first topic that we will address is whether  $\cinf$ is dense in $\wsps1$.
Here, $s>0$ and $1\le p<\infty$, and we let
\be*
W^{s,p}(\O ; \so)=\{ u\in W^{s,p}(\O ; \R^2);\, |u(x)|=1\text{ a.e.}\};
\ee*
for a set $N\subset\R^m$, we define  $W^{s,p}(\O ; N)$ similarly.

Of special interest to us is the case where $0<s<1$. Recall that in this case a standard norm on $W^{s,p}(\O)$ is $u\mapsto \|u\|_{L^p}+|u|_{\wsp}$, where
\be*
|u|_{\wsp}^p=\int_\O\int_\O \frac{|u(x)-u(y)|^p}{|x-y|^{N+sp}}\, dxdy.
\ee*

When $s>1$ is not an integer, we write $s=m+\sigma$, $m\in\N$, $0<\sigma<1$, and then a standard norm on $\wsp$ is $u\mapsto \|u\|_{L^p}+\|D^mu\|_{W^{\sigma, p}}$.

In this direction, our main result is the following.

\begin{theo}\label{thma}$\cinf$ is dense in $\wsps1$ when $sp<1$ or $sp\ge 2$.\\
If $1\le sp<2$, then $\cinf$ is not dense in $\wsps1$.
\end{theo}

\medskip
\noindent
Many special cases were already known (see the beginning of Section \ref{sec3}), but
the case where $n\ge 3$, $s<1$ and $2\le sp<n$ was left open (see \cite[Conjecture 2]{bm}).
This is an interesting and unusual situation where density holds and lifting fails; more precisely, there exists some $u\in\wsps1$ which cannot be written as $u=e^{\im\va}$ with $\va\in W^{s,p}(\O ; \R)$ \cite{bbm1}. 

The proof of Theorem \ref{thma}, which is presented in Section \ref{sec3}, relies on a new approximation result, valid only when $0<s<1$, which is discussed below (this is the content of Theorems \ref{thmd} and \ref{thme}). This original construction has its own interest and we believe that it might be useful in other contexts. An important  feature of Theorem \ref{thmd} is that it does not use any kind of smoothing or averaging. Hence it is especially appropriate in situations where maps take values into an arbitrary given set -- not necessarily a manifold.

\begin{rema}
\l{rem11}
A completely different proof of Theorem \ref{thma} for the case $n\ge 3$, $s<1$ and $2\le sp<n$ can be found in \cite{bmbook}. The main ingredient is the (non trivial) factorization theorem which asserts that each $u\in\wsps1$ can be written as $u=e^{\im\va}v$, with $\va\in W^{s,p}(\O ; \R)$ and $v\in W^{1,sp}(\O ; \so)$ \cite{mironescucras2}, \cite{bmbook}. 
\end{rema}

\begin{rema}
\label{r1a}
In the case where $1\le sp<2$, the reader may wonder what is the closure of $\cinf$ into $\wsp$. This question is answered in \cite{bmbook}. Roughly speaking, we are able to define a  distributional Jacobian $Ju$ for every $u\in\wsps1$ with $1\le sp<2$, and then
\be*
\overline{\cinf}^{\wsp}=\{ u\in\wsps1;\, Ju=0\}.
\ee*
This is the $\so$ fractional counterpart of a result of Bethuel  for maps in $H^1(B^3 ; {\mathbb S}^2)$ \cite{betdip}.

\end{rema}

\medskip
In the range $1\le sp<2$, the substitute of $\cinf$ for density purposes is the following class, inspired by the important work of Bethuel and Zheng \cite{bz} and Bethuel \cite{bet}:
\be*
\calr_{s,p} =\{ u\in\wsps1;\, u \text{ is smooth outside some finite union of }(n-2)- \text{manifolds}\}.
\ee*

For completeness, we recall the following known result.
\begin{theo}
\label{oldthm}
Let $n \ge 2$ and $s>0$. Assume that $1\le sp<2$. Then $\calr_{s,p}$ is dense in $\wsps1$.
\end{theo}
 
If $s=1$, Theorem  \ref{oldthm} was obtained by Bethuel and Zheng \cite{bz} when $n=2$ and by Bethuel \cite{bet} when $n\ge 3$. Other special cases were treated by Hardt, Kinderlehrer and Lin \cite{hkl} and by Rivi\`ere \cite{r}. In \cite{bbm}, Theorem \ref{oldthm} was proved for  $\d s=\frac 12$ and $p=2$; the argument in \cite{bbm} extends readily to the full range $0<s<1$, $1\le sp<2$; this is done in \cite{bmbook}. Finally, when $s>1$ Theorem \ref{oldthm}   was established by Bousquet \cite{pb}.

\medskip
We next consider the more general situation where the target space $\s1$ is replaced by a  compact connected manifold $N$ without boundary, embedded in $\R^m$.
To start with, we prove that when $n=1$,   $C^\infty(\overline\O\, ; N)$ is always dense in $\wsp (\O\, ; N)$; see Corollary \ref{n=1}.
Our main result in Section \ref{sec4} is a fractional version of a remarkable result of Bethuel \cite{bet}, which asserts that, when $n\ge 2$ and $1\le p<n$, the class
\be*
\calr_{1,p} =\{ u\in W^{1,p}(\O\, ; N);\, u \text{ is smooth outside some finite union of }(n-[p]-1)- \text{manifolds}\}
\ee*
is dense in $W^{1,p}(\O ; \, N)$ (with $[\ ]$ denoting the integer part).
When $0<s<1$, we prove 
\begin{theo}
\label{thmb}
Assume that $n\ge 2$, $0<s<1$ and $sp<n$. Then 
\be*
{\mathscr R}_{s, p}=\{u\in W^{s,p}(\Omega\, ; N)\ ;\ u\text{ is continuous outside a finite union of }(n-[sp]-1)-\text{ manifolds}\}
\ee*
is dense in $W^{s,p}(\Omega\, ; N)$.
\end{theo}
\begin{rema}
Let $n\ge 2$ and $s>0$. Assume that either $sp<1$ or $sp\ge n$. Then $C^\infty(\overline\O ; N)$ is dense in $\wsp (\O ; N)$. For the case $sp<1$, see Section \ref{sub32}; the case $sp\ge n$ is handled as in \cite{su}, \cite{bn}. On the other hand, given any $s>0$ and $p\ge 1$ such that $1\le sp<n$, there exists some manifold $N$ such that $C^\infty(\overline\O ; N)$ is not dense in $\wsp (\O ; N)$; it suffices to take $N={\mathbb S}^{[sp]}$ and apply Theorem \ref{thmc} below. 
\end{rema}
\begin{rema}
With more work, it is possible to improve the conclusion of Theorem \ref{thmb} by replacing, in the definition of the class ${\mathscr R}_{s, p}$, \enquote{$u$ continuous} by \enquote{$u$ smooth}. This requires a smoothing procedure. Such a procedure with $s=1$ (in the spirit of the proof of the $H=W$ theorem of Meyers and Serrin) is described in \cite{brezisli}. This can be adapted to arbitrary $s$, but  will not be detailed  here.
\end{rema}
\begin{rema}
When $1<p<\infty$ and $s=1-\d\frac 1p$, Theorem \ref{thmb} was proved by Mucci \cite{mucci}, using a method inspired by Bethuel \cite{bet} and completely different from ours. It is not clear whether this kind of method might lead to a proof of Theorem \ref{thmb}. 
\end{rema}

Recall the following result due to Bethuel \cite{bet}: 
Assume that $\O$ is a ball (or a cube). For $p<n$, $\cinn$ is dense in $W^{1,p}(\Omega\, ; N)$ if and only if $\pi_{[p]}(N)=0$. The extension of this result to $s=2,3,\ldots$ can be found in Bousquet, Ponce and Van Schaftingen \cite{bpvs}.
A partial analog in our situation is
\begin{theo}
\label{thmc}
Assume  that $0<s<1$,  $sp<n$ and that $\Omega$ is a ball. Then $\cinn$ is dense in $\wspn$ if and only if $\pi_{[sp]}(N)=0$.
\end{theo}

For special target manifolds $N$, Theorem \ref{thmc} was obtained  by Bousquet, Ponce and Van Schaftingen \cite{bpvs2}.

When $\O$ is more complicated, one may still give  necessary and sufficient conditions for the density of $\cinn$ in $\wspn$. Indeed, when $s=1$ such conditions (depending on $[p]$) were discovered by Hang and Lin \cite[Theorem 6.3]{hl}. The proof of Theorem \ref{thmc} shows that the same conditions govern the  case $s<1$, provided we replace $[p]$ by $[sp]$.

Two natural questions remain open:

\noindent
{\bf Open Problem 1.}
Assume that $s>1$ is not an integer and that $sp<n$. Is it true that $\mathscr R_{s,p}$ is dense in $\wspn$?

By Theorem \ref{oldthm}, the answer is positive when $N=\s1$. This is also the case when $N$ is arbitrary and  $s=2,3,\ldots$ (Bousquet, Ponce and Van Schaftingen \cite{bpvs}). 
However, the general case is still open even for simple targets such as $N={\mathbb S}^2$.

\noindent
{\bf Open Problem 2.}
Assume  that $s>1$  is not an integer,  $sp<n$ and that $\Omega$ is a ball. Is it true that $\cinn$ is dense in $\wspn$ if and only if $\pi_{[sp]}(N)=0$?

\bigskip
\noindent
{\bf The main idea for the proof of Theorem \ref{thma}.} 
We describe here, without proof, the basic tool, namely approximation by piecewise $j$-homogeneous maps. 

For simplicity, we explain our construction first in $3$-d.
Let $Q=[-1,1]^3$ and let $g:\p Q\to\R^m$. We may extend $g$ to a map $h:Q\to\R^m$ through the formula  $h(x)=\d g\left(\frac x{|x|}\right)$, where $|\ |$ stands for the sup norm. The map $h$ is the  \enquote{homogeneous} extension of $g$.\\
Let now $K$ be the $1$-dimensional skeleton (=union of edges) of $Q$ and let $g:K\to\R^m$. One may extend $g$ to $Q$ in two steps: first, by homogeneous extension on each face of $\partial Q$, next by homogeneous extension from $\p Q$ to $Q$. This extension will be again called \enquote{homogeneous}.\\
Similarly, given a map defined on the $0$-skeleton (=union of vertices) of $Q$, one may extend it in  three steps  \enquote{homogeneously}  to $Q$.\\
More generally, if $K$ is the $j$-skeleton of the cube $Q=[-1,1]^n$ and $g:K\to\R^m$, then $g$ has a  \enquote{homogeneous} extension $h:Q\to\R^m$, obtained in $(n-j)$ steps. Such a map will be called $j$-homogeneous.\\
One can also consider the more general situation where the cube is replaced by a finite mesh ${\mathscr C}=\d \cup_iQ_i$ and extend maps defined on the $j$-skeleton of ${\mathscr C}$ to \enquote{piecewise $j$-homogeneous} maps on ${\mathscr C}$.

\medskip
We may now state our main approximation result.

Let $F\subset\R^m$ be an arbitrary set, $0<s<1$, $sp<n$.  
Let $\O$, $\omega$ be two smooth open bounded subsets of $\R^n$ such that $\overline\O\subset\omega$ and let $f\in\wsp (\omega\, ; F)$.
\begin{theo}
\label{thmd}
Assume that $0<s<1$ and $sp<n$. Let $j$ be an integer such that $[sp]\le j\le n-1$. Then there exists a sequence $\{ {\mathscr C}^k\}$ of finite meshes, such that $\overline\O$$\subset$${\mathscr C}^k$$\subset\omega$, and a 
sequence  of maps $f_k:{\mathscr C}^k\to F$ such that:
\begin{enumerate}[a)]
\item
Each $f_k$ is piecewise $j$-homogeneous on ${\mathscr C}^k$, i.e., $f_k$ is the  $j$-homogeneous extension of its restriction to the $j$-dimensional skeleton 
${\mathscr S}^k$ 
of ${\mathscr C}^k$.
\item
 Each $f_k$ belongs to $\wsp ({\mathscr C}^k ; F)$.
 \item
$f_k\to f$ in $\wsp (\O)$ as $k\to\infty$.
\end{enumerate}
\end{theo}
%
%\medskip
\noindent
When $j=n-1$, the main ingredient in the proof of Theorem \ref{thmd} is presented in Section \ref{appa}; Section \ref{appb} treats the case where $j\le n-2$ and contains the proof of Theorem \ref{thmd}.\\
Of special interest to us will be the case where $j=[sp]$. When $sp$ is not an integer,  the restriction of $\d {f_k}$ to ${{\mathscr S}^k}$ is continuous. In particular, each $f_k$  is continuous on ${\cal C}^k$ outside some finite union of $\ell$-dimensional cubes, with $\ell=n-[sp]-1$. This need not be the case when $sp$ is an integer.\\
 When $F$ is a compact manifold and $j=[sp]$, Theorem \ref{thmd} can be considerably improved:
\begin{theo}
\label{thme}
Assume that $0<s<1$, $1\le sp<n$ and that $F$  is  a compact manifold without boundary. Let $j=[sp]$. Then there exist  sequences $\{ {\mathscr C}^k\}$ and $\{f_k\}$ such that a)-c) hold and, in addition,
\begin{enumerate}[a)]
\setcounter{enumi}{3}
\item
For each $k$, the restriction to ${\mathscr S}^k$ of  ${f_k}$ is Lipschitz.
\end{enumerate}
\end{theo}

\noindent
The proof, presented in Section \ref{appe}, uses tools developed in Sections \ref{appc} and \ref{appd}.
\begin{rema}
We emphasize the fact that these approximation results are specific to the case where $0<s<1$. For example, the map $u(x_1,x_2)=x_1$ cannot be approximated in $W^{1,1}((0,1)^2)$ by piecewise $1$-homogeneous maps associated  to meshes contained in $(-1,2)^2$; see Lemma \ref{lema9} in Section \ref{appa}. One may extend the argument given there in order to prove that, for any $p$ and $j$,  non constant smooth maps cannot be approximated in $W^{1,p}(\O)$ by piecewise $j$-homogeneous maps associated  to meshes contained in $\omega$.

The technique of homogeneous extensions has roots in  White \cite{white}, who used it in the study of topological invariants of $W^{1,p}$ maps between manifolds. Homogeneous extensions were also used by Bethuel \cite{bet} in his proof of the $W^{1,p}$ versions of Theorems \ref{thmb} and \ref{thmc}. We point out that our method is different from Bethuel's one. His method involves smoothing of $u$ on a set $A\subset\O$ such that $\O\setminus A$ is small.  Homogeneous extensions are used only in $\O\setminus A$. In our approach, homogeneous extensions are used in all of $\O$.
\end{rema}

%In connection with the approximation of such maps, 
%On the other hand, although the use of homogeneous extensions appears already in Bethuel \cite{bet} , 

The main results of this paper have been mentioned in personal communications starting in 2003 and a sketch of proof can be found in \cite{pm2004} and \cite{pm2007}. Since then, several papers have addressed related questions.

\subsubsection*{Acknowledgments} 
HB was partially supported by NSF grant DMS-1207793 and also by grant number 238702 of the European Commission (ITN, project FIRST).
PM was partially  supported by the ANR project \enquote{Harmonic Analysis at its Boundaries},   ANR-12-BS01-0013-03, and by the LABEX MILYON (ANR-10-LABX-0070) of Universit\'e de Lyon,
within the program \enquote{Investissements d'Avenir} (ANR-11-IDEX-0007) operated 
by the French National Research Agency (ANR). Part of this work was done while PM was visiting the Mathematics Department at Rutgers University. He warmly thanks HB and the department for their hospitality.  We are grateful to Pierre Bousquet, Augusto Ponce and Jean Van Schaftingen for useful discussions.
We are particularly indebted to Pierre Bousquet for his extremely careful and constructive reading of the manuscript.

\tableofcontents

\section{Proof of Theorem \ref{thma} using Theorem \ref{thme}}
\l{sec3}
We start by presenting more details about the cases already known.
\begin{enumerate}[a)]
\item Assume that $1\le sp<2$ and that $0\in\O\subset\R^2$. Let $\d u(x)=\frac x{|x|}$; here, $|\ |$  stands for the Euclidean norm. One may check that $u\in\wsps1$.
Indeed, assume first that $s<1$. We have $u\in W^{1,q}\cap L^\infty$, for each $q<2$. To obtain that $u\in\wsp$, we take $sp<q<2$ and use the Gagliardo-Nirenberg-Sobolev embedding $ W^{1,q}\cap L^\infty\subset \wsp$. Assume next that $s>1$. Since $\nabla u$ is homogeneous of degree $-1$ and smooth outside the origin, we have $\nabla u\in W^{\sigma,q}$ whenever $(1+\sigma)q<2$; this is obtained by arguing as in  \cite[proof of Lemma 1 (ii), p. 44]{rs}. In particular, $\nabla u\in W^{s-1,p}$, so that $u\in\wsp$.

We claim that there is no sequence $\{u_k\}\subset C^\infty(\overline\O\, ; \s1)$ such that $u_k\to u$ in $\wsp$. Argue by contradiction as in \cite{su}. Then there is some small $r>0$ such that, possibly after passing to a subsequence, $u_k\to u$ in $\wsp (C(0 , r))$; here, $C(0 , r)$ is the circle of radius $r$ centered at the origin.

If $sp>1$, this implies uniform convergence of $u_k$ to $u$ on $C(0 , r)$. Therefore, deg$(u_k, C(0 , r))\to$ deg$(u, C(0 , r))=1$. However, deg$(u_k, C(0 , r))=0$ since $u_k$ is smooth in $\O$.

When $sp=1$, convergence need not be uniform anymore. However, we know that $\wsp ( C(0 , r))$$\subset$
VMO with continuous embedding, see  e. g. \cite{bn}. We conclude as above using the continuity of the degree under BMO convergence \cite{bn}.

When $\O\subset\R^n$, with $n\ge 3$, one argues similarly using the map $u(x)=\d\frac {(x_1,x_2)}{|(x_1,x_2)|}$, $x=(x_1,\ldots, x_n)$.
\item Assume that $sp<1$. Let $u\in\wsps1$. By \cite{bbm1},  one may write $u=e^{\imath\varphi}$, with $\varphi\in\wsp (\O\, ;\R)$. If $\{\varphi_k\}\subset C^\infty(\overline\O\, ; \R)$ converges to $\varphi$ in $\wsp$, it is immediate that $u_k:=e^{\imath\varphi_k}\to u$ in $\wsp$ (see e.g. \cite[proof of (5.43)]{bbm}).
\item Assume that $s\ge 1$ and $sp\ge 2$. Then we may write $u=ve^{\imath\varphi}$, with $v\in  C^\infty(\overline\O\, ; \s1)$ and $\varphi\in \wsp\cap W^{1, sp}(\O\, ; \R)$ \cite[Cases 2 and 3, pp. 128-129]{bm}. Let now $\{\varphi_k\}\subset C^\infty(\overline\O\, ; \R)$ converge to $\varphi$ in $\wsp\cap W^{1, sp}$. Then $e^{\imath\varphi_k}\to e^{\imath\varphi}$ in $\wsp$ \cite[Theorem 1.1']{bm1}, which immediately implies that $ve^{\imath\varphi_k}\to v e^{\imath\varphi}=u$ in $\wsp$.
\item Assume that $sp\ge n$. Then density of $C^\infty(\overline\O\, ; \s1)$ in $\wsps1$ is well-known \cite{su} via the Sobolev embeddings $\wsp\subset C^0$ when $sp>n$ and $\wsp\subset$VMO when $sp=n$. For further use, we note that density holds also when $\s1$ is replaced by an arbitrary compact manifold.
\end{enumerate}
We next turn to the case  $0<s<1$ and $2\le sp<n$, which is the only one really new.

\medskip
\noindent
{\it Proof of Theorem \ref{thma}.} 
We assume that $0<s<1$ and $2\le sp<n$. 
Let $q=sp$. Recall
the Gagliardo-Nirenberg type embedding \cite[Appendix D]{bbm1}
\begin{equation}
\label{gn1}
W^{1,q}\cap L^\infty\subset W^{s,p}
\eeq
(valid since $q>1$). This embedding is continuous in the sense that
\beq
\label{gn2}
\text{ if } f_k\to f \text{ in }W^{1,q}\text{ and }\|f_k\|_{L^\infty}\le C, \text{ then }f_k\to f \text{  in }\wsp.
\eeq
On the other hand, since $q\ge 2$, a result of Bethuel and Zheng  \cite{bz} asserts that $\cinf$ is dense in $W^{1,q}(\O ; \so)$. Combining this with \eqref{gn1}-\eqref{gn2}, we find that
\be*
W^{1,q}(\Omega\, ; \s1)=\overline{\cinf}^{W^{1,q}}\subset\overline{\cinf}^{\wsp}.
\ee*
Let now $u\in\wsps1$. We start by extending  $u$ to a neighborhood $\omega$ of $\overline\Omega$; this is achieved via reflections and yields a map $f\in\wsp (\omega\, ; \s1)$.\\
We next  claim that 
the maps $f_k$ given by Theorem \ref{thme}  are in $W^{1,r}$ for each $r<[sp]+1$.  In particular, we have $f_k\in  W^{1,q}$. To establish this fact we rely on the following
\begin{lemm}
\label{yb5}
Assume that $n\ge 2$ and let $U\subset\R^n$ be an open set. Let $K$ be a closed subset of $U$ such that ${\cal H}^{n-1}(K)=0$. Let $u\in W^{1,1}_{loc}(U\setminus K)$ be such that $\d \int_{U\setminus K}|\na u|<\infty$. Then $u\in W^{1,1}_{loc}(U)$ and the Sobolev gradient of $u$ is the Sobolev gradient of $u_{|U\setminus K}$. 
\end{lemm}

This result is proved in \cite[Lemma 2.15]{bm}. 
[For similar results, see e.g. \cite[Lemma 3]{koskela}, \cite[Introduction]{dupaigneponce}.] We apply this lemma with $U={\cal C}^k$ and
$K=\Sigma^k$, the set of discontinuity points of $f_k$; this is the \enquote{dual skeleton} of ${\cal S}^k$. Then $\Sigma^k$ is a finite union of $(n-j-1)$-dimensional cubes, and thus ${\cal H}^{n-1}(K)=0$. On the other hand, a straightforward calculation yields 
\beq
\label{g1,1}
|\na\, f_k(x)|\le \frac{C_k}{\dist (x,\Sigma^k)},\ \fo x\in{\cal C}^k\setminus \Sigma^k,
\eeq
and thus $\na\, f_k\in L^{r}$ when $r<j+1$. 
We find that $f_k\in W^{1,1}_{loc}({\cal C}^k)$, and actually $f_k\in W^{1,r}({\cal C}^k)$ when $r<j+1$.

 Thus
\be* \wsps1
\subset\overline{W^{1,q}(\Omega\, ; \s1)}^{\wsp}\subset \overline{\cinf}^{\wsp}.\hfill\square
\ee*

\section{The case of a general target manifold}
\l{sec4}
Here we will address several questions related to the space $\wsp (\Omega\, ; N)$, where   $\O$ is a smooth bounded domain in $\R^n$ and $N$ is a compact manifold without boundary embedded in $\R^m$.

\subsection{Proof of Theorem \ref{thmb} using Theorems \ref{thmd} and \ref{thme}}

We start by extending a map $u\in\wspn$ to a map $f\in\wsp (\omega\, ; N)$.\\
If $sp<1$, then the maps $f_k$ given by Theorem \ref{thmd} are piecewise constant, and thus in ${\mathscr R}_{s, p}$, and we are done.\\
Assume next that $sp\ge 1$. Let $q$ be such that $sp<q<[sp]+1$. Note that $1<q<n$ and that $[q]=[sp]$. As in the proof of Theorem \ref{thma},  \eqref{gn1} and \eqref{gn2} hold  (since $q>1$ and $q\ge sp$). Combining \eqref{gn1} and \eqref{gn2} with Bethuel's density result for the class ${\cal R}_{1,q}$ (valid since $q<n$), we find that 
 ${\mathscr R}_{1,q}\subset  {\mathscr R}_{s, p}$ and 
\be*
 W^{1,q}(\Omega\, ; N)=\overline{{\mathscr R}_{1,q}}^{W^{1,q}}\subset\overline{{\mathscr R}_{1,q}}^{\wsp} \subset \overline{{\mathscr R}_{s,p}}^{\wsp}.
\ee*
Since
the maps $f_k$ given by Theorem \ref{thme} are in $W^{1,q}$ (this uses the fact that $q<[sp]+1$), we obtain
\be*\wspn\subset \overline{W^{1,q}(\Omega\, ; N)}^{\wsp}\subset \overline{{\mathscr R}_{s,p}}^{\wsp}.\hfill\square
\ee*

\subsection{Proof of Theorem \ref{thmc} using Theorems \ref{thmd} and \ref{thme}}
\label{sub32}

We start with the case $sp<1$; here, the topological condition is that $N$ is connected, which is satisfied by assumption. As we will see, in this case  $\O$ could be any smooth domain.

If $N$ is a curve, then $N$ is diffeomorphic to $\s1$, and a straightforward argument reduces the problem to the one of the density of $\cinf$ in $\wsps1$, which follows from Theorem \ref{thma}.\\
Assume next that dim $N\ge 2$. Let $u\in\wspn$. We first extend it near $\overline\O$, next we consider a map $f_k$ as in Theorem \ref{thmd}. It suffices to prove that such a map, which is piecewise constant, can be approximated by smooth $N$-valued maps. Now $f_k$ assumes only finitely many values, say $a_1,\ldots, a_l$. Let $\Gamma\subset N$ be a smooth simple curve that contains $a_1,\ldots, a_l$. Then $f_k\in\wsp (\O\, ; \Gamma)$. By our discussion on curves, $f_k$ may be approximated by $\Gamma$-valued (thus $N$-valued) smooth maps.

\medskip
We now turn to the case $1\le sp<n$.

\medskip
\noindent
{\bf Condition  $\pi_{[sp]}(N)=0$ is necessary.}  Let $j=[sp]$. Argue by contradiction and let $v\in C^\infty ({\mathbb S}^{j}\, ; N)$ such that $v$ is not homotopic to a constant. Assume that $\Omega$ is the unit ball and let $u:\O\to N$, $\d u(x)=v\left(
\frac {(x_1,\ldots, x_{j+1})} {|(x_1,\ldots, x_{j+1})|}\right)$; here, $|\ |$  stands for the Euclidean norm. It is easy to see that $u\in W^{1,q}$ for each $q<j+1$, and thus $u\in \wsp$. As in the proof of  a) in Section \ref{sec3}, the stability of the homotopy class under uniform (or BMO) convergence implies that there is no sequence $\{u_k\}$ of smooth $N$-valued maps such that $u_k\to u$ in $\wsp$.

\medskip
\noindent
{\bf Condition  $\pi_{[sp]}(N)=0$ is sufficient.} It suffices to prove that each map $f_k$ given by Theorem \ref{thme} can be approximated by smooth maps. 
Let $q$ be such that $sp<q<[sp]+1$, so that $[q]=[sp]$. Then $\cinn$ is dense in $W^{1,q}(\O\, ; N)$, since  $\pi_{[q]}(N)=0$
and $\O$ is a ball \cite{bet}, \cite{hl}. The proof of Theorem \ref{thmb} implies that $\cinn$ is dense in $\wspn$.
\hfill$\square$
\begin{coro}
\label{n=1}
If $I$ is a bounded interval, then $C^\infty(\overline I ;\, N)$ is dense in $\wsp (I ; \, N)$ for each $s$ and $p$.
\end{coro}
\begin{proof}
When $sp<1$, density follows from Theorem \ref{thmc}. When $sp\ge 1$, we are in case d) discussed in Section \ref{sec3} and we still have density.
\end{proof}

\def\ve{\varepsilon}
\section{Approximation by homogeneous extensions}
\l{appa}

At the end of Section \ref{appb}, we will present two proofs of Theorem \ref{thmd}. The first one is quite long, but covers all the possible cases and has the advantage of introducing several  calculations which will prove useful in Sections \ref{appc}-\ref{appe}. 

The second proof, much shorter,  is valid under the additional assumption $j\ge 1$.   It  relies on two rather short calculations and on interpolation. While the same strategy could serve to prove some of the auxiliary results in later sections, e.g. Lemma \ref{cont1}, it is unclear 
whether this approach could be used in obtaining Lemmas \ref{l1ab} and \ref{ug1}, which are at the heart of the proof of Theorem \ref{thme}. If interpolation could help in obtaining Lemmas \ref{l1ab} and \ref{ug1}, then this approach  would lead to significantly shorter proofs of Theorems \ref{thmd} and \ref{thme}. 

For the convenience of the reader, the \enquote{long proof} of Theorem \ref{thmd} is split into two parts: this section  is devoted to approximation by piecewise $(n-1)$-homogeneous maps. Section \ref{appb} treats the case of piecewise $j$-homogeneous maps, with $j\le n-2$. The proofs of Theorem \ref{thmd} are presented at the end of Section \ref{appb}.

Throughout the remaining sections, $C$ will denote a constant depending only on $n$, $s$ and $p$. If necessary, we will enhance the dependence on the parameters by denoting $C=C(n,s,p)$, etc.

If $f:\R^n\to\R^m$, one may associate to $f$ a family $\{ f_{T,\ve}\}_{T\in\R^n,\, \ve>0}$ of piecewise $(n-1)$-homogeneous maps as follows: for each $T\in\R^n$, there exists exactly one horizontal (=with faces parallel to the coordinate hyperplanes) mesh of size $2\ve$ having $T$ as one of its centers. [The mesh consists of the cubes $T+2\ve\, K+(-\ve,\ve)^n$, with $K\in\Z^n$.] We restrict $f$ to the boundary of this mesh, next extend homogeneously this restriction to the cubes of the mesh. The map obtained by this procedure will be denoted $f_{T,\ve}$ or simply $f_T$ when $\ve$ is fixed. \\
Analytically, $f_{T,\ve}$ is defined as follows: let $ |\ |$ denote the $\sup$ norm in $\mathbb R^n$.  For $\varepsilon > 0$,
let 
\be*
Q_\varepsilon(X) = \{Y \in \mathbb R^n\ ;\ |Y - X| < \ve\},\  Q_\ve =
Q_\varepsilon (0).
\ee*

For a.e. $X\in \mathbb R^n$,
there exists a unique $K\in \mathbb Z^n$ such that $X \in Q_\varepsilon
(T+2\varepsilon K)$.  Then 
\be*
f_T(X) = f_{T,\varepsilon}(X) = f\left(T+2\varepsilon K + \varepsilon \frac{X-T - 2
\varepsilon K}{|X-T-2\varepsilon K|}\right).
\ee*

\medskip

This section is essentially devoted to the proof of
\begin{lemm}
\l{lema1}  Let $0< s< 1, 1 \leq p < \infty$ be
such that $sp < n$.  For each $f\in W^{s,p}(\mathbb R^n ; \mathbb R^m)$ there
are sequences $\varepsilon_k \to 0$ and $\{T_k\} \subset \mathbb R^n$ such
that $f_{T_k, \varepsilon_k}\to f$ in $W^{s,p}(\mathbb R^n)$.
\end{lemm}
\begin{proof}  We will establish the following estimate 
\beq
\l{e05091}
\frac{1}{\varepsilon^n} \int\limits_{Q_\varepsilon}\|
f-f_{T,\varepsilon}\|^p_{W^{s,p}}dT\leq a(\varepsilon) + b(\varepsilon),
\eeq
where
\beq
\l{e05092}
a(\varepsilon) \to 0 \text{ as } \varepsilon \to 0\ \text{and }
\int\limits^1_0\frac{b(\varepsilon)}{\varepsilon}\, d\varepsilon <\infty.
\eeq
Assume \eqref{e05091} proved for the moment.  Then \eqref{e05092} implies that, for a
sequence $\ve_k\to 0$, we have $a (\varepsilon_k)+ b(\varepsilon_k) \to 0.$ The conclusion of
Lemma \ref{lema1} is then an immediate consequence of \eqref{e05091}.

\medskip
We next turn to the proof of \eqref{e05091}. A warning about notation. The calculations below will involve multiple integrals. In order to make these calculations easier to follow, an integral of the form $\d \int_{A\times B} f(X,Y)\, dXdY$ will be  denoted $\d\int_A dX\int_B dY\, f(X,Y)$.

\smallskip
For the convenience of the reader, we split the proof of \eqref{e05091} into
several steps.

\bigskip
\noindent{\bf Step 1.}  We have
\beq
\label{a3}
A : = \frac{1}{\varepsilon^n}\int\limits_{Q_\varepsilon} \| f-f_{T,\varepsilon}
\|^p_{L^p} dT \to 0 \text{ as } \varepsilon\to 0. 
\eeq
\def\Z{\mathbb Z}
Indeed, since $(Q_\varepsilon(T+2\varepsilon K))_{K\in\mathbb Z^n}$ is an
a.e. partition of $\mathbb R^n$ and $f_T=f_{T+2\ve K}$ for $T\in\R^n$ and $K\in\Z^n$, we have

\beq
\label{a4}
\aligned
A=& \frac{1}{\varepsilon^n} \int\limits_{|T|<\varepsilon} dT \sum_{K\in \mathbb Z^n}\
\int\limits_{Q_\varepsilon(T+2\varepsilon K)} |f(X) -
f_T(X)|^p dX =  \frac{1}{\varepsilon^n} \int\limits_{\mathbb R^ n} dX \int\limits_{Q_\varepsilon(X)} |f(X) -
f_T(X)|^p dT=\\
=& \frac{1}{\varepsilon^n} \int\limits_{\mathbb R^n} dX \int\limits_{Q_\varepsilon}\left|f(X) -
f\left(X+Y-\varepsilon \frac{Y}{|Y|}\right)\right|^p dY= \frac{1}{\varepsilon^n} \int\limits_{Q_\varepsilon}\left\| f(\cdot) - f \left(\cdot + Y -
\varepsilon \frac{Y}{|Y|}\right) \right\|_{L^p}^p dY.
\endaligned 
\eeq
We note that $Y\in Q_\varepsilon \implies \d Y-\varepsilon \frac{Y}{|Y|} \in
Q_\varepsilon$.  Therefore,
\beq
\label{a5}
A\leq 2^n \sup \{ \|f (\cdot) - f(\cdot + Y)\|_{L^p}^p  ; |Y| \leq
\varepsilon\}.
\eeq
Inequality \eqref{a5} implies  \eqref{a3} and  completes Step 1.

\bigskip
In order to complete the proof of Lemma \ref{lema1}, it remains to estimate
\be*
\begin{aligned}
B:=& \frac 1{\ve^n}\int\limits_{Q_\ve}|f-f_T|_{\wsp}^p\, dT\\
 =&\frac{1}{\varepsilon^n} \int\limits_{Q_\varepsilon} dT \int\limits_{\mathbb R^n}\int\limits_{\mathbb R^n} dX dY
\ \frac{|[f(X) - f_T(X)]- [f(Y) - f_T(Y)]|^p}{|X-Y|^{n+sp}}.
\end{aligned},
\ee*
and more specifically to obtain an upper  bound of the form $B\le a(\ve)+b(\ve)$, with $a$ and $b$ as in \eqref{e05092}.

To this end, we use the inequalities
\be*
|[f(X) - f_T(X)]- [f(Y) - f_T(Y)]|^p\leq\begin{cases}
 C (|f(X) - f_T(X)|^p + |f(Y) - f_T(Y)|^p), & \text{if } |X-Y|>
\varepsilon
\\
 C (|f(X) - f(Y)|^p +|f_T(X) - f_T(Y)|^p), & \text{if } |X-Y|\leq
\varepsilon
\end{cases}.
\ee*
We find that
\beq
\label{bigdeal}
B\leq C(I+J+D),
\eeq
where
\be*
I=\frac{1}{\varepsilon^n}\int\limits_{Q_\varepsilon} dT \iint\limits_{|X-Y|>\varepsilon} dXdY
\ \frac{|f(X) - f_T(X)|^p}{|X-Y|^{n+sp}},
\ee*
\be*
J
=\frac{1}{\varepsilon^n}\int\limits_{Q_\varepsilon} dT \iint\limits_{|X-Y|<\varepsilon} dXdY
\ \frac{|f(X) - f(Y)|^p}{|X-Y|^{n+sp}}
\ee*
and
\beq
\label{defd}
D=\frac{1}{\varepsilon^n}\int\limits_{Q_\varepsilon} dT \iint\limits_{|X-Y|<\varepsilon} dXdY
\ \frac{|f_T(X) - f_T(Y)|^p}{|X-Y|^{n+sp}}.
\eeq

Thus our purpose is to establish the estimates 
\beq
\label{finest}
I\le a(\ve)+b(\ve),\ J\le a(\ve)+b(\ve),\ D\le a(\ve)+b(\ve),
\eeq
with $a(\ve)$ and $b(\ve)$ as in \eqref{e05092}.

Clearly,
\be*
J= 2^n \iint\limits_{|X-Y|<\varepsilon} dXdY
\ \frac{|f(X) - f(Y)|^p}{|X-Y|^{n+sp}} \to 0 \text{ as } \varepsilon \to 0.
\ee*
Therefore, it remains to estimate $I$ and $D$.

\bigskip
\noindent{\bf Step 2.}  Estimate of $I$\\
We have 
\be*
I= \frac{C}{\varepsilon^n} \int\limits_{Q_\varepsilon} dT \int\limits_{\mathbb R^n}  dX
\ \frac{|f(X) - f_T(X)|^p}{\varepsilon^{sp}}=
\frac{C}{\varepsilon^{n+sp}}\int\limits_{Q_\varepsilon} dT\int\limits_{\mathbb R^n}dX\  |f(X) - f_T(X)|^p.
\ee*
As in the proof of \eqref{a4}, we find that 
\be*
I= \frac{C}{\varepsilon^{n+sp}} \int\limits_{\mathbb R^n}  dX \int\limits_{Q_\varepsilon} dY
 \ \left|f(X) - f\left(X+ Y-\varepsilon\frac{Y}{|Y|}\right)\right|^p.
\ee*

We next introduce a change of variables widely used in what follows. We write $Y=\delta\, \omega$ (or $Y=r\, \omega$ or $Y=\lambda\,\omega$ elsewhere), where $\delta=|Y|=\max\{ |Y_1|,\ldots, |Y_n|\}$ and $|\omega|=1$. We will denote the new variables $\delta$ and $\omega$ as polar coordinates. These are not the \enquote{Euclidean} polar coordinates, but rather \enquote{cubic} polar coordinates adapted to the norm $|\ |$. Let us note that the Jacobian of these coordinates is still $\delta^{n-1}\, d\delta d\omega$.

\smallskip
In polar coordinates, the expression of $I$ becomes
\beq
\label{ue6}
\aligned
I&= \frac{C}{\varepsilon^{n+sp}} \int\limits_{\mathbb R^n}  dX \int\limits^\varepsilon_0
\delta^{n-1}d\delta \int\limits_{|\omega|=1} d\omega
\ |f(X) - f(X+ \delta\omega -\varepsilon\omega)|^p=\\
&= \frac{C}{\varepsilon^{n+sp}} \int\limits_{\mathbb R^n}  dX
\int\limits^\varepsilon_0(\varepsilon-\lambda)^{n-1} d\lambda\int\limits_{|\omega|=1}
d\omega\ |f(X) - f(X- \lambda \omega)|^p=\\
&= \frac{C}{\varepsilon^{n+sp}} \int\limits_{\mathbb R^n}  dX \int\limits_{|Y|<\varepsilon}dY \, 
\frac{(\varepsilon - |Y|)^{n-1}}{|Y|^{n-1}}\  |f(X) - f(X-Y)|^p.
\endaligned
\eeq
Since clearly
\beq
\label{ue7}
\frac{(\varepsilon - |Y|)^{n-1}}{\varepsilon^{n+sp} |Y|^{n-1}} \leq
\frac{1}{|Y|^{n+sp}}\text{ if } |Y|<\varepsilon\le 1,
\eeq
we find that
\be*
I\leq C \iint\limits_{|X-Z|<\varepsilon} dXdZ\
\frac{|f(X) - f(Z)|^p}{|X-Z|^{n+sp}} \to 0 \text{ as } \varepsilon\to 0,
\ee*
and thus $I$ satisfies \eqref{finest}.

\bigskip
\noindent{\bf Step 3.}  Estimate of $D$\\
We start by noting that, if $X\in Q_\varepsilon(T+2\varepsilon K)$ and $|Y-X|<\ve$, then 
\be*
Y\in \bigcup_{ \shortstack{$ \scriptstyle L \in \mathbb Z^n$ \\ $\scriptstyle |L|\leq 1$}}
Q_\varepsilon (T+2\varepsilon (K+L)).
\ee*
Therefore, 
\be*
D\leq
\frac{1}{\varepsilon^n}\int\limits_{Q_\varepsilon} dT \sum_{\shortstack{$\scriptstyle K, L\in\mathbb Z^n$ \\ $\scriptstyle |L|\le1$}}\ \ 
\int\limits_{Q_\varepsilon
(T+2\varepsilon K)}dX \int\limits_{Q_\varepsilon (T+2\varepsilon (K+L))}dY\
\frac{|f_T(X) - f_T(Y)|^p}{|X-Y|^{n+sp}}.
\ee*
For $L\in \mathbb Z^n$, set
\beq
\label{abac}
D_L= \frac{1}{\varepsilon^n} \int\limits_{Q_\varepsilon}dT \sum_{K\in \mathbb
Z^n}\ \  \int\limits_{Q_\varepsilon (T+2\varepsilon K)} dX
\int\limits_{Q_\varepsilon (T+2\varepsilon (K+L))} dY 
\ \frac{|f_T(X) - f_T(Y)|^p}{|X-Y|^{n+sp}},
\eeq
so that
\be*
D\leq
\sum_{\shortstack{$\scriptstyle L\in \mathbb Z^n$ \\ $\scriptstyle |L|\leq 1$}} D_L.
\ee*
We estimate separately each $D_L$.  We consider two cases:  $L=0$ and
$|L|=1$.

\bigskip
\noindent{\bf Step 3.1.}  Estimate of $D_0$\\
Since $f_T=f_{T+2\ve K}$, $\forall\, T\in\R^n$, $\forall\, K\in\Z^n$, we have
\be*
\aligned
D_0=&
\frac{1}{\varepsilon^n}\int\limits_{Q_\varepsilon} dT \sum_{ K\in\mathbb Z^n}
\int\limits_{Q_\varepsilon (T+2\varepsilon K)}dX
\int\limits_{Q_\varepsilon (T+2\varepsilon K)}dY
\ \frac{|f_T(X) - f_T(Y)|^p}{|X-Y|^{n+sp}}=\\
=&
\frac{1}{\varepsilon^n}\int\limits_{\mathbb R^n} d U 
\int\limits_{Q_\varepsilon (U)} dX\int\limits_{Q_\varepsilon(U)}dY
\ \frac{|f_U(X) - f_U(Y)|^p}{|X-Y|^{n+sp}}.
\endaligned
\ee*
In polar coordinates, we obtain
\be*
D_0=
\frac{1}{\varepsilon^n}\int\limits_{\mathbb R^n} dU \int\limits_0^\varepsilon \delta^{n-1}
d\delta \int\limits_0^\varepsilon \lambda^{n-1} d\lambda\int\limits_{|\omega|=1}
d\omega \int\limits_{|\sigma|=1} d\sigma
\ \frac{|f_U(U+\delta\omega) - f_U(U+\lambda
\sigma)|^p}{|\delta\omega-\lambda\sigma|^{n+sp}}.
\ee*
Since $\d
f_U(U+\delta\omega)= f(U+\varepsilon \omega)  $ and $\d f_U(U+\lambda\sigma)=f(U+\varepsilon \sigma)$,
we find that
\beq
\label{a10}
D_0=
\frac{1}{\varepsilon^n}\int\limits_{\mathbb R^n} dU \int\limits_{|\omega|=1} d\omega\int\limits_{|\sigma|=1} d\sigma\
 |f(U+\varepsilon\omega) - f(U+\varepsilon\sigma)|^p k(\omega, \sigma),
\eeq
where
\be*
k(\omega, \sigma)= \int\limits_0^\varepsilon \delta^{n-1}
d\delta \int\limits_0^\varepsilon \lambda^{n-1}\  d\lambda\frac{1}{|\delta\omega
-\lambda\sigma|^{n+sp}}.
\ee*

In order to complete Step 3.1, we will use
\begin{lemm} 
\l{lema2} Assume that $sp<n$. Then, for $|\omega| = |\sigma| = 1$, we have 
\beq
\label{a11}
k(\omega, \sigma)\leq \frac{C\varepsilon^{n-sp}}{|\omega-\sigma|^{n+sp-1}}.
\eeq
\end{lemm}
\begin{rema}
\l{ra3}
In the proof of Lemma \ref{lema1}, the condition $sp<n$ is used only to obtain \eqref{a11} and its more general form \eqref{gene}.
\end{rema}
\begin{proof}[Proof of Lemma \ref{lema2}] We have $k(\omega, \sigma)=\overline k(\omega, \sigma) + \tilde k(\omega, \sigma)$, where $\overline k(\omega, \sigma)
=\d \iint\limits_{\lambda \leq \delta}\ldots$ ,  $\tilde k(\omega, \sigma) =\d \iint\limits_{\delta\leq \lambda}\ldots$
We will establish \eqref{a11} for $k(\omega, \sigma)$ replaced by $\overline k(\omega, \sigma)$; a similar
inequality holds for $\tilde k(\omega, \sigma)$.  We have 
\be*
\aligned
\overline k
(\omega, \sigma)=& \int\limits_0^\varepsilon \delta^{n-1}
d\delta \int\limits_0^\delta \lambda^{n-1} d\lambda\ \frac{1}{|\delta\omega
-\lambda\sigma|^{n+sp}}=\\
=&\int\limits^1_0(t\varepsilon)^{n-1} \varepsilon dt \int\limits^1_0(t\tau\varepsilon)^{n-1}
t\varepsilon d \tau\ \frac{1}{t^{n+sp} \varepsilon^{n+sp} |\omega -
\tau\sigma|^{n+sp}}.\endaligned
\ee*
Thus,
\be*
\overline k
(\omega, \sigma)= \varepsilon^{n-sp} \int\limits_0^1 t^{n-sp-1} dt
\int\limits_0^1 \tau^{n-1}\  \frac{d\tau} {|\omega - \tau\sigma|^{n+sp}}
\leq  C \varepsilon^{n-sp} \int\limits^1_0 \tau^{n-1}\
\frac{d\tau}{|\omega - \tau\sigma|^{n+sp}}
\ee*
(here, we use the fact that $sp<n$).

We complete the proof of Lemma \ref{lema2} by establishing the following inequality.
\beq
\label{a13}
F:=\int\limits^1_0 \tau^{n-1}
\frac{d\tau}{|\omega - \tau\sigma|^{n+sp}} \leq
\frac{C}{|\omega-\sigma|^{n+sp-1}} \text{ if } 
|\omega | = |\sigma|=1.
\eeq
Indeed, 
if $\d  |\omega-\sigma| \geq \frac{1}{20}$, inequality \eqref{a13} is clear,
since in this case we have $|\omega - \tau\sigma|\geq C,$ for $0 \leq
\tau\leq 1.$

Let now $\d  |\omega - \sigma| < \frac{1}{20}$.  We split $F= F_1 + F_2$,
where $\d F_1 = \int\limits_0^{1-3|\omega-\sigma|}\ldots  $,  $\d F_2 =
\int\limits^1_{1-3|\omega-\sigma|}\ldots$\\
On the one hand, we have 
\be*
|\omega - \tau \sigma |\geq C|\omega - \sigma| \text{ if } |\omega| = |
\sigma| = 1\text{ and }\tau\in\R.
\ee*
Therefore,
\beq
\label{a14}
F_2 \leq \frac{C|\omega - \sigma|}{|\omega - \sigma|^{n+sp}} \leq
\frac{C }{|\omega - \sigma|^{n+sp -1}}.
\eeq
On the other hand, when $0\le\tau\le 1$ we have
\be*
|\omega - \tau\sigma |= |(1-\tau) \omega + \tau(\omega-\sigma)| \geq 1
- \tau-\tau|\omega-\sigma|=1-\tau(1+|\omega-\sigma|).
\ee*
Thus
\beq
\label{a15}
\aligned
F_1 & \leq \int\limits^{1-3|\omega-\sigma|}_0 \tau^{n-1}
\ \frac{d\tau}{[1-\tau(1+|\omega-\sigma|)]^{n+sp}}
 =\frac{1}{(1+ |\omega-\sigma|)^n}\int\limits^1_{2|\omega-\sigma|+ 3|\omega
-\sigma |^2} \frac{(1-t)^{n-1}}{t^{n+sp}}dt\\
& \leq \int\limits^1_{2|\omega-\sigma|+ 3|\omega
-\sigma |^2}
 \frac{dt}{t^{n+sp}}\leq \frac{C}{|\omega-\sigma |^{n+sp -1}}.
\endaligned
\eeq

We obtain \eqref{a13} when $\d  |\omega-\sigma|<\frac{1}{20}$  combining
\eqref{a14} with \eqref{a15}.

The proof of Lemma \ref{lema2} is complete.
\end{proof}
\begin{rema}
For further use, we note that the proof of Lemma \ref{lema2} shows that \eqref{a13} holds under more general assumptions on $\omega$ and $\sigma$. More specifically, if $sp<n$ then we have
\beq
\label{gene}
\int\limits^1_0 \tau^{n-1}\
\frac{d\tau}{|\omega - \tau\sigma|^{n+sp}} \leq
\frac{C}{|\omega-\sigma|^{n+sp-1}} \text{ if } 
|\sigma | =1\text{ and }1\le  |\omega|\le 3.
\eeq
\end{rema}

\bigskip
\noindent{\bf Step 3.1 continued.}  Recall that we want to establish an estimate of the form $D_0\le a(\ve)+b(\ve)$.

By \eqref{a10} and Lemma \ref{lema2}, we have 
\beq
\label{a16}
\aligned
D_0 &\leq \frac{C}{\varepsilon^{sp}} \int\limits_{\mathbb R^n}dU
\int\limits_{|\omega|=1}d\omega\int\limits_{|\sigma|=1} d\sigma\
\frac{|f(U+\varepsilon \omega) - f(U + \varepsilon \sigma)|}{|\omega -\sigma|^{n+sp-1}}^p\\
&=\frac{C}{\varepsilon^{n-1}}\int\limits_{\mathbb R^n} dU \int\limits_{|\omega|
=\varepsilon} d\omega \int\limits_{|\sigma|=\varepsilon} d\sigma\
\frac{|f(U+\omega) - f(U+\sigma)|}{|\omega - \sigma|^{n+sp -1}}^p \\
&=\frac{C}{\varepsilon^{n-1}}\int\limits_{\mathbb R^n} dU \int\limits_{|\omega|
=\varepsilon} d\omega \int\limits_{|\sigma|=\varepsilon} d\sigma\
\frac{|f(U+\omega-\sigma)- f(U)|}{|\omega - \sigma|^{n+sp-1}}^p \\
&=\frac{C}{\varepsilon^{n-1}}\int\limits_{\mathbb R^n} dU \int\limits_{|\omega|
=\varepsilon} d\omega \int\limits_{|\lambda-\omega|=\varepsilon} d\lambda\
\frac{|f(U+\lambda) - f(U)|^p}{|\lambda|^{n+sp -1}}.
\endaligned
\eeq

Here is another lemma needed in Step 3.1.

\begin{lemm} 
\l{lema5} Let $G(\lambda) \geq 0$ be any measurable function.  Then 
\beq
\label{a17}
H:=\int\limits_{|\omega|=\varepsilon} d\omega
\int\limits_{|\lambda-\omega|=\varepsilon} d\lambda G(\lambda) \leq C\left(\ve^{n-2}H_0+\ve^{n-1} \sum^n_{j=1}\ \sum^1_{q=-1}H_{j,q}\right),
\eeq
where
\be*
H_0:=\int\limits_{|\lambda|\leq 2\ve}d\lambda\, G(\lambda) ,\  H_{j,q}:= \int\limits_{\shortstack{$\scriptstyle|\widehat\lambda_j|\leq 2 \ve$ \\  $\scriptstyle \lambda_j = 2q\ve$}} d\widehat \lambda_j\, G(\lambda).
\ee*
\end{lemm}
Here, we use the standard notation $\widehat\lambda_j = (\lambda_1,
\ldots, \lambda_{j-1}, \lambda_{j+1}, \ldots, \lambda_n)$.
The vector $\widehat\omega_j$ is defined simlarly, and we let $\widehat{\lambda-\omega}_j=\widehat\lambda_j-\widehat\omega_j$.

\begin{proof}[Proof of Lemma \ref{lema5}]  We have 
\beq
\label{a18}
H = \sum^n_{j=1}\, \sum^n_{l= 1} \ \ \int\limits_{\shortstack{$\scriptstyle|\widehat\omega_j
|\leq\ve$ \\ $\scriptstyle |\omega_j| = \ve$}} d\widehat \omega_j
\ \
\int\limits_{\shortstack{$\scriptstyle|\widehat{\lambda-\omega}_l |\leq\ve$ \\ $\scriptstyle
|(\lambda-\omega)_l|=\ve$}} d\widehat \lambda_l\  G(\lambda):= 
\sum^n_{j=1}\ \sum^n_{l=1} E_{j,l}.
\eeq

We first estimate $E_{j,l}$ for $j\neq l$.  Assume e.g.           $j=1$,
$l = n$.  Then 
\beq
\label{a19}
\aligned
E_{1,n} &= \int\limits_{\shortstack{$\scriptstyle|\widehat\omega_1|\leq \ve$ \\ $\scriptstyle \omega_1 = \pm \ve$}} d\widehat
\omega_1\ \  \int\limits_{\shortstack{$\scriptstyle|\widehat {\lambda - \omega}_n|\leq \ve$ \\ $\scriptstyle \lambda_n = \omega_n
\pm \ve$}} d\widehat \lambda_n\, G(\lambda)
\\
&\leq 2\int\limits_{|\lambda| \leq 2\ve}d\lambda\,  G(\lambda) 
\int\limits_{|\omega_k|\leq \ve,\,  \forall\, k\in \llbracket 2, n-1\rrbracket} d\omega_2 \ldots d
\omega_{n-1} \leq C\ve^{n-2} \int\limits_{|\lambda|\leq 2\ve}d\lambda\,  G (\lambda) .
\endaligned
\eeq

Let now $j=l$.  Assume e.g. $j=l=n$. Since $\omega_n = \pm
\ve$ and $\lambda_n = \omega_n\pm \ve$, we have $\lambda_n\in \{
-2\ve, 0, 2\ve\}$.  Therefore,
\beq
\label{a20}
\aligned
E_{n,n}&\le 2\int\limits_{|\widehat\omega_n|\leq \varepsilon}d\widehat\omega_n\ \ 
\int\limits_{\shortstack{$\scriptstyle|\widehat{\lambda-\omega}_n|\leq \ve$ \\ $\scriptstyle \lambda_n \in\{0, \pm
2\ve\}$}}d\widehat\lambda_n\, G(\lambda)\\
&\leq C\int\limits_{\shortstack{$\scriptstyle|\widehat\lambda_n|\leq 2 \ve$ \\ $\scriptstyle \lambda_n\in\{0, \pm 2
\ve\}$}}d\widehat\lambda_n \int\limits_{|\widehat\omega_n|\leq \ve}d\widehat\omega_n\,  G(\lambda)= C\ve^{n-1} \sum^1_{q=-1}\  \int\limits_{\shortstack{$\scriptstyle|\widehat\lambda_n|\leq 2 \ve$ \\ $\scriptstyle
\lambda_n=2q\ve$}} d\widehat\lambda_n\, G(\lambda).\endaligned
\eeq

Lemma \ref{lema5} follows from \eqref{a18}-\eqref{a20}.
\end{proof}

\bigskip
\noindent{\bf Step 3.1 continued.} Recall that we look for an estimate of the form $D_0\le a(\ve)+b(\ve)$.

  By \eqref{a16} and Lemma \ref{lema5} applied with 
\be*
G(\lambda) =G(U,\lambda)= \frac{|f(U+\lambda) - f(U)|^p}{|\lambda|^{n+sp-1}},
\ee*
we find that
\beq
\label{a21}
\aligned
D_0 &\leq C\frac{1}{\ve }\int\limits_{\mathbb R^n} dU\int\limits_{ |\lambda|\leq2\ve}d\lambda\  G(U,\lambda)+ C\sum^n_{j=1}
\sum^1_{q=-1}\ \int\limits_{\mathbb R^n} dU  \int\limits_{\shortstack{$\scriptstyle|\widehat\lambda_j|\leq 2 \ve$ \\ $\scriptstyle\lambda_j =
2q\ve$}}d\widehat\lambda_j\, G(U,\lambda)\\
&:= C(P_0 + \sum^n_{j=1} \left(P_{j,0} + P_{j,2\ve} + P_{j,-2\ve})\right).
\endaligned
\eeq

In view of the above, we will establish estimates of the form $P\le a(\ve)+b(\ve)$, where $P$ is one of the $P_0$, $P_{j,0}$, $P_{j,\pm 2\ve}$.

\medskip
\noindent{\bf Estimate of $P_0$.}  We have 
\beq
\label{a22}
\aligned
\int\limits^1_0  d\ve\, \frac{P_0}{\ve} &= \int\limits_{\mathbb R^n}
dU\int\limits^1_0\frac{d\ve}{\ve^2}\int\limits_{|\lambda|\leq2\ve} d\lambda \
\frac{|f(U+\lambda) -f(U)|^p}{|\lambda|^{n+sp-1}}\\
&=\int\limits_{\mathbb R^n} dU \int\limits_{|\lambda|\leq 2} d\lambda\
\frac{|f(U+\lambda)
-f(U)|^p}{|\lambda|^{n+sp-1}}
\int\limits^1_{|\lambda|/2}\frac{d\ve}{\ve^2}\\
&\leq C\int\limits_{\mathbb R^n} dU\int\limits_{|\lambda|\leq 2} d\lambda\
\frac{|f(U+\lambda)-f(U)|^p}{|\lambda|^{n+sp}}\\
&=C\iint\limits_{|X-Y|<2} dXdY\
\frac{|f(X)-f(Y)|^p}{|X-Y|^{n+sp}}<\infty.
\endaligned
\eeq

\noindent{\bf Estimate of $P_{j,2\ve}$.} (A similar estimate holds for 
$P_{j,-2\ve}$.) Assume e.g. $j=n$.  Then
\be*
P_{n, 2\ve} = \int\limits_{\mathbb R^n} dU\int\limits_{|\widehat\lambda_n|\leq
2\ve}d\widehat\lambda_n \
\frac{|f(U+(\widehat \lambda_n, 
2\ve))-f(U)|^p}{(2\ve)^{n+sp-1}},
\ee*
so that
\beq
\label{a23}
\aligned \int\limits_0^1d\ve\, \frac{P_{n,2\ve}}{\ve} 
&= C \int\limits^1_0 \frac{d\ve}{\ve^{n+sp}} \int\limits_{\mathbb R^n} dU
\int\limits_{|\widehat\lambda_n|\leq 2\ve}d\widehat\lambda_n\  |f(U+(\widehat \lambda_n, 2\ve))-f(U)|^p\\
&=C \int\limits_{\mathbb R^n} dU\int\limits_{|\lambda|=\lambda_n\leq 2}d\lambda\  \frac{|f(U+\lambda)-f(U)|^p}{|\lambda|^{n+sp}}\\
&\leq C \int\limits_{\mathbb R^n} dU\int\limits_{|\lambda|\leq 2}d\lambda\  \frac{|f(U+\lambda)-f(U)|^p}{|\lambda|^{n+sp}}< \infty.
\endaligned
\eeq

\noindent{\bf Estimate of $P_{j,0}$.}  Assume $j=n$.  Then
\beq
\label{a24}
P_{n,0}= \int\limits_{\mathbb R^n} dU \int\limits_{|\widehat\lambda_n|\leq 2\ve}
d\widehat\lambda_n\
\frac{|f(U+(\widehat\lambda_n,
0))-f(U)|^p}{|\widehat\lambda_n|^{n+sp-1}}.
\eeq
In order to estimate $P_{n,0}$, we rely on a variant of a well-known lemma due to Besov \cite[proof of Lemma 7.44, p. 208]{a}, more precisely
\begin{lemm}
\l{lema6}  We have, for $1\leq l \leq n$,
\beq
\label{a25}
\aligned
R_l:&=\int\limits_{\mathbb R^n} dU \int\limits_{|\lambda_k| \leq \ve,\, \forall\, k \leq l}
d\lambda_1\ldots d\lambda_l \
\frac{\left|f\left(U+\d\sum^l_{k=1} \lambda_k e_k\right)-f(U)\right|^p}{|(\lambda_1,\ldots,\lambda_l)|^{l+sp}}\\
&\leq C\int\limits_{\mathbb R^n} dU \int\limits_{|\lambda|\leq 2\ve} d\lambda\ 
\frac{|f(U+\lambda)-f(U)|^p}{|\lambda|^{n+sp}}.\endaligned
\eeq
\end{lemm}
The standard form of Lemma \ref{lema6} corresponds corresponds to $l=1$. The proof we present below for arbitrary $l$ is essentially the same as for $l=1$.
\begin{proof}[Proof of Lemma \ref{lema6}] For
$\lambda' = (\lambda_1,\ldots, \lambda_l) \in \mathbb R^l$ and $ U\in
\mathbb R^n$, let $\d W=W_{\lambda',U}:=U +\sum^l_{k=1} \lambda_k e_k$. Let $Q=Q_{\lambda',U}$ be the cube centered at the midpoint
of the segment $[U, W]$ and of sidelength
$\d\frac{1}{4}|\lambda'|$.  For any $V\in Q$, we have
\beq
\label{a26}
\left|f\left(U+\sum^l_{k=1} \lambda_k e_k\right) - f(U)\right|^p \leq C (|f(V) - f(U)|^p +
|f(V) - f(W) |^p).
\eeq
By taking the average integral of \eqref{a26} in $V$ over $Q$ , we find that
\be*
\aligned
R_l&\leq C\int\limits_{\mathbb R^n} dU \int\limits_{|\lambda' |\leq\ve}d\lambda' \int\limits_{Q}
\frac{dV}{|\lambda'|^{l+sp+n}}\ \{|f(V)- f(U)|^p+ |f(V)-f(W)|^p\}\\
&=2C\int\limits_{\mathbb R^n} dU \int\limits_{|\lambda' |\leq\ve}d\lambda' \int\limits_Q dV\ \frac{|f(V) -
f(U)|^p}{|\lambda'|^{l+sp+n}}.\endaligned
\ee*
Noting that $|V-U|\le\d  |\lambda'|$, we obtain
\be*
\aligned
R_l&\leq C\int\limits_{\mathbb R^n} dU \int\limits_{|V-U|\leq 2
\ve}dV\ |f(V)-f(U)|^p\int\limits_{1/2 |V-U|\leq |\lambda'|\leq
\ve}\frac{d\lambda'}{|\lambda'|^{l+sp+n}}\\
&\leq C\int\limits_{\mathbb R^n} dU \int\limits_{|V-U|\leq 2 \ve} dV\  \frac{|f(V) -
f(U)|^n}{|V-U|^{n+sp}}.\endaligned
\ee*

The proof of Lemma \ref{lema6} is complete.
\end{proof}

\bigskip
\noindent{\bf Step 3.1. completed.}
Lemma \ref{lema6} and \eqref{a24} imply that
\be*
P_{j,0} \longrightarrow 0 \text{ as } \ve \longrightarrow 0,\ \forall\, j\in\llbracket 1,
n\rrbracket.
\ee*

Step 3.1 is now complete.

\bigskip
\noindent{\bf Step 3.2.}  Estimate of $D_L$ when $L\in \mathbb Z^n$ and
$|L|=1$. Similarly to Step 3.1, we will establish an estimate of the form $D_L\le a(\ve)+b(\ve)$.\\
Recall that 
\be*
D_L = \frac{1}{\ve^n} \int\limits_{Q_\ve} dT \sum_{K\in \mathbb Z^n}\
\int\limits_{Q_\ve(T+2\ve K)} dX\ \int\limits_{Q_\ve(T+2\ve(K+L))}dY\ \frac{|f_T(X) -
f_T(Y)|^p}{|X-Y|^{n+sp}}.
\ee*
If we set $V=V_U=U+2\ve L$, $X_U=U+\d \frac{X-U}{|X-U|}$ and $Y_V=V+\d \frac{Y-V}{|Y-V|}$, then we have 
\be*
D_L= \frac{1}{\ve^n}\int\limits_{\mathbb R^n} dU
\int\limits_{Q_\ve(U)}dX\int\limits_{Q_\ve(V)} dY\
\frac{|f(X_U)-f(Y_V)|^p}{|X-Y|^{n+sp}}.
\ee*
In polar coordinates, we obtain
\be*
\aligned
D_L &= \frac{1}{\ve^n} \int\limits_{\mathbb R^n}d U \int\limits^\ve_0 \delta^{n-1}
d\delta\int\limits^\ve_0\lambda^{n-1} d\lambda\int\limits_{|\omega|=1} d\omega
\int\limits_{|\sigma|=1}d\sigma\ 
\frac{|f(U+\ve\omega)-f(V+\ve\sigma)|^p}{|\delta\omega -\lambda\sigma-2\ve L|^{n+sp}}\\
&= \frac{1}{\ve^n} \int\limits_{\mathbb R^n} dU\int\limits_{|\omega|=1}d\omega
\int\limits_{|\sigma|=1}d\sigma \  |f(U+\ve\omega) - f(V+
\ve\sigma)|^p\, k(\omega, \sigma),\endaligned
\ee*
where
\be*
k(\omega, \sigma) = \int\limits^\ve_0\delta^{n-1} d\delta\int\limits^\ve_0
\lambda^{n-1}d\lambda\ \frac{1}{|\delta\omega -\lambda\sigma-2\ve L|^{n+sp}}.
\ee*

To estimate $D_L$, 
we rely on a variant of Lemma \ref{lema2} (which formally corresponds to  $L=0$ in \eqref{a27}).
\begin{lemm}
\l{lema7}
Assume that $sp<n$. For $|\omega| = |\sigma| =1$ and $L\in\mathbb Z^n$
with $|L|=1$ we have
\beq
\label{a27}
k(\omega, \sigma) \leq \frac{C\ve^{n-sp}}{|\omega-\sigma-
2L|^{n+sp-1}}.
\eeq
\end{lemm}
\begin{proof}[Proof of Lemma \ref{lema7}] We have
\beq
\label{ab1}
k(\omega,\sigma) = \ve^{n-sp}\int\limits^1_0 t^{n-1}dt \int\limits^1_0 \tau^{n-1}
d\tau\  \frac{1}{|t\omega-\tau \sigma-2L|^{n+sp}}.
\eeq
We claim that
\beq
\label{a28}
|t\omega-\tau \sigma-2L|\geq C|t\omega- \sigma-2L|.\eeq
Indeed, when $0\leq \tau \leq 1/2$, inequality \eqref{a28} is clear, since
in this case we have 
$\d |t\omega -  \sigma-2L|\leq 4\text{ and }|t\omega-\tau\sigma-2L|\geq 2- t-\tau\geq 1/2$.\\
Assume now $\tau \geq 1/2$.  We consider the map 
\be*
\varphi:\overline Q_1 \cup (\overline Q_1(2L) \smallsetminus Q_{1/2} (2L))
\longrightarrow \mathbb R^n,
\ee*
defined by
\be*
\varphi(X)=\begin{cases} X,  & \text{ if } X\in \overline Q_1\\
\d 2L + \frac{X-2L}{|X-2L|}, & \text{ if } X\in \overline Q_1(2L)\smallsetminus
Q_{1/2}(2L)\end{cases}.
\ee*

It is easy to check that $\va$ is well-defined, in the sense that
\be*
X=2L + \frac{X-2L}{|X-2L|}\ \text{for every }X\in \overline Q_1\cap \overline Q_1(2L).
\ee*

Note that, in $\d \overline Q_1(2L)\smallsetminus
Q_{1/2}(2L)$, $\varphi$ is  the radial  projection centered at $2L$ on $\p Q_1(2L)$.
Clearly, $\varphi$ is Lipschitz.  Inequality \eqref{a28} for $1/2 \leq \tau
\leq 1$ is now obvious, since it reads
\be*
|\varphi(t\omega) - \varphi(\tau\sigma+2L)| \leq \frac{1}{C} |t\omega
- (\tau\sigma+2L)|.
\ee*
Combining \eqref{ab1} and \eqref{a28}, we obtain 
\beq
\label{a31}
k(\omega, \sigma) \leq C\ve^{n-sp} \int\limits^1_0 t^{n-1}\
\frac{dt}{|t\omega- \sigma-2L|^{n+sp}}.
\eeq
Applying \eqref{gene} with $\omega$ replaced by
$\sigma+2L$ and $\sigma$ replaced by $\omega$ (here, we use $sp<n$), we obtain \eqref{a27} from
\eqref{gene} and \eqref{a31}.

The proof of Lemma \ref{lema7} is complete.
\end{proof}

\bigskip
\noindent{\bf Step 3.2 continued.} We continue our way to an estimate of the form $D_L\le a(\ve)+b(\ve)$.

By Lemma \ref{lema7} we obtain
\be*
\aligned
D_L &\leq\frac{C}{\ve^{sp}}\int\limits_{\mathbb R^n} dU \int\limits_{|\omega|=1} d\omega
\int\limits_{|\sigma|=1} d\sigma \ \frac{|f(U+\ve\omega) - f(U + \ve
\sigma+2\ve L)|^p}{|\omega- \sigma-2L|^{n+sp-1}}\\
& =
\frac{C}{\ve^{n-1}}\int\limits_{\mathbb R^n} dU \int\limits_{|\omega|=\ve} d\omega
\int\limits_{|\sigma|=\ve} d\sigma\  \frac{|f(U)- f(U+2\ve L +
\sigma-\omega)|^p}
{|2 \ve L+\sigma-\omega|^{n+sp-1}}.
\endaligned
\ee*
Thus
\beq
\label{a311}
D_L\leq\frac{C}{\ve^{n-1}}\int\limits_{\mathbb R^n} dU \int\limits_{|\omega|=\ve} d\omega
\int\limits_{|\lambda+ \omega-2\ve L |=\ve} d\lambda\ \frac{|f(U) - f(U+\lambda)|^p}{|\lambda|^{n+sp-1}}.
\eeq
We combine \eqref{a311} with the following straightforward variant of Lemma
\ref{lema5}, whose proof is left to the reader:

\begin{lemm} 
\l{lema8} 
Let $G(\lambda) \geq 0$ be any measurable function.  Then for $L\in \mathbb
Z^n$ with $|L|=1$ we have 
\beq
\label{a32}
H:=\int\limits_{|\omega|=\ve} d\omega
\int\limits_{|\lambda+\omega-2\ve L|=\ve} d\lambda
\, G(\lambda)\leq C\left(\ve^{n-2}H_0+\ve^{n-1} \sum^n_{j=1}\ \sum_{q=-2}^2H_{j,q}\right),
\eeq
where 
\be*
H_0:=\int\limits_{|\lambda|\leq 4\ve}d\lambda\,  G(\lambda) ,\  H_{j,q}:=\int\limits_{\shortstack{$\scriptstyle|\widehat\lambda_j|\leq 4
\ve$ \\ $\scriptstyle \lambda_j = 2q\ve$}}d\widehat\lambda_j\, G(\lambda).
\ee*
\end{lemm}

\bigskip
\noindent{\bf Step 3.2 completed.}   By \eqref{a311} and Lemma \ref{lema8}, we obtain
\beq
\label{a33}
\aligned
D_L\leq&  \frac C\ve \int\limits_{\mathbb R^n} dU \int\limits_{|\lambda|\leq 4
\ve} \frac{|f(U+\lambda) - f(U)|^p}{|\lambda|^{n+sp-1}}d\lambda \\
& +C \sum^n_{j=1}\  \sum^2_{q=-2}\ \int\limits_{\shortstack{$\scriptstyle|\widehat\lambda_j|\leq 4\ve$ \\ $\scriptstyle \lambda_j =
2q\ve$}} d\widehat\lambda_j \frac{|f(U+\lambda) -
f(U)|^p}{|\lambda|^{n+sp-1}}.
\endaligned
\eeq
Estimate \eqref{a33} is similar to \eqref{a21} and we handle it in the same
way.

The proof of Lemma \ref{lema1} is complete.
\end{proof}

We end this section by proving that, in $W^{1,p}$, approximation by piecewise homogeneous maps fails. The special case we treat below ($p=1$, $n=2$)  is easily generalized to any dimension or $p$.
\begin{lemm}
\l{lema9}
Let $u(x_1,x_2)=x_1$. Then there is no sequence $\{u_k\}$ of piecewise $1$-homogeneous maps associated to meshes contained in $(-1,2)^2$ such that $u_k\to u$ in $W^{1,1}((0,1)^2)$.
\end{lemm}
Note that the conclusion of the lemma is that not only the estimates given by  Lemma \ref{lema1} do not hold  when $s=1$, but also that any  possible approximation method  relying on piecewise homogeneous maps  fails.
\begin{proof}
We argue by contradiction and assume that there exists a sequence $\{ u_k\}$ of piecewise $1$-homogeneous maps associated to meshes contained in $(-1,2)^2$ such that  $u_k\to u$ in $W^{1,1}$. Let $u_k$ be piecewise $1$-homogeneous on the mesh ${\mathcal C}^k$, with $(0,1)^2\subset{\mathcal C}^k\subset (-1,2)^2$. Let $2l_k$  (with $l_k\le 2$) be the size of the squares in ${\mathcal C}_k$ and set 
\be*
{\mathcal D}_k=\{ Q\in {\mathcal C}_k\ ;\ Q\subset (0,1)^2\}.
\ee*
Clearly, there is some $l^0>0$ such that
\beq
\label{impl}
\text{if } l_k<l^0, \text{ then }\bigg|\bigcup_{Q\in {\mathcal D}_k}Q\bigg|\ge 1/2.
\eeq
We distinguish two possibilities:

\medskip
\noindent
{\bf Case 1.} $l_k<l^0$\\
Let $S$ be the center of $Q=Q_{l_k}(S)\in {\mathcal D}_k$. For $X\in Q\setminus\{S\}$, set $V=V(X)=(X-S)/\|X-S\|$; here, $\|\ \|$ stands for the Euclidean norm. Since $u_k$ is constant along the segment  $[S, X]$, we have $\d\frac{\p u_k}{\p V}(X)=0$ a.e. in $X$. Therefore, 
\beq
\label{abab}
|(\nabla u_k-\nabla u)(X)|\ge \left|\frac{\p u_k}{\p V}(X)-\frac{\p u}{\p V}(X)\right|=|V_1|.
\eeq
We find that
\be*
\int\limits_Q|\nabla u_k-\nabla u|\ge \int\limits_Q|V_1|\ge Cl_k^2=C|Q|;
\ee*
the last inequality follows by scaling. Using \eqref{impl}, we find that 
\beq
\label{con1}
\|u_k-u\|_{W^{1,1}((0,1)^2)}\ge C.
\eeq
Thus, for large $k$, we are in 

\medskip
\noindent
{\bf Case 2.} $l_k\ge l^0$\\
Possibly after passing to a subsequence, we may assume that:
\begin{enumerate}[a)]
\item
$l_k\to l$ for some $l\ge l^0$.
\item
All the meshes ${\mathcal C}_k$ contain the same number of squares, say $m$.
\item
The centers of the squares $Q_{1,k},\ldots, Q_{m,k}$ in ${\mathcal C}_k$, say $S_{1,k},\ldots, S_{m,k}$, converge respectively to $S_1,\ldots, S_m$.
\end{enumerate}
Set $Q_j=Q_l(S_j)$. By \eqref{abab} and dominated convergence, we have
\be*
\lim_k\int\limits_{(0,1)^2}|\nabla (u_k-u)|\ge \lim_k\sum_j\int\limits_{Q_{j,k}\cap (0,1)^2}\frac{|(X-S_{j,k})_1|}{|X-S_{j,k}|}=\sum_j\int\limits_{Q_{j}\cap (0,1)^2}\frac{|(X-S_{j})_1|}{|X-S_{j}|}>0.
\ee*
This contradiction completes the proof of Lemma \ref{lema9}.
\end{proof}

\def\noin{\noindent}
\def\rhpd{\rightharpoondown}

\section{A more general approximation method}
\l{appb}
The approximation method described in Section \ref{appa} goes as follows: fix
some $\varepsilon > 0$ and $T\in \mathbb R^n$.  Consider the mesh $\mathscr
C_n$ of $n$-dimensional cubes of sidelength $2\ve$ having $T$ as one of
the centers.  Let $\mathscr C_{n-1}$ be the $(n-1)$-dimensional skeleton
associated to this mesh, i.e., $\mathscr C_{n-1}$ is the union of the
boundaries of the cubes in $\mathscr C_n$. 
  Let $H_n$ be the mapping that
associates to every $g: \mathscr C_{n-1}\to \mathbb R^m$ its homogeneous
extension (on each cube of $\mathscr C_n$) to $\mathbb R^n$.  Lemma \ref{lema1} asserts that, if $0< s < 1,
sp < n$ and $f\in W^{s, p}(\mathbb R^n; \mathbb R^m)$, then 
$H_n(f|_{\mathscr C_{n-1}}) \to f$ in $W^{s, p}(\mathbb R^n)$ for some suitable
choice of $\ve_k \to 0$ and $T_k \in \mathbb R^n$.
\medskip

We will describe below a more general situation.  We start by defining
the lower dimensional skeletons associated to $\mathscr C_n$. This is done
by backward induction: $\mathscr C_{n-2}=\mathscr C_{n-2, T}$ is the union of the
$(n-2)$-dimensional boundaries of the cubes in $\mathscr C_{n-1}=\mathscr C_{n-1, T}$,   
and so on.  For $g: \mathscr C_j\to \mathbb R^m$, let $H_{j+1}(g)$ be its
homogeneous extension to $\mathscr C_{j+1}$.

Let $0\leq j< n$.  For $\ve > 0$ and $T\in \mathbb R^n$, we associate to
each map $f: \mathbb R^n\to \mathbb R^m$ a map $f_T= f_{T,\ve}: \mathbb R^n \to
\mathbb R^m$ through the formula
\beq
\label{fgh1}
f_T = H_n(H_{n-1} (\cdots (H_{j+1} (g))\cdots));\ \text{here, we set }g=f|_{\mathscr C_j}.
\eeq
We start by deriving a useful formula for $f_T$. For this purpose and for further use, we start by introducing some (slightly abusive) notation that we discuss in some length in the next four paragraphs.

In order to keep  notation easier to follow, we will sometimes denote a point in $Q_\ve$ by $X^n$ rather than $X$. We denote by $X^{n-1}$ the radial   projection (centered at $0$) of $X^n\in Q_\ve$ onto the $(n-1)$-skeleton of $\p Q_\ve\cap {\cal C}_{n-1, 0}$ of $Q_\ve$; this projection is defined except when $X^n=0$.  
The abuse of notation is that $X^{n-1}$ also denotes a \enquote{generic} point of $\p Q_\ve\cap {\cal C}_{n-1, 0}$. 
We next let $X^{n-2}$ denote the radial projection of $X^{n-1}$ onto the $(n-2)$-skeleton of $\p Q_\ve\cap {\cal C}_{n-2, 0}$ of $Q_\ve$. The point $X^{n-2}$ is obtained as follows: if $X^{n-1}\in  \p Q_\ve\cap {\cal C}_{n-1, 0}$ belongs to an  $(n-2)$-dimensional face $F$ of $\p Q_\ve\cap {\cal C}_{n-1, 0}$ and is not the center $C$ of $F$, then the radial projection (centered at $C$) of $X^{n-2}$ on $F$ is well-defined, and yields $X^{n-2}$. By backward induction, we define $X^j$, $j=\llbracket 0, n-1\rrbracket$, as the radial projection of $X^{j+1}$ onto $\p Q_\ve\cap {\cal C}_{j, 0}$; this is defined for all but a finite number of $X^{j+1}$s. Again, with an abuse of notation $X^j$ is the \enquote{generic} point of $\p Q_\ve\cap {\cal C}_{j, 0}$.

When $X^j$ is obtained starting from $X^n$, we will denote $X^j$ as the radial projection of $X^n$ onto $\p Q_\ve\cap {\cal C}_{j, 0}$. This projection is defined except on a set of finite ${\cal H}^{n-j-1}$ measure.

More generally, let $j<\ell\le n$. We identify $X^\ell$ with a \enquote{generic} point of   $\p Q_\ve\cap{\cal C}_{\ell,0}$. Then  $X^j$ is, except on a set of finite ${\cal H}^{\ell-j-1}$ measure, the projection of $X^\ell$  onto $\p Q_\ve\cap {\cal C}_{j, 0}$.

Let $K\in \mathbb Z^n$ and set $U = T+2\ve K$. Then we we define the radial projection of $U+ X^n$ onto $\mathscr C_{j}$ as $U + X^{j}$.  This is consistent with the projection we defined when $j=n-1$ and makes sense for ${\cal H}^n$-a.e. $X^n\in Q_{\ve}$. Moreover, if $j<\ell\le n$, then for ${\cal H}^\ell$-a.e. $X^\ell\in \p Q_\ve\cap{\cal C}_{\ell,0}$, the 
projection of $U+ X^\ell$ onto $\mathscr C_{j}$ is $U + X^{j}$.
 
 With the above notation, formula \eqref{fgh1} is equivalent to
\beq
\label{fgh2}
f_T (T+2\ve K+X^n)= f(T+2\ve K+X^j),\ \forall\, K\in\Z^n,\text{ for }{\cal H}^n-\text{a.e. }X^n\in Q_\ve. 
\eeq

Our next task is to derive a convenient formula for $X^j$. 
 We consider the following
a.e. partition of $Q_\ve$:
\beq
\label{part}
Q_\ve = \bigcup_{q\in\{-1,1\}^{n-j}}\ \ \bigcup_{\sigma \in S_{n-j, n}}
Q_{\ve, q,\sigma}.
\eeq
Here, 
$\d
S_{n-j,n} = \{ \sigma:\{ 1, \ldots, n-j\} \to \{ 1, \ldots, n\};\,  \sigma
\text{ into } \}
$.\\
A point $X^n\in Q_\ve$ belongs to  $Q_{\ve, q,\sigma}$ provided:
\begin{enumerate}[a)]
\item The $\d\sigma(i)^{th}$ coordinate of $X^n$, denoted $(X^n)_{\sigma (i)}$, has the sign of $q_i$, for $i\in\llbracket 1,n-j\rrbracket$.
\item
In absolute value,  the largest coordinate of $X^n$ is $(X^n)_{\sigma (1)}$, the second largest is $(X^n)_{\sigma (2)}$,$\ldots$, the $\d (n-j)^{th}$ largest is $(X^n)_{\sigma (n-j)}$.
\end{enumerate}
Analytically, this means that $Q_{\ve, q,\sigma}$ is defined by the inequalities
\be*
q_{1}(X^n)_{\sigma(1)} \geq\cdots \geq
q_{n-j}(X^n)_{\sigma(n-j)}\geq |(X^n)_k | , \forall\, k \neq \sigma(1), \ldots, \sigma (n-j).
\ee*
%Before going further, recall that $T$ is fixed, for the moment, and that the skeletons ${\mathscr C}_j$ are obtained from the mesh of size $2\ve$ that has $T$ as one of its edges.\\
Let, for $K\in\mathbb Z^n$, $U=T+2\ve K$. \be*
(X^{n-1})_{\sigma(1)} = \ve q_{1},\  (X^{n-1})_l =\ve
\frac{(X^n)_l}{|(X^n)_{\sigma(1)}|},\, \forall\, l\neq \sigma(1).
\ee*
Similary, one may check that the projection of $U+X^{n}$ onto $\mathscr C_{n-2}$ is
$U+X^{n-2}$, with 
\be*
(X^{n-2})_{\sigma(1)} = \ve q_{1},\, (X^{n-2})_{\sigma(2)}=\ve
q_{2}, \ 
(X^{n-2})_l = \ve \frac{(X^n)_l}{|(X^n)_{\sigma(2)}|},\, \forall\,  l\neq \sigma(1),
\sigma (2),
\ee*
and so on.  In particular, \eqref{fgh2} reads 
$\d
f_T(U+X^n) = f(U+X^j)$ for ${\cal H}^n$-a.e. $\d X^n\in Q_{\ve,q,\sigma}$, where
\be*
(X^j)_{\sigma(k)} = \ve q_{k},\,  \forall\, k \in\llbracket 1, n-j\rrbracket,\  (X^j)_l =
\ve\frac{(X^n)_l}{|(X^n)_{\sigma(n-j)}|},\,  \forall\, l\neq \sigma(1),
\ldots, \sigma (n-j).
\ee*

This section is essentially devoted to the proof of the following generalization
of Lemma \ref{lema1}.
\begin{lemm}
\l{lemb1}
 Let $0\leq j < n, 0 < s<1, sp < j+1$ and let
$f\in W^{s,p}(\mathbb R^n\,; \mathbb R^m)$.  Then there are sequences $\ve_k
\to 0$ and $\{T_k\}\subset \mathbb R^n$ such that $f_{T_k,\ve_k}\to f$ in
$W^{s,p}(\mathbb R^n)$.
\end{lemm}
Note that Lemma \ref{lema1} corresponds to $j=n-1$.
\begin{proof}[Proof of Lemma \ref{lemb1}]  The proof of Lemma \ref{lemb1} is similar
to that of Lemma \ref{lema1}, some computations being essentially identical.
An additional difficulty appears in the estimate of $D$ (for the
definition of $D$, see \eqref{defd}).  In order to facilitate the
presentation we use the same notation as in Section \ref{appa}, and follow the steps in Section \ref{appa}. Let us recall that our goal is to obtain estimates of the form
\be*
I\le a(\ve)+b(\ve),\ J\le a(\ve)+b(\ve),\ D_L\le a(\ve)+b(\ve),
\ee*
with $I$, $J$, $D_L$ analogous to the quantities introduced in the previous section, and $a(\ve)$ and $b(\ve)$ satisfying \eqref{e05092}.

\medskip
\noindent{\bf Step 1.}  We have
\beq
\label{b1}
A= \frac{1}{\ve^n} \int\limits_{Q_\ve}\| f-f_{T,\ve} \|^p_{L^p} d T\to 0
\text{ as } \ve\to 0.
\eeq
Indeed, as in the proof of \eqref{a3} we find that 
\be*
A= \frac{1}{\ve^n}\int\limits_{Q_\ve}\|f(\cdot) - f(\cdot + X^n  -
X^j)\|^p_{L^p} dX^n. 
\ee*
Since $X^n \in Q_\ve \Rightarrow X^n-X^j\in Q_\ve$, the argument used in
the proof of \eqref{a3} yields \eqref{b1}.

\medskip
\noindent{\bf Step 2.}   Estimate of $I$\\
In our situation, $I$ is given by
\beq
\label{b2}
I= \frac{C}{\varepsilon^{n+sp}} \int\limits_{\mathbb R^n}dU \int\limits_{Q_\ve}  dX^n
|f(U) - f(U+X^n-X^j)|^p.
\eeq
It is convenient to split the integral $\d\int\limits_{Q_\ve} \ldots $ in \eqref{b2}  as
\be* \int\limits_{Q_\ve}\ldots=
\sum_{ q\in\{ -1, 1\}^{n-j}}\  \sum_{\sigma \in S_{n-j, n}
}\ \ \int\limits_{Q_{\ve,q,\sigma}}\ldots
\ee* 
We estimate, e.g., the
integral
$\overline I$ corresponding to $q_i=1$, $\sigma(i)= i,\, \forall\, i\in \llbracket 1,n-j\rrbracket$, the
other terms being similar.  
%\beq
%\overline I = 
%\frac{C}{\varepsilon^{n+sp}} \int\limits_{\mathbb R^n}dU \int\limits_{0\leq |X_n|,
%\ldots,|X_{N-J+1}|\leq X_{N-J}\leq \ldots \leq X_1\leq \ve}  dX
%|f(U) - f(U+X-(\ve,\ldots, \ve, \ve\frac{X_{N-J+1}}{X_{N-J}} , \ldots, \ve\frac{X_n}{X_{N-J}})|^p.
%\eeq
If we set $Z:=X^j-X^n$, then
\be*
\overline I = \frac{C}{\ve^{n+sp}} \int\limits_{\mathbb R^n} dU
\int\limits_{\shortstack{$\scriptstyle  | (X^n)_l|\leq  (X^n)_{n-j},\, \forall\, l>n-j$ \\ $\scriptstyle (X^n)_{n-j}\leq (X^n)_{n-j-1}\leq\cdots
\leq  (X^n)_1 \leq \ve$}} dX^n\, |f(U) - f(U-Z)|^p
\ee*
and
\be*
Z_l=\begin{cases} \ve - (X^n)_{l}, & \text{if }l\le n-j\\
\d\left(\frac{\ve}{ (X^n)_{n-j}}-1\right)  (X^n)_l, & \text{if }l>n-j
\end{cases}.
\ee*

The following properties are straightforward:
\beq
\label{b3}
0\leq Z_1 \leq \cdots \leq Z_{n-j} \leq \ve \ \text{ and }\ |Z_{n-j+1}|, \ldots,
| Z_n|\leq Z_{n-j},
\eeq
\be*
(X^n)_1= \ve - Z_1, \ldots, (X^n)_{n-j} = \ve - Z_{n-j}, 
\ee*
\be*
(X^n)_{n-j+1}   =
\frac{\ve -Z_{n-j}}{Z_{n-j}} Z_{n-j + 1},\ldots, (X^n)_n  =
\frac{\ve -Z_{n-j}}{Z_{n-j}} Z_n,
\ee*
\be*
\text{the Jacobian of the mapping }Z\mapsto X\text{ is given by }\left|\frac{dX^n}{dZ}\right| = \left( \frac{\ve -
Z_{n-j}}{Z_{n-j}}\right)^j.
\ee*
Thus
\beq
\label{b4}
\overline I \leq \frac{C}{\ve^{n+sp}} \int\limits_{\mathbb R^n} dU \int\limits^\ve_0 d
Z_{n-j}\int\limits_{|\widehat Z_{n-j}|\leq Z_{n-j}} d\widehat Z_{n-j}
\left(\frac{\ve-Z_{n-j}}{Z_{n-j}}\right)^j |f(U) - f(U-Z)|^p.
\eeq
Since for any $Z$ satisfying \eqref{b3} we have
\be*
\frac{1}{\ve^{n+sp}}\left(\frac{\ve- Z_{n-j}}{Z_{n-j}}\right)^j \leq
\frac{1}{\left(Z_{n-j}\right)^{n+sp}} = \frac{1}{|Z|^{n+sp}},
\ee*
\eqref{b4} implies  that
\be*
\overline I\leq C\int\limits_{\mathbb R^n} dU \int\limits_{|Z|\leq\ve}
dZ\, \frac{|f(U)-f(U-Z)|^p}{|Z|^{n+sp}} = C \iint\limits_{|U-W|<\ve}
dUdW\, \frac{|f(U)-f(W)|^p}{|U-W|^{n+sp}}\to 0 \text{ as } \ve \to 0.
\ee*
\medskip
\noindent{\bf Step 3.}  Estimate of $D$

With $D_L$ as in \eqref{abac},   we have
\be*
D \leq \sum_{\shortstack{$\scriptstyle L\in \mathbb Z^n$ \\ $\scriptstyle|L|\leq 1$}} D_L.
\ee*

\medskip
\noindent{\bf Step 3.1.}  Estimate of $D_0$\\
Recall that
\be*
\aligned
D_0=&
\frac{1}{\varepsilon^n}\int\limits_{\mathbb R^n} dU
\int\limits_{Q_\varepsilon} dX^n \int\limits_{Q_\varepsilon} dY^n\,
\frac{|f_U(U +X^n) - f_U(U+Y^n)|^p}{|X^n-Y^n|^{n+sp}}\\
=&
\frac{1}{\varepsilon^n}\int\limits_{\mathbb R^n} d U 
\int\limits_{Q_\varepsilon } dX^n\int\limits_{Q_\varepsilon}dY^n\,
\frac{|f(U+X^j) - f(U+Y^j)|^p}{|X^n-Y^n|^{n+sp}}.
\endaligned
\ee*
If we take the partition \eqref{part} of $Q_\ve$ into account, we find that
\beq
\label{equb}
D_0=
\sum_{ q\in\{ -1, 1\}^{n-j}}\ 
\sum_{\sigma \in S_{n-j, n}}\ 
\sum_{ r\in \{ -1, 1\}^{n-j}}\ 
\sum_{ \tau \in S_{n-j,n}}\
D_{0,q,\sigma, r,\tau},
\eeq
where
\beq
\label{equc}
D_{0,q,\sigma, r,\tau}= \frac{1}{\ve^n} \int\limits_{\mathbb R^n} dU \int\limits_{Q_{\ve,
q,\sigma}} dX^n\int\limits_{Q_{\ve, r,\tau}} dY^n\
\frac{|f(U+X^j) - f(U+Y^j)|^p}{|X^n-Y^n|^{n+sp}}.
\eeq
We next consider a convenient parametrization of $Q_{\ve, q,\sigma}$, given by
\beq
\label{defXn}
\begin{cases}
(X^n)_{\sigma(1)} &=\ve q_{1} t_{1},\  0 \leq
t_{1} \leq 1\\
(X^n)_{\sigma(2)} &=\ve q_{2} t_{1} t_{2},\  0 \leq
t_{2} \leq 1\\
\quad \vdots  & \\
(X^n)_{\sigma(n-j)} &=\ve q_{n-j}t_{1}
t_{2}\ldots t_{n-j},\  0\leq t_{n-j}\leq 1\\
(X^n)_l &= \ve  t_{1}t_2\ldots
t_{n-j}\omega_l,\ |\omega_l|\leq 1,\ \forall\, l\neq\sigma(1), \ldots, \sigma(n-j)
\end{cases}. 
\eeq
We note that
\beq
\label{defXj}
(X^j)_l=\begin{cases}\ve q_{\sigma^{-1}(l)},& \text{ if }l=\sigma (i)\text{ for some }i\\
\ve\omega_l, & \text{ else}
\end{cases}.
\eeq
In particular, $X^j$ depends only on the $\omega_l$'s, not on the $t_i$'s; this will be used to give a meaning to \eqref{equa} below.

We consider a similar parametrization of $C_{\ve, r, \tau},$ the $t$'s
being replaced by $u$'s and the $\omega$'s by $\lambda$'s.

We use the following compact notations: 
\be*
\omega=(\omega_l)_{l\not\in\sigma \left(\llbracket 1,n-j\rrbracket\right)}  \text{ and } t=(t_i)_{i\in\llbracket 1,n-j\rrbracket}
\ (\lambda, u \text{ are defined similarly}). 
\ee*
Note that $X^j$ depends only on $\omega$, $\sigma$ and $q$; similarly, $Y^j$ is expressed in terms of $\lambda$, $\tau$ and $r$.\\
With the convention that $0\le t\le 1$ stands for $0\le t_i\le 1$ for each $i$, we find that
\beq
\label{equa}
D_{0,q,\sigma, r,\tau}= \frac{1}{\ve^{sp}} \int\limits_{\mathbb R^n} dU
\int\limits_{|\omega|\leq 1} d\omega \int\limits_{|\lambda|\leq
1} d\lambda\  k(\omega, \lambda)\, |f(U+X^j) - f(U+Y^j)|^p,
\eeq
where
\be*
k(\omega, \lambda) = \int\limits_{0\leq t\leq 1}
dt\int\limits_{0\leq u\leq 1} du\
 t^{n-1}_{1} \ldots  t^j_{n-j} u^{n-1}_{1}
\ldots u^j_{n-j}\ \frac{\ve^{n+sp}}{|X^n-Y^n|^{n+sp}}.
\ee*

We rely on the following generalization of Lemma \ref{lema2}.
\begin{lemm}
\l{lemb2}
  Let $0<s< 1, sp < j + 1\leq n$.  Then 
\beq
\label{b7}
k(\omega, \lambda) \leq \frac{C\ve^{j+sp}}{|X^j-Y^j|^{j+sp}}.
\eeq
\end{lemm}
The case $j=
n-1$ corresponds to Lemma \ref{lema2}.
 \begin{proof} 
We note that inequality \eqref{b7} makes sense, since $X^j$ (respectively
$Y^j$) depends only on $\omega$ (respectively $\lambda$).\\
On the other hand, the formula that gives $k(\omega,\lambda)$ does not depend on $\ve$; neither does the r.h.s. of \eqref{b7}. Thus, when we estimate $k(\omega,\lambda)$, we may assume that $\ve=1$. \\ We proceed by induction on $n$: assuming that \eqref{b7} holds for all integers $m\le n-1$ and all $j\le m-1$, we prove it for $n$ and each $j\le n-1$.  Note that the case $n=1$ (and $j=0$) is covered by Lemma \ref{lema2}.\\
Since $X^n = t_{1} X^{n-1}$ and $Y^n =
u_{1} Y^{n-1}$, we have 
\be*
k(\omega, \sigma) = \int\limits_{0\leq  t\leq 1} d t
\int\limits_{0\leq  u\leq 1} d u\  t_{1}^{n-1}\ldots
 t^j_{n-j}
u^{n-1}_{1} \ldots u^j_{n-j} \
\frac{1}{|t_{1}X^{n-1}- u_{1} Y^{n-1}|^{n+sp}}.
\ee*
Using the fact that $|X^{n-1} | = |Y ^{n-1}|=1$ and Lemma \ref{lema2}, we find that
\beq
\label{new}
\int\limits^1_0 t_{1}^{n-1} dt_{1} \int\limits^1_0
u^{n-1}_{1}du_{1}\ \frac{1}{|t_{1}X^{n-1}-
u_{1} Y^{n-1}|^{n+sp}}\leq \frac{C}{|X^{n-1}-
 Y^{n-1}|^{n+sp-1}}.
\eeq
If $j=n-1$, then \eqref{new} is the desired inequality. Assume $j<n-1$. Then \eqref{new} implies that
\beq
\label{b8}
k(\omega,\sigma)\leq C\int\limits_{0\leq \widehat t_1\leq 1} d\widehat t_1
\int\limits_{0\leq \widehat u_1\leq 1} d\widehat u_1\  t_{2}^{n-2}\ldots
 t^j_{n-j}
u^{n-2}_{2} \ldots u^j_{n-j} \
\frac{1}{|X^{n-1}- Y^{n-1}|^{n+sp-1}}.
\eeq

Next we note that one of the three cases occurs.

\medskip
\noin{\bf Case 1.}  $\sigma (1)= \tau (1), q_{1} = r_{1}$.
\\
{\bf Case 2.}  $\sigma (1)= \tau (1), q_{1} \neq r_{1}$.
\\
{\bf Case 3.}  $\sigma (1)\neq \tau (1)$.

We will estimate the right-hand side of \eqref{b8} in each of these cases.

\medskip
\noin{\bf Case 1.}  Assume e.g. $\sigma(1) = \tau(1) = 1, q_1 =
r_1 = 1$.  In this case, the first coordinate of $X^{n-1}$ or $Y^{n-1}$ is $1$, so that
\be*
|X^{n-1} - Y^{n-1}| = |\widehat{X^{n-1}}_1
- \widehat{Y^{n-1}}_1 |.
\ee*
The vectors $\widehat{X^{n-1}}_1$ and
$\widehat{Y^{n-1}}_1$ belong to $\mathbb R^{n-1}$ and are obtained from $\omega$ and $\lambda$ via \eqref{defXn}, with an obvious shift in the indices of the coordinates and with $n$ replaced by $n-1$.\\
Thus, in this case, \eqref{b7} follows from \eqref{b8} and
the fact that the conclusion of the lemma holds for $n-1$ and $j$.

\medskip
\noin{\bf Case 2.} In this case, we have $|X^{n-1}-
 Y^{n-1}|= 2$ and $|X^{j}-
 Y^{j}|= 2$.  Inequality \eqref{b7} follows easily from
\eqref{b8}.

\medskip
\noin{\bf Case 3.} With no loss of generality, we may assume
$\sigma(1)=1, \tau(1) = 2, q_1 = 1, r_1 = 1$.  Thus 
\beq
\label{bb1}
X^{n-1}=e_1 + t_{2} v,\  v\bot e_1,\ | v|=1
\eeq
and 
\beq
\label{bb2}
\hskip 4mm Y^{n-1}=e_2 + u_{2} w,\  w\bot e_2,\ | w|=1.
\eeq
We rely on the fact that \eqref{bb1}-\eqref{bb2} imply
\beq
\label{b9}
|(e_1+t_2v) - (e_2+ u_2w) | \geq C |(e_1+t_2v) - (e_2+w)|,\  0\le t_2, u_2\le 1.
\eeq
The  proof of this inequality is postponed (see Lemma \ref{lemb3} below).\\
Using \eqref{b9}, we obtain
\beq
\label{b10}
\aligned
M:&=\int\limits^1_0 t^{n-2}_{2} dt_{2} \int\limits^1_0
u^{n-2}_{2} du_{2}\ 
\frac{1}{|X^{n-1}-
 Y^{n-1}|^{n+sp-1}}\\
&\leq C\int\limits^1_0 t^{n-2}_{2} dt_{2} \int\limits^1_0
u^{n-2}_{2} du_{2}\ \frac{1}{|(e_{1} + t_{2}v)- (e_{2} + w)|^{n+sp-1}}\\
&\leq C\int\limits^1_0 t^{n-2}_{2} dt_{2}\ \frac{1}{|(e_{1} + t_{2}v)- (e_{2} + w)|^{n+sp-1}}\leq\frac{C}{|(e_{1} + v)- (e_{2} + w)|^{n+sp-2}}.\endaligned
\eeq
The  last inequality in \eqref{b10} is a consequence of \eqref{gene}.\footnote{We are in position to apply \eqref{gene} since $sp<n-1$.} 

Since $e_{1} + v =X^{n-2}$ and $e_2+w=Y^{n-2}$, we find that, with $\overline t=(t_3,\ldots, t_{n-j})$ and $\overline u=(u_3,\ldots, u_{n-j})$, we have
\be*
k(\omega, \sigma) \leq
 C\int\limits_{0\le \overline t\le 1}\, dt\ \prod_{i=3}^{n-j}t_i^{n-i} \int\limits_{0\le \overline u\le 1}\, du\ \prod_{l=3}^{n-j}u_l^{n-l}\ \frac{1}{|X^{n-2} - Y^{n-2}|^{n+sp-2}}.
\ee*
If $j=n-2$, then we are done. Otherwise, we continue as in the estimate of \eqref{b8}, distinguishing at each step the three cases
mentioned before (and using again the induction assumption when
encountering Case 1).  At the end of this process, we are led to 
\be*
k(\omega, \lambda) \leq \frac{C}{|X^j-Y^j|^{j+sp}},\  \forall\, j\in \llbracket 0, n-1\rrbracket,
\ee*
assuming the same inequality valid up to $n-1$. 

The proof of Lemma \ref{lemb2} is complete.
\end{proof}

As promised, we now established \eqref{b9}. 
\begin{lemm} 
\l{lemb3} 
If $0\leq t, u\leq 1$ are real numbers, and if $v \bot e_1$, $|v| = 1$, $w\bot
e_2$, $|w| = 1$, then
\beq
\label{b14}
|(e_1+tv) - (e_2+ uw) | \geq C |(e_1+tv) - (e_2+w)|.
\eeq
\end{lemm}
\begin{proof}[Proof of Lemma \ref{lemb3}]  Assume first that $\d u\leq\frac{1}{2}$.  Then
\be*
|e_1 + tv - (e_2 + uw)|\geq |\langle e_1 + tv - (e_2 + uw), e_1\rangle |= |1-uw_1|\geq \frac{1}{2}
\ee*
and $|e_1 + tv
- (e_2 + w)| \leq 4$, 
so that \eqref{b14} is clear in this case.
\medskip

We next consider the case $\d u\geq\frac{1}{2}$.  Consider the following  compact subset of $\partial Q_1$:
\be*
\mathscr K = \{X\in
\mathbb R^n;\, |X|=1\}\setminus \{ X\in \mathbb R^n;\,  X_2 = 1, |\widehat{X}_2| <
1\} .
\ee*
Let $P$ be the radial
projection centered at $e_2$ of $Q_1 \setminus\{ e_2\}$ onto $\mathscr K$.\footnote{$P$ is given by the formula $\d P(X)=e_2+\tau (X-e_2)$, where $\tau$ is the only number $\ge 1$ such that $|e_2+\tau (X-e_2)|=1$.}  Then
\begin{gather}
\label{bb15}
P \text{ is Lipschitz in } Q_1\setminus Q_{1/2}(e_2),
\\
\l{bebe17}
P (e_2 + uw) = e_2+w \text{ and }
P (e_1 + t v) = e_1 +tv,
\\
\label{bb16}
e_1 + tv, e_2 + uw \in Q_1\setminus Q_{1/2}(e_2) 
\text{ if } u \geq 1/2.
\end{gather}
Inequality \eqref{b14}, which is equivalent to
\be*
|P(e_1 +tv)-P (e_2+uw)|\leq \frac{1}{C}|(e_1+tv) - (e_2 +
uw)|,
\ee*
is then a consequence of \eqref{bb15} - \eqref{bb16}.

The proof of Lemma \ref{lemb3} is complete.
\end{proof}
\medskip

\noin{\bf Step 3.1 continued.}  Recall that we want to establish an estimate of the form $D_0\le a(\ve)+b(\ve)$.

For this purpose, we start by establishing \eqref{abae}, which is the analog of \eqref{a21} adapted to the case of a general $j$.

By Lemma \ref{lemb2} and \eqref{equa}, we find that
\beq
\label{b18}
D_{0,q,\sigma, r,\tau}
\leq C\ve^j \int\limits_{\mathbb R^n} dU
\int\limits_{|\omega|\leq 1} d\omega \int\limits_{|\lambda|\leq
1} d\lambda\ \frac{|f(U+X^j) - f(U+Y^j)|^p}{|X^j-Y^j|^{j+sp}}:=\overline D_{0,q,\sigma, r,\tau}:=\overline D_0.
\eeq
Estimate \eqref{b18} leads to the following: 
\beq
\label{b19}
\aligned
\overline D_0&\le \frac C{\ve^{sp}}\int\limits_{\mathbb R^n} dU
\int\limits_{\shortstack{$\scriptstyle |\omega_l|\leq 2,\, \forall\, l \in
(\sigma,\tau)_1$ \\  $\scriptstyle\omega_l = q_{\sigma^{-1}(l)} - r_{\tau^{-1}(l)},\, \forall\, l\in (\sigma, \tau)_2$}}
\bigotimes_{l \in (\sigma, \tau)_1}d\omega_l\   \frac{|f(U+\ve\omega) - f(U)|^p}{|\omega|^{j+sp}}\\
&=C\ve^{n(\sigma,\tau)}\int\limits_{\mathbb R^n} dU\int\limits_{\shortstack{$\scriptstyle |\omega_l|\leq 2\ve,\, \forall\, l \in
(\sigma,\tau)_1$ \\ $\scriptstyle\omega_l = \ve(q_{\sigma^{-1}(l)} - r_{\tau^{-1}(l)}),\, \forall\,  l\in (\sigma, \tau)_2$}}
\bigotimes_{l \in (\sigma, \tau)_1}d\omega_l\   \frac{|f(U+\omega) - f(U)|^p}{|\omega|^{j+sp}}.
\endaligned
\eeq

Here,
\be*
(\sigma,\tau)_2 = \sigma(\{1, \ldots, n-j\}) \cap \tau(\{1, \ldots,
n-j\})\subset\{1, \ldots, n\};
\ee*
\be*
(\sigma,\tau)_1 = \{1, \ldots,  n\}\setminus (\sigma,
\tau)_2;
\ee*
\be*
n(\sigma,\tau) = j-n + \#(\sigma, \tau)_2.
\ee*

Indeed, inequality \eqref{b19} is easily proved by noting that 
\be*
(X^j-Y^j)_l = \begin{cases}
\ve(q_i - r_m),& \text{if } l=\sigma(i)=\tau(m)\in (\sigma, \tau)_2\\
\ve (q_i - \lambda_m),& \text{if } l=\sigma(i)\in\sigma(\{1,\ldots, n-j\})\setminus
(\sigma, \tau)_2\\
\ve(\omega_l- r_m),& \text{if } l=\tau(m)\in\tau(\{1,\ldots,
n-j\})\setminus(\sigma, \tau)_2\\
\ve(\omega_l - \lambda_l),& \text{if } l\notin \sigma(\{1,\ldots, n-j\})\cup
\tau(\{1,\ldots, n-j \})\end{cases}.
\ee*

For further use, let us prove that
\beq
\label{abaf}
n-2j\le \#(\sigma, \tau)_2\le n-j.
\eeq

 To see this, we note that on the one hand we have
\be*
\d \# [\sigma(\{1,\ldots, n-j\})\cup
\tau(\{1,\ldots, n-j \})]=2n-2j-\#(\sigma, \tau)_2\le n.
\ee*

On the other hand, clearly $\#(\sigma, \tau)_2\le n-j$.

If we insert \eqref{b19} into \eqref{b18} and next take the sum over $q,\sigma,
r, \tau$ and use \eqref{abaf}, we obtain the following analog of \eqref{a21}:
\beq
\label{abae}
\aligned
D_0 &\leq C \sum^{n-j}_{ k= \min (0, n-2j)}  \ \sum_{ \shortstack{$\scriptstyle A\subset\{1,
\ldots, n\}$ \\ $\scriptstyle \# A= k$}} \frac{1}{\ve^{n-j-k}}  \int\limits_{\mathbb R^n}dU
\int\limits_{\shortstack{$\scriptstyle\omega_l
\in\{0, \pm 2\ve\},\, \forall\, l\in A$ \\ $\scriptstyle |\omega_l|\leq 2\ve, \, \forall\, l\notin
A$}}\bigotimes_{l\notin A}d\omega_l\ \frac{|f(U+\omega) - f(U)|^p}{|\omega|^{j+sp}}
\\
&
:= C \sum_{ k, A}  D_{0,k, A}.
\endaligned
\eeq
We complete the proof of Step 3.1 by estimating each $D_{0,k,A}$.  By
symmetry, it suffices to estimate the  integrals
\be*
I_{l, m} = \frac{1}{\ve^{n-j-(l +m)}}\int\limits_{\mathbb
R^n} dU \int\limits_{\shortstack{$\scriptstyle\omega_1 = \cdots
= \omega_l = 2\ve$ \\  $\scriptstyle\omega_{l+1} = \cdots = \omega_{l+m} = 0$ \\ $\scriptstyle |\omega_k|\leq 2\ve,\, \forall\, k>l +m$}}d\omega_{l+m+1}\ldots d\omega_n\, \frac{|f(U+\omega) - f(U)|^p}{|\omega|^{j+sp}},
\ee*
with $\max\{0,n-2j\}\leq l + m \leq n-j$.

\medskip
\noin{\bf Case 1.}  $l= 0, m> 0$\\
 In this case, we have 
\beq
\label{ud1}
I_{0,m}=\frac{1}{\ve^{n-j-m}}\int\limits_{\mathbb
R^n} dU \int\limits_{\shortstack{$\scriptstyle\omega_1 = \cdots
= \omega_m =0$ \\ $\scriptstyle |\omega_k|\leq 2\ve,\,\forall\,  k>m$}}d\omega_{m+1}\ldots d\omega_n\, \frac{|f(U+\omega) - f(U)|^p}{|\omega|^{j+sp}}.
\eeq

\medskip
\noin{\bf Case 1.1.}  $m=n-j$

\medskip
\noindent
By Lemma \ref{lema6}, we have
\be*
\aligned
I_{0,n-j} &= \int\limits_{\mathbb R^n} dU \int\limits_{\shortstack{$\scriptstyle\omega'\in \mathbb R^j$ \\ $\scriptstyle |\omega'|\leq 2 \ve$}}
d\omega'  \ \frac{ |f(U + (0,\ldots
0, \omega')) - f(U)|^p}{|\omega'|^{j+sp}}
\leq C\int\limits_{\mathbb R^n} dU\int\limits_{\shortstack{$\scriptstyle\omega\in\mathbb R^n$ \\ $\scriptstyle |\omega|\leq 4\ve$}} d\omega\ \frac{|f(U+\omega) -f(U)|^p}{|\omega|^{n+sp}}\\
&=C \iint\limits_{|U-V|\le 4\ve} dU\, dV\,  \frac{|f(U) -
f(V)|^p}{|U-V|^{n+sp}} \to 0 \text{ as } \ve\to 0.
\endaligned
\ee*

\medskip

\noin{\bf Case 1.2.}  $m<n-j$

\medskip
\noindent
Using again Lemma \ref{lema6}, we find 
\be*
\aligned
\int\limits^1_0\frac{I_{0,m}}{\ve} d\ve &=\int\limits^1_0 \frac{d\ve}{\ve^{n+1-j-m}} \int\limits_{\mathbb R^n} dU\int\limits_{\shortstack{$\scriptstyle\omega'\in \mathbb R^{n-m}$ \\ $\scriptstyle |\omega'|\leq 2\ve$}} d\omega'\
\frac{|f(U+(0,\ldots , 0,
\omega'))-f(U)|^p}{|\omega'|^{j+sp}}\\
&\leq C\int\limits_{\mathbb R^n} dU\int\limits_{\shortstack{ $\scriptstyle\omega' \in \mathbb R^{n-m}$ \\ $ \scriptstyle|\omega'|\leq 2$}}
d\omega'\ \frac{ |f(U +
(0, \ldots, 0, \omega') )- f(U)|^p}{|\omega'|^{n-m+sp}}\\
&\leq
C\int\limits_{\mathbb R^n} dU\int\limits_{\shortstack{$\scriptstyle\omega\in \mathbb R^n$ \\ $\scriptstyle |\omega'|\leq 4$}} d\omega\
 \frac{|f(U+\omega) -
f(U)|^p}{|\omega|^{n+sp}}= C \iint\limits_{|U-V|\le 4} dU dV\,
\frac{|f(U)-f(V)|^p}{|U-V|^{n+sp}} < \infty.
\endaligned
\ee*

\medskip
\noin{\bf Case 2.}  $l > 0$\\
In this case we have $|\omega| = 2\ve$.  We set $V=U +2\ve e_1 + \cdots
+2\ve e_l$ and $\omega'=\d (\omega_k)_{k\in\llbracket l+m+1,n\rrbracket}$. Since
\be*
\aligned
|f(U+\omega) - f(U)|^p \leq C&(|f(U+ 2\ve e_1) - f(U)|^p + |f(U+ 2\ve
e_1 + 2 \ve e_2) - f(U + 2\ve e_1)|^p \\
&+ \cdots + |f(V) - f(U+2\ve e_1 +\cdots + 2\ve e_{l-1})|^p
\\&
+|f(V+(0,
\ldots, 0, \omega'))- f(V)|^p),
\endaligned
\ee*
we find that
\be*
\aligned
I_{l, m}\leq &\frac{C}{\ve^{sp}} \sum^l_{j=1}\  \int\limits_{\mathbb
R^n} dU\  |f(U+ 2\ve e_j) - f(U)|^p
\\
&+\frac{C}{\ve^{n-(l+m) + sp}}\int\limits_{\mathbb R^n} dU\int\limits_{\shortstack{ $\scriptstyle\omega'\in\mathbb
R^{n-(l+m)}$ \\ $\scriptstyle|\omega'| \leq 2\ve$}} d\omega'\  |f(U+(0, \ldots, 0 , \omega'))-f(U)|^p
:=
C\left(\sum^l_{j=1} P_j + P_0\right).
\endaligned
\ee*
{\bf Estimate of $P_0$.} 
We have
\be*
\aligned
\int\limits^1_0 \frac{P_0}{\ve} d\ve &= \int\limits^1_0
\frac{d\ve}{\ve^{n+1-(l + m) + sp}} \int\limits_{\mathbb R^n} dU
 \int\limits_{ \shortstack{$\scriptstyle\omega' \in \mathbb
R^{n-(l +m)}$ \\ $\scriptstyle |\omega'|\leq 2 \ve$}}d\omega'\ |f(U + (0, \ldots, 0, \omega')) - f(U)|^p
\\
&\leq C \int\limits_{\mathbb R^n} dU\int\limits_{\shortstack{ $\scriptstyle\omega' \in \mathbb R^{n-(l+m)}$ \\ $\scriptstyle |\omega'| \leq
2$}} d\omega'\ \frac{|f(U+(0,\ldots,0,
\omega')) - f(U)|^p}{|\omega'|^{n-(l+m)+ sp}}
\\
\
&\leq C\int\limits_{\mathbb R^n} dU\int\limits_{\shortstack{ $\scriptstyle\omega\in \mathbb R^n $ \\ $ \scriptstyle|\omega| \leq
4$}} d\omega\  \frac{|f(U+\omega) -
f(U)|^p}{|\omega|^{n+ sp}}= C\iint\limits_{|U-V|\le 4}dUdV\ \frac{|f(U) -
f(V)|^p}{|U-V|^{n+ sp}}<\infty;
\endaligned
\ee*
here, we used Lemma \ref{lema6}.

\medskip
\noin{\bf Estimate of $P_1$.} (The estimates of $P_2, \ldots, P_l$
are similar.) By Lemma \ref{lema6}, 
we have 
\be*
\aligned
\int\limits^1_0 \frac{P_1}{\ve} d\ve &
= \int\limits^1_0
\frac{d\ve}{\ve^{1+sp}} \int\limits_{\mathbb R^n} dU\, |f(U+2\ve e_1) -
f(U)|^p\\
&
\le C\int\limits_{\mathbb R^n} dU\int\limits_{\shortstack{ $\scriptstyle\omega \in \mathbb R^n$ \\ $\scriptstyle |\omega|\leq 4$}}
d\omega\
\frac{|f(U+\omega) - f(U)|^p}{|\omega|^{n+sp}} = C\iint\limits_{|U-V|\le 4}
dUdV\, \frac{|f(U) - f(V)|^p}{|U-V|^{n+sp}} < \infty.
\endaligned
\ee*

\medskip
\noin{\bf Case 3.}  $l = m=0$\\
In this case, the inequality 
\be*
\ve^{n-j}|\omega|^{j+sp}\ge 2^{j-n}
|\omega|^{n+sp} \text{ if }|\omega|\le 2\ve
\ee*
yields
\be*
\aligned
I_{0,0} &= \frac{1}{\ve^{n-j}}\int\limits_{\mathbb R^n} dU \int\limits_{|\omega|\leq 2\ve} d\omega\
 \frac{|f(U+\omega) - f(U)|^p}{|\omega|^{j+sp}}\leq C\int\limits_{\mathbb R^n} dU\int\limits_{|\omega|\leq 2\ve} d\omega\
 \frac{|f(U+\omega) -
f(U)|^p}{|\omega|^{n+sp}}\\
&= C\iint\limits_{|U-V|\le 2\ve} dUdV\, \frac{|f(U) -
f(V)|^p}{|U-V|^{n+sp}} \to 0\ \text{ as } \ve \to 0.
\endaligned
\ee*
Step 3.1 is complete.

\medskip
\noin{\bf Step 3.2.}  Estimate of $D_L, L \in \mathbb Z^n, |L|=1$

\medskip
\noindent
The proof is essentially the same as for $D_0$.  One has to use
instead of Lemma \ref{lemb2} the following
\begin{lemm}
\l{lemb4}  Assume that $sp<j+1$. Let $L\in \mathbb Z^n$, with $|L| =1$.
Set
\be*
k(\omega, \lambda) = \int\limits_{0\leq t \leq 1}
dt \int\limits_{0\leq u \leq 1}
du\ t^{n-1}_{1} \ldots t^j_{n-j}
u^{n-1}_{1} \ldots u^j_{n-j}\ \frac{\ve^{n+sp}}{|X^n-(2\ve L +
Y^n)|^{n+sp}}.
\ee*
Then
\be*
k(\omega,\lambda) \leq \frac{C\ve^{j+sp}}{|X^j - (2\ve L + Y^j)|^{j+sp}}.
\ee*
\end{lemm}

Lemma \ref{lemb4}, which generalizes Lemma \ref{lema7}, is a rather  straightforward variant of  Lemma Lemma \ref{lemb2}. Its proof requires the following variant of \eqref{b9}:
\be*
|(e_1+t_2v)-(e_2+u_2w+2L)|\ge C|(e_1+t_2v)-(e_2+w+2L)|, 
\ee*
which is obtained by adapting the argument in the proof of  Lemma \ref{lemb3}.

Using Lemma \ref{lemb4}, we estimate $D_L$ as in Step 3.2 in Section \ref{appa}; the details  are left
to the reader.  The proof of Lemma \ref{lemb1} is complete.
\end{proof}

\begin{rema}
\label{uc5}
By Steps 3.1 and 3.2, we have an estimate of the form
 \be*
 \sum_{\shortstack{$\scriptstyle L\in\Z^n$ \\ $\scriptstyle |L|\le 1$}}\overline D_L\le a(\ve)+b(\ve).
 \ee*
 
 Here, $\overline D_0$ is as in \eqref{b18}, and the quantities $\overline D_L$ are defined similarly (this is implicit in Step 3.2). The numbers $a(\ve)$ and $b(\ve)$ satisfy  \eqref{e05092}. If we adapt the averaged estimates leading to the existence of $b(\ve)$ (more specifically, to the estimates of $P_0$, of $P_1$, and of $I_{0,m}$ with $m<n-j$), we see that, for a {\it fixed} $\ve$, there exists some $C(\ve)$ such that
 \beq
 \label{yd1}
 \sum_{\shortstack{$\scriptstyle L\in\Z^n$ \\ $\scriptstyle |L|\le 1$}}\overline D_L\le C(\ve)|f|_{W^{s,p}(\R^n)}^p.
 \eeq
  
 In order to justify the above, we  examine e.g. the case of $I_{0,m}$ and of $P_1$, the other cases being similar. 
 
 By \eqref{ud1} and Lemma \ref{lema6}, we have
\be*
I_{0,m}=\frac{1}{\ve^{n-j-m}}\int\limits_{\mathbb
R^n} dU \int\limits_{\shortstack{$\scriptstyle\omega_1 = \cdots
= \omega_m =0$ \\ $\scriptstyle |\omega_k|\leq 2\ve,\,\forall\,  k>m$}}d\omega_{m+1}\ldots d\omega_n\, \frac{|f(U+\omega) - f(U)|^p}{|\omega|^{j+sp}}\le \frac C{\ve^{n-j-m}}|f|_{W^{s,p}(\R^n)}^p.
\ee*
 
 We next estimate $P_1$. We may assume that $\ve=1/2$. We start from the following Poincar\'e type inequality for functions $g:\R\to \R$ \cite{leta}:
 \beq
 \label{fgi1}
 \int_0^{2}\left|g(t)- \fint_0^2 g \right|^p\, dt\le C\int\limits_0^2\int\limits_0^2dtd\tau\,\frac{|g(t)-g(\tau)|^p}{|t-\tau|^{1+sp}}.
 \eeq
 
 Using \eqref{fgi1}, we obtain
 \beq
 \label{fgi2}
 \int_\sigma^{\sigma+1}|g(t+1)-g(t)|^p\, dt\le C\int\limits_\sigma^{\sigma+2}\int\limits_\sigma^{\sigma+2}dtd\tau\, \frac{|g(t)-g(\tau)|^p}{|t-\tau|^{1+sp}},\ \forall\, \sigma\in\R.
 \eeq
 
 Integration of \eqref{fgi2} with respect to $\sigma$ leads to 
 \beq
 \label{fgi3}
 \int_\R^{\R+1}|g(t+1)-g(t)|^p\, dt\le C\iint\limits_{|t-\tau|<2}dtd\tau\,\frac{|g(t)-g(\tau)|^p}{|t-\tau|^{1+sp}},
 \eeq 
 and thus
 \beq
 \label{fgi4}
 P_1=C\int\limits_{\R^n}dU\, |f(U+e_1)-f(U)|^p\le C\int\limits_{\R^n}dU\int\limits_{|t|<2}dt\, |f(U+te_1)-f(U)|^p.
 \eeq
 
 We estimate $P_1$ via \eqref{fgi4} and Lemma \ref{lema6}.
 
  [Alternatively, we could have obtained the estimate $P_1\le C(\ve)|f|_{\wsp}^p$ directly by adapting the proof of Lemma \ref{lema6}.]
 
  The conclusion of this remark will be needed in order to complete the proof of Lemma \ref{l1aa} below.
\end{rema}
 
\medskip
We end this section with the 
\begin{proof}[Proof of Theorem \ref{thmd}]
Let $g\in\wsp (\R^n\, ; \R^m)$ be an extension of $f$, not necessarily  $F$-valued. We apply Lemma \ref{lemb1} to $g$. Let $\d g_k=g_{T_k,\ve_k}$ and let ${\mathscr C}(k)$ be the mesh of size $2\ve_k$ having $T_k$ as one of its centers. We take $\d f_k={g_k}_{|{\mathscr C}^k}$, where ${\mathscr C}^k$ is the union of cubes in ${\mathscr C}(k)$ which are contained in $\omega$.  Clearly, for large $k$ the maps $f_k$ have  all the desired properties.
\end{proof}

\begin{proof}[Short proof of Theorem \ref{thmd} when $1\le sp<n$]

We consider the mappings $f\xmapsto{F_\ve} F_\ve(f)$, where
\be*
F_\ve(f) : Q_\ve\times\R^n\to\R^m,\  F_\ve(f) (T, X)=f_{T,\ve}(X).
\ee*

Here, $f_{T,\ve}$ is the piecewise $j$-homogeneous extension associated to $T$ and $\ve$ as in this section.

\medskip
\noin{\bf Step 1.} Estimate for $s=0$\\
By estimate \eqref{a5} (which holds for an arbitrary $j$), for $1\le q<\infty$  we have
\beq
\label{ya0}
\|F_\ve(f)\|_{L^q(Q_\ve ; L^q(\R^n))}\le C\, \ve^{n/q} \| f\|_{L^q(\R^n)}, 
\text{ with $C$ independent of }\ve.
\eeq

\medskip
\noin{\bf Step 2.} Estimate for $s=1$\\
Let $1\le r<j+1$. We claim that 
\beq
\label{ya1}
\|F_\ve(f)\|_{L^r(Q_\ve ; W^{1, r}(\R^n))}\le C\, \ve^{n/r} \| f\|_{W^{1, r}(\R^n)}, 
\text{ with $C$ independent of }\ve\text{ or }f. 
\eeq

In view of Step 1, in order to obtain \eqref{ya1} it suffices to establish, with $C=C(n, j, r)$, the estimate
\beq
\label{yb1}
\int\limits_{Q_\ve}dT\int\limits_{\R^n}dX\, |\na f_T(X)|^r\le C\ve^n\int\limits_{\R^n}dX\, |\na f(X)|^r.
\eeq

We next observe that it suffices to prove \eqref{yb1} when $f\in C^\infty_c$. Indeed, assuming for the moment that \eqref{yb1} holds for such $f$, Step 1 combined with \eqref{yb1} for $f\in C^\infty_c$ and with a standard limiting argument implies that \eqref{yb1} holds for every $f\in W^{1,r}$.

We finally turn to the proof of \eqref{yb1} when $f\in C^\infty_c$. 
We use the same notation as at the beginning of this section: we set $U=T+2\ve K$, with $K\in\Z^n$, and we let $X^n$ be a point in $Q_\ve$, whose projection on the $j$-skeleton of $ Q_\ve$ is denoted $X^j$.  Set 
$
g_U(X^j)=f(U+X^j)$. Then for a.e. $X^n\in Q_\ve$ we have 
\beq
\label{yb2}
\na f_T(U+X^n)=\na g_U(X^j).
\eeq

We claim that \eqref{yb2} holds also in the  sense of distributions. Indeed, let ${\cal C}_{\ell, V, \ve}$ denote the $\ell$-skeleton obtained from the mesh of cubes of radius $\ve$ having $V$ as one of its centers. With this notation, the map  $f_T$ is locally Lipschitz in $\R^n\setminus {\cal C}_{n-j-1, W, \ve}$, where $W=T+(\ve, \ldots, \ve)$.  [The skeleton ${\cal E}={\cal C}_{n-j-1, W, \ve}$ is the \enquote{dual skeleton} of ${\cal C}_{j, T, \ve}$.] This observation leads to the validity of \eqref{yb2} in the  sense of distributions in $\R^n\setminus {\cal E}$. On the other hand, 
as we will see in a moment, we have 
\beq
\label{yb3}
|\na g_U(X^j)|\le C(f, \ve)\frac 1{\d\dist\, (U+X^n, {\cal E})}.
\eeq

In view of \eqref{yb3} and the fact that $f$ is compactly supported, we have
\beq
\label{yb4}
\na f_T\in L^1(\R^n\setminus {\cal E}).
\eeq
[Here, we also use the fact that $j\ge 1$ and thus $\cal E$ is a union of $m$-planes, with $m=n-j-1\le n-2$.]

In order to obtain \eqref{yb2}, it then suffices to invoke \eqref{yb4} and Lemma \ref{yb5}. [Note that this lemma applies to our situation since $j\ge 1$.]

In view of the above, it suffices to prove that
\beq
\label{yb6}
\int\limits_{Q_\ve}dT\, \sum_{K\in\Z^n}\int\limits_{Q_\ve}dX^n\, |\na g_{T+2\ve K}(X^j)|^r\le C\ve^n\int\limits_{\R^n}dX\, |\na f(X)|^r
\eeq
and to obtain, on the way, the estimate \eqref{yb3}. Splitting, in \eqref{yb6},  the integral in $X^n$ as a sum over $q\in \{ -1,1\}^{n-j}$ and over $\sigma\in S_{n-j, n}$, it suffices, by symmetry, to consider the case where  $X^n$ belongs to $Q_{\ve, q, \sigma}$, with 
\be*
q_i=1,\ \sigma(i)=i,\ \forall\, i\in \llbracket 1, n-j\rrbracket.
\ee*
With $q$ and $\sigma$ as above, every $X^n\in Q_{\ve, q, \sigma}$ satisfies
\beq
\label{yy1}
\ve \ge (X^n)_1\ge \cdots=(X^n)_{n-j}\ge \max\{ |(X^n)_i|;\,  n-j+1\le i\le n\}
\eeq
and
\beq
\label{yy2}
(X^j)_1=\cdots=(X^j)_{n-j}=\ve,\ (X^j)_i=\ve\frac{(X^n)_i}{(X^n)_{n-j}},\ \fo i\in \llbracket n-j+1, n\rrbracket.
\eeq

By \eqref{yb2}, \eqref{yy1} and \eqref{yy2}, for a.e. $X^n\in Q_{\ve, q, \sigma}$ we have
\beq
\label{yb7}
|\na f_T(U+X^n)|\le C\, \ve\, \sum_{i=n-j}^n\frac{|(X^n)_i|}{[(X^n)_{n-j}]^2}\, |\na f(U+X^j)|\le C\, \frac{\ve}{(X^n)_{n-j}}|\na f(U+X^j)|.
\eeq

In particular, \eqref{yb7} and the fact that 
\be*
\dist\, (U+X^n, {\cal E})=(X^n)_{n-j},\ \forall X^n\in Q_{\ve, q, \sigma},
\ee*
lead to \eqref{yb3}.

In view of \eqref{yb7}, in order to prove \eqref{yb6} it suffices to prove that 
\beq
\label{yb8}
I=\int\limits_{Q_\ve}dT\, \sum_{K\in\Z^n}\, \int\limits_{Q_{\ve, q, \sigma}}dX^n\, \frac{\ve^r}{[(X^n)_{n-j}]^{r}}\, |\na f(T+2\ve K+X^j)|^r
\le C\int\limits_{\R^n}dX\, |\na f(X)|^r.
\eeq

We let 
$X'=(X_1,\ldots, X_{n-j})$ and $Z''=(Z_{n-j+1},\ldots, Z_n)$, where 
\be*
Z_i=\ve\frac{(X^n)_i}{(X^n)_{n-j}}\in [-\ve, \ve],\ \forall\, i\in\llbracket n-j+1, n\rrbracket.
\ee*

We set
\beq
\label{fgy1}
W=(T_1+2\ve K_1+\ve,\ldots, T_{n-j}+2\ve K_{n-j}+\ve, T_{n-j+1}+2\ve K_{n-j+1}+Z_{n-j+1},\ldots, T_{n}+2\ve K_{n}+Z_{n}). 
\eeq

Then with the change of variables 
\be*
Q_{\ve, q, \sigma}\ni X\mapsto (X', Z'') 
\ee*
and with $W$ as in \eqref{fgy1} we have
\beq
\label{yb9}
I\le \int\limits_{Q_\ve}dT\, \sum_{K\in\Z^n}\ \int\limits_{0< (X^n)_{n-j}\le\cdots \le  (X^n)_1\le \ve}dX'\int\limits_{|Z''|\le \ve}dZ''\, \ve^{r-j}\, [(X^n)_{n-j}]^{j-r}\, |\na f(W)|^r.
\eeq

If we calculate, in \eqref{yb9}, the integral with respect to $X'$ and use the assumption $r<j+1$, we find (after summation in $K$) that
\be*
\begin{aligned}
I&\le C\, \ve^{n-j}\int\limits_{\R^n}dX\int\limits_{|Z''|<\ve}dZ''\, |\na f(X_1,\ldots, X_{n-j}, X_{n-j+1}+Z_{n-j+1},\ldots, X_n+Z_n)|^r\\
&= C\ve^n\int\limits_{\R^n}dX\, |\na f(X)|^r.
\end{aligned}
\ee*

Step 2 is now completed.

\medskip
\noin{\bf Step 3.} Estimate for $0<s<1$ (provided $sp\ge 1$ and $sp<j+1$)\\
Let $0<s<1$, $1\le p<\infty$ and $j\in\llbracket 1, n-1\rrbracket$ be such that $sp<j+1$. Pick  $1< q<\infty$ and $1< r<j+1$ such that
\beq
\label{ya2}
\frac 1p=\frac sr+\frac{1-s}q.
\eeq
This is always possible. Indeed,  since $sp<j+1$ we may pick some $r$ such that
\be*
\max\left\{ \frac 1{j+1},\, \frac 1{sp}-\frac 1s+1\right\}<\frac 1r<\frac 1{sp},
\ee*
and for any such $r$ the couple $(q, r)$, with  $q$ determined by \eqref{ya2}, has all the required properties.

We next recall three classical interpolation results. Given two Banach spaces $X$ and $Y$, we use the standard notation $[X, Y]_{s,p}$; see e.g. \cite[Section 1.5]{triebel1}. First, when \eqref{ya2} holds we have \cite[Section 2.4.2, Theorem 1 (a), eq. (2), p. 185]{triebel1}
\beq
\label{aaa1}
[W^{1,r}, L^q]_{s, p}=W^{s,p}.
\eeq

Next, if $X$ and $Y$ are Banach spaces and $s$, $p$, $q$, $r$ are as above, then \cite[Section 1.18.4, Theorem, eq. (3), p. 128]{triebel1}
\beq
\label{aaa2}
[L^r(\Omega ; X), L^q(\Omega ; Y)]_{s, p}=L^p(\Omega ; [X, Y]_{s,p}).
\eeq

By \eqref{aaa1} and \eqref{aaa2}, 
\beq
\label{ya3}
\text{with }r, q\text{ as in }\eqref{ya2},\text{ we have }[L^r(Q_\ve ; W^{1,r}(\R^n)) , L^q(Q_\ve ; L^q(\R^n))]_{s, p}=L^p(Q_\ve ; W^{s,p}(\R^n)).
\eeq

Final classical result. Let $s$, $p$, $q$, $r$, $X$ and $Y$ be as above. Let $F$ be a linear continuous operator from $X$ into $L^r(\Omega ; X)$ and from $Y$ into $L^q(\Omega ; Y)$. Then $F$ is linear continuous from $[X, Y]_{s,p}$ into $L^p(\Omega ; [X, Y]_{s,p})$ and satisfies the norm inequality
\beq
\label{aaa3}
\|F\|_{{\cal L}([X, Y]_{s,p} ; L^p(\Omega ; [X, Y]_{s,p}))}\le \|F\|_{{\cal L}(X ; L^r(\Omega ; X))}^s\, \|F\|_{{\cal L}(Y ; L^q(\Omega ; Y))}^{1-s}.
\eeq

By \eqref{ya0}, \eqref{ya1} and \eqref{aaa3}, we find that 
\beq
\label{ya4}
\|F_\ve(f)\|_{L^p(Q_\ve ; W^{s, p}(\R^n))}\le C\, \ve^{n/p} \| f\|_{W^{s, p}(\R^n)}, 
\text{ with $C$ independent of }\ve. 
\eeq

[In principle, the constant $C$ in \eqref{ya4} may depend on $\ve$, since we apply the interpolation result \eqref{ya3} in an $\ve$-dependent domain. The fact that $C$ does not depend on $\ve$ is obtained by a straightforward scaling argument: we consider, instead of $F_\ve$, the map 
\be*
G_\ve(f) : Q_1\times\R^n\to\R^m,\ G_\ve(f) (T, X)=f_{\ve\, T,\ve}(X).
\ee*
We obtain \eqref{ya4} by applying \eqref{aaa3} to $G_\ve(f)$ in  $Q_1$. Details are left to the reader.]

A clear consequence of \eqref{ya4} is
\beq
\label{ya5}
\frac 1{\ve^n}\int\limits_{Q_\ve}\|f_{T,\ve}-f\|_{W^{s,p}(\R^n)}^p\, dT\le C\|f\|_{W^{s,p}(\R^n)}^p.
\eeq

In order to complete the proof of Theorem \ref{thmd}, it suffices to obtain \eqref{ya6} below.
 
\medskip
\noin{\bf Step 4.} We have
\beq
\label{ya6}
\lim_{\ve\to 0}\, \frac 1{\ve^n}\int\limits_{Q_\ve}\|f_{T,\ve}-f\|_{W^{s,p}(\R^n)}^p\, dT=0,\ \forall\, f\in W^{s,p}(\R^n).
\eeq

Equation \eqref{ya6} is a version of \eqref{ya5} and is obtained as follows. We let $q$, $r$ be as in Step 3.

In view of \eqref{ya5}, it suffices to prove \eqref{ya6} when $f\in C^\infty_c$.  For such $f$, we have $f_{T, \ve}\to f$ uniformly in $T$ when $\ve\to 0$; this leads easily to
\beq
\label{ya7}
f_{T, \ve}\to f\text{ in }L^q\text{ uniformly in }T\text{ as }\ve\to 0.
\eeq

We obtain \eqref{ya6} via \eqref{ya7}, \eqref{ya1} and \eqref{aaa3}.
\end{proof}

\section{Restrictions of Sobolev maps to good complexes}
\l{appc}

Sections \ref{appc} to \ref{appe} are devoted to the proof of Theorem \ref{thme}. 

The current section is partly inspired by \cite[Appendix B, Appendix E]{bm}. 
The results we  prove here are fractional Sobolev versions of the following Fubini type result: if $f\in L^1(\R^2)$, then for a.e. $y\in\R$ we have $f(\cdot, y)\in L^1(\R)$.

 As elsewhere in this paper, we let $0<s<1$ and $1\le p<\infty$, and we let $f\in W^{s,p}(\R^n; \R^m)$.

We use notation consistent with  Section \ref{appb}, but we emphasize dependence of meshes on $T$ by writing, instead of ${\mathcal C}_{j}$,   ${\mathcal C}_{j, T, \ve}$ or (when $\ve$ is fixed) ${\mathcal C}_{j, T}$.  A generic point of ${\mathcal C}_{j, T}$ is denoted $X_\ast^j$, $Y_\ast^j,\ldots$ (instead of $U+X^j$ or $U+Y^j$). Also in order to be consistent with Section \ref{appb}, the projection of $X_\ast^n$ on ${\mathcal C}_{j, T}$ is denoted $X_\ast^j$. Similarly, if $j\ge 1$, then the projection of $X_\ast^j$ onto ${\mathcal C}_{j-1, T}$ is  denoted $X_\ast^{j-1}$; this projection is defined ${\mathcal H}^j$-a.e. on ${\mathcal C}_{j, T}$.

 Given a (say Borel and everywhere defined) map $f:\R^n\to\R^m$, an integer $j\in\llbracket  1, n-1\rrbracket$ and a point $T\in\R^n$, we define the norm 
\be*
\|f\|_{W^{s,p}({\mathcal C}_{j, T})}^p=\int\limits_{{\mathcal C}_{j, T}}d X_\ast^j\, |f(X_\ast^j)|^p +\iint\limits_{\shortstack{$\scriptstyle {\mathcal C}_{j, T}\times {\mathcal C}_{j, T}$ \\ $\scriptstyle |X_\ast^j-Y_\ast^j|<2\ve$}} dX_\ast^j   dY_\ast^j\,\frac{|f(X_\ast^j)-f(Y_\ast^j)|^p}{|X_\ast^j-Y_\ast^j|^{j+sp}}=\|f\|_{L^p({\cal C}_{j, T})}^p+|f|_{W^{s,p}({\cal C}_{j, T})}^p.
\ee*

The above definition extends to $j=0$ by replacing the integrals by sums. 

We will prove later in this section two results  on slicing, in which $\ve$ is fixed.
\begin{lemm}
\label{l1aa}
We have 
\begin{equation}
\label{e1aa}
\int\limits_{Q_\ve} \|f\|_{W^{s,p}({\mathcal C}_{j, T})}^p\, dT\le C(\ve)\|f\|_{W^{s,p}}^p, \ \forall\, j\in\llbracket 0, n-1\rrbracket.
\end{equation}
\end{lemm}

\begin{lemm}
\label{l1ab}
We have 
\begin{equation}
\label{e1ab}
\int\limits_{Q_\ve} dT\int\limits_{{\mathcal C}_{j, T}}
dX_\ast^j\, \frac{|f(X_\ast^j)-f(X_\ast^{j-1})|^p}{|X_\ast^j-X_\ast^{j-1}|^{sp}}
\le C(\ve)\|f\|_{W^{s,p}}^p,\ \forall\, j\in\llbracket 1, n-1\rrbracket.
\end{equation}
\end{lemm}

For $j\in\llbracket 1, n-1\rrbracket$, we define an ad hoc space ${\mathcal W}^{s,p}_j={\mathcal W}^{s,p}_{j, T,\ve}$ as follows: ${\mathcal W}^{s,p}_j$ consists of the functions  $g:{\mathcal C}_{j,T}\to\R^m$ such that
\beq
\label{e1ac}
\|g\|_{W^{s,p}({\mathcal C}_{\ell, T})}<\infty,\ \forall\, \ell\in \llbracket 1, j\rrbracket
\eeq
and
\beq
\label{e1ad}
\int\limits_{{\mathcal C}_{\ell, T}}dX_\ast^\ell\, 
\frac{|f(X_\ast^\ell)-f(X_\ast^{\ell-1})|^p}{|X_\ast^\ell-X_\ast^{\ell-1}|^{sp}}<\infty,\ \forall\, \ell\in \llbracket 1, j\rrbracket. 
\eeq

Though this is not needed in order to understand the remaining part of this article, we pause here to comment the definition of ${\mathcal W}^{s,p}_j$, which is inspired by Hang and Lin \cite[Section 3]{hl} and also by \cite{bm}.  \cite{hl} deals with $W^{1,p}$-maps, and ${\mathcal W}^{1,p}_j$ is defined there as the space of (say Borel) maps defined on ${\cal C}_j$ such that $f_{|{\cal C}_\ell}$ is in $W^{1,p}({\cal C}_\ell)$, $ \ell=\llbracket 1, j\rrbracket$, and such that 
\beq
\label{fho1}
\tr \left( f_{|{\cal C}_\ell} \right)=f_{|{\cal C}_{\ell-1}}, \ \ell=\llbracket 1, j\rrbracket.
\eeq

Clearly, if $f\in {\mathcal W}^{s,p}_{j, T, \ve}$ and $\ell\le j$, then the restriction of $f$ to an $\ell$-dimensional cube $C$ of the mesh ${\mathcal C}_{\ell, T, \ve}$ belongs to $W^{s,p}(C)$. When $sp>1$ (and thus maps in $\wsp$ have traces), one may prove that condition \eqref{e1ad} implies 
\eqref{fho1}.

When $sp=1$, we are in a limiting case of the trace theory: maps in $W^{1/p, p}$ do not have traces, but sometimes have \enquote{good restrictions} \cite[Appendix B]{bm}. In this case, condition \eqref{e1ad} implies that $f_{|{\cal C}_{\ell-1}}$ is the good restriction to ${\cal C}_{\ell-1}$ of $f_{|{\cal C}_{\ell}}$ (which is the substitute of \eqref{fho1} when $sp=1$). 

When $sp<1$, one may still view \eqref{e1ad} as a substitute of \eqref{fho1}. Note however that in this case condition \eqref{e1ad} is very mild, since the value of $f$ at the interior of ${\cal C}_{\ell}$ combined with condition \eqref{e1ad} does not determine ( ${\cal H}^{\ell-1}$-a.e.) the value of $f$ on  ${\cal C}_{\ell-1}$.

As we will see in the next section, property  \eqref{e1ad} is essential in the proof of Lemma \ref{ua1}.

Let $s$, $p$ be such that $1\le sp<n$. Let $j$ be an integer such that $sp<j+1\le n$. For such $j$, we consider $f_{T,\ve}$ as in Section \ref{appb}. Combining  Lemmas \ref{l1aa} and \ref{l1ab} with the fact that, by the proof of Lemma \ref{lemb1}, there exists a sequence $\ve_k\to 0$ such that
\be*
\frac 1{(\ve_k)^n}\int\limits_{Q_{\ve_k}}\|f-f_{T,\ve_k}\|_{W^{s,p}}^p\,  dT\to 0\ \text{as }k\to\infty,
\ee*
 we obtain the following
\begin{coro}
\label{c1aa}
Let $s$, $p$, $j$ be such that $1\le sp<j+1\le n$. Let $f\in W^{s,p}(\R^n ; \R^m)$ be a Borel function. Then there exist sequences $\ve_k\to 0$ and $\{ T_k\}\subset\R^n$ such that:\\
1. The restriction  $f^k$ of $f$ to ${\mathcal C}_{j, T_k,\ve_k}$ belongs to ${\mathcal {W}}^{s,p}_{j, T_k, \ve_k}$, $\forall\, k$.\\
2. If $f_k$ is the $j$-homogeneous extension  of $f^k$, then $f_k\to f$ in $W^{s,p}$ as $k\to\infty$.
\end{coro}

The remaining part of this section is devoted to the proofs of Lemmas \ref{l1aa} and \ref{l1ab}. 

A word about the proofs. Many of the calculations we need in  Sections \ref{appc}--\ref{appe}  are quite close to the ones in Section \ref{appb}. For such calculations, we point to the analog formulas in Section \ref{appb} and omit part of details.

We will use the same notation as in Section \ref{appb}, and more specifically as in Step 3.1 in the proof of Lemma \ref{lemb1}; see on the one hand \eqref{equb} and \eqref{equc}, and on the other hand \eqref{defXn} and the derivation of \eqref{equa} starting from \eqref{defXn}.

\begin{proof}[Proof of Lemma \ref{l1aa}] {\bf Step 1.} Averaged estimate of $\|f\|_{L^p({\cal C}_{j, T})}^p$\\
We establish here the identity
\beq
\label{ua2}
\int\limits_{Q_\ve}\|f\|_{L^p({\cal C}_{j, T})}^p\, dT= C(n, j)\, \ve^{j}\, \|f\|_{L^p(\R^n)}^p.
\eeq

Indeed, arguing as in the proof of \eqref{equb} and \eqref{equa} and with $X^j$ as in \eqref{defXj}, we have
\beq
\label{ua3}
\begin{aligned}
\int\limits_{Q_\ve}\|f\|_{L^p({\cal C}_{j, T})}^p\, dT&=
2^{j-n}\sum_{q\in \{ -1, 1\}^{n-j}}\ 
\sum_{\sigma\in S_{n-j, n}}\ve^j\int\limits_{|\omega|\le 1}d\omega\int\limits_{\R^n} dU 
  \,  |f(U+X^j)|^p.
\end{aligned}
\eeq
[The constant $2^{j-n}$ comes from the fact that  on the right-hand side of \eqref{ua3} the integral over a $j$-dimensional cube $C$ of ${\cal C}_{j, T}$ is counted $2^{n-j}$ times.]

In order to obtain \eqref{ua2}, it suffices to observe that the 
last integral in \eqref{ua3} does not depend on $\omega$. 

\medskip
\noindent
{\bf Step 2.} Averaged estimate of $|f|_{W^{s, p}({\cal C}_{j, T})}^p$\\
We have
\beq
\label{uc1}
\begin{aligned}
|f|_{W^{s,p}({\cal C}_{j, T})}^p=&\iint\limits_{\shortstack{$\scriptstyle {(X_\ast^j, Y_\ast^j)\in {\cal C}_{j, T}\times {\cal C}_{j, T}} $ \\ $\scriptstyle {|X_\ast^j-Y_\ast^j|< 2\ve} $}} dX_\ast^j dY_\ast^j\, \frac{|f(X_\ast^j)-f(Y_\ast^j)|^p}{|X_\ast^j-Y_\ast^j|^{j+sp}}
=I_1(T)+I_2(T),
\end{aligned}
\eeq
where
\be*
I_1(T)=\iint\limits_{|X_\ast^j-Y_\ast^j|<\ve}\ldots,\ I_2(T)=\iint\limits_{\ve\le |X_\ast^j-Y_\ast^j|<2\ve}\ldots
\ee*

We first note that
\beq
\label{ucc}
I_2(T)\le C\int\limits_{{\cal C}_{j, T}}dX_\ast^j \int\limits_{\ve\le |X_\ast^j-Y_\ast^j|<2 \ve}dY_\ast^j\, \frac 1{|Y_\ast^j-X_\ast^j|^{j+sp}}\, |f(X_\ast^j)|^p= C(\ve) \int\limits_{{\cal C}_{j, T}}dX_\ast^j\, |f(X_\ast^j)|^p,
\eeq
since
\be*
\int\limits_{\ve\le |X_\ast^j-Y_\ast^j|< 2 \ve}dY_\ast^j\, \frac 1{|Y_\ast^j-X_\ast^j|^{j+sp}}= C(\ve)<\infty, \ \forall\, T,\ \forall\,  X_\ast^j.
\ee*

By \eqref{ucc} and Step 1, we have
\beq
\label{uc2}
\int\limits_{Q_\ve}  I_2(T)\, dT\le C(\ve)\, \|f\|_{L^p(\R^n)}^p.
\eeq

We next note that (with notation as in \eqref{defXj} and \eqref{equa})
\beq
\label{uc3}
I_1(T)\le C\, \ve^{2j}\sum^\circ \int\limits_{|\omega|\le 1}d\omega\int\limits_{|\lambda|\le 1}d\lambda\, \frac{|f(T+2\ve K+X^j)-f(T+2\ve K+2\ve L+Y^j)|^p}{|X^j-(2\ve L+Y^j)|^{j+sp}},
\eeq
where 
\be*
\sum^\circ=\sum_{\shortstack{$\scriptstyle {L\in\Z^n} $ \\ $\scriptstyle {|L|\le 1} $}}\, \sum_{K\in\Z^n}\, \sum_{q, r\in \{-1,1\}^{n-j}}\, \sum_{\sigma, \tau\in S_{n-j, n}}.
\ee*

Integrating \eqref{uc3}, we find that
\beq
\label{uc4}
\int\limits_{Q_\ve}I_1(T)\, dT\le C\, \ve^{2j} \sum_\circ\ \int\limits_{|\omega|\le 1}d\omega\int\limits_{|\lambda|\le 1}d\lambda\, \int_{\R^n}dU\, \frac{|f(U+X^j)-f(U+2\ve L+Y^j)|^p}{|X^j-(2\ve L+Y^j)|^{j+sp}}, 
\eeq
with 
\be*
\sum_\circ=\sum_{\shortstack{$\scriptstyle {L\in\Z^n} $ \\ $\scriptstyle {|L|\le 1} $}}\, \sum_{q, r\in \{-1,1\}^{n-j}}\, \sum_{\sigma, \tau\in S_{n-j, n}}.
\ee*
By \eqref{uc4} and estimate \eqref{yd1} in Remark \ref{uc5}, we have
\beq
\label{uc6}
\int\limits_{Q_\ve}I_1(T)\, dT\le C(\ve)|f|_{W^{s,p}(\R^n)}^p.
\eeq

We complete the proof of Lemma \ref{l1aa} using \eqref{uc1}, \eqref{uc2} and \eqref{uc6}.
\end{proof}

\begin{proof}[Proof of Lemma \ref{l1ab}] {\bf Step 1.} A dimensional reduction\\
Assume for the moment that we proved the following estimate (with $X_\ast^{n-1}$ the projection of $X_\ast^n$ onto ${\cal C}_{n-1, T,\ve}$):
\beq
\label{ue1}
I=\int\limits_{Q_\ve}dT\int\limits_{\R^n}dX_\ast^n\, \frac{|f(X_\ast^n)-f(X_\ast^{n-1})|^p}{|X_\ast^n-X_\ast^{n-1}|^{sp}}\le C(n,\ve)\iint\limits_{|X_\ast^n-Y_\ast^n|<\ve}dX_\ast^ndY_\ast^n\, \frac{|f(X_\ast^n)-f(Y_\ast^n)|^p}{|X_\ast^n-Y_\ast^n|^{n+sp}}.
\eeq

Then we claim that the conclusion of the lemma holds. Indeed, if $j\in\llbracket 1, n-1\rrbracket$ then \eqref{ue1} applied with $n=j$ and with $\R^n$ replaced by the intersection of ${\cal C}_{j, T}$ with the $j$-dimensional plane
\be*
\{ (x_1,\ldots, x_n);\, x_l=T_l+2\ve\, K_l, \forall\, l\in I\}, \ \text{with }\# I=n-j\text{ and }K_l\in\Z
\ee*
leads (after  the use of the Fubini theorem in the variables $T_l$ with $l\not\in I$ and summation in $I$) to 
\beq
\label{ue2}
\int\limits_{Q_\ve}dT\int\limits_{{\mathcal C}_{j, T}}
dX_\ast^j\, \frac{|f(X_\ast^j)-f(X_\ast^{j-1})|^p}{|X_\ast^j-X_\ast^{j-1}|^{sp}}\le C(n, j, \ve)\int\limits_{Q_\ve}dT\iint\limits_{\shortstack{$\scriptstyle {(X_\ast^j, Y_\ast^j)\in {\cal C}_{j, T}\times {\cal C}_{j, T}} $ \\ $\scriptstyle {|X_\ast^j-Y_\ast^j|<\ve} $}}dX_\ast^j dY_\ast^j\, \frac{|f(X_\ast^j)-f(Y_\ast^j)|^p}{|X_\ast^j-Y_\ast^j|^{j+sp}}.
\eeq

We then obtain the conclusion of Lemma \ref{l1ab} using \eqref{ue2} and Lemma \ref{l1aa}.

\medskip
\noindent
{\bf Step 2.} Proof of \eqref{ue1}\\
We follow Step 2 in the proof of Lemma \ref{lema1} in Section \ref{appa}. Following the calculation \eqref{ue6}, the left-hand side $I$ of \eqref{ue1} satisfies 
\beq
\label{ue8}
\begin{aligned}
I=\int\limits_{\R^n}dX\int\limits_{|Y|<\ve}dY\, \frac{(\ve-|Y|)^{n-1}}{|Y|^{n+sp-1}}\, |f(X)-f(X-Y)|^p.
\end{aligned}
\eeq

We obtain \eqref{ue1} by noting that 
\be*
\frac{(\ve-|Y|)^{n-1}}{|Y|^{n+sp-1}}\le C(\ve)\frac 1{|Y|^{n+sp}}   \ \text{if }|Y|<\ve.\qedhere
\ee*
\end{proof}

\section{Approximation of maps defined on good skeletons}
\l{appd}
\def\mc{\mathscr C}
Throughout the next two sections, we take $0<s<1$, $1\le p<\infty$, $j\in\llbracket 1, n-1\rrbracket$ and we use the same notation as in Sections \ref{appb} and \ref{appc}. We consider a fixed finite submesh $\mc$ of $\mc_n$ and a map $g:{\cal C}_j\cap \mc\to\R^m$. For such maps, we  define  the norm
\be*
\|g\|_{L^p}^p=\|g\|_{L^p(\mc_j\cap\mc)}^p=\int\limits_{\mc_j\cap\mc}dX_\ast^j\, |g(X_\ast^j)|^p
\ee*
and the semi-norm
\beq
\label{fgz1}
|g|_{W^{s,p}}^p=|g|_{W^{s,p}(\mc_j\cap\mc)}^p=\int\limits_{\mc_j\cap\mc}dX_\ast^j\ \int\limits_{\mc_j\cap\mc}dY_\ast^j\, \frac{|g(X_\ast^j)-g(Y_\ast^j)|^p}{|X_\ast^j-Y_\ast^j|^{j+sp}}.
\eeq

Note that, in line with definition of $W^{s,p}({\cal C}_j)$, we could have restricted the double integral in \eqref{fgz1} to $|X^j_\ast-Y^j_\ast|<2\ve$. However, ${\cal C}$ being fixed and bounded, the double integral over $|X^j_\ast-Y^j_\ast|\ge 2\ve$ is controlled by $\|g\|_{L^p}^p$.

With the natural definition, we also consider the space ${\cal W}^{s,p}_j={\cal W}^{s,p}_j(\mc_j\cap\mc)$.

In this section, we adapt to the fractional Sobolev case some approximation techniques of maps defined on skeletons devised by Hang and Lin    \cite[Section 3]{hl}. The main result  is the following 
\begin{lemm}
\label{ua1}
Let $0<s<1$, $1\le p<\infty$  and $j\in\N$ be such that $1\le j\le sp< n$.  Let $N$ be a compact manifold without boundary embedded in $\R^m$. Let $g\in {\mathcal W}^{s,p}_j({\mathcal C}_j\cap\mc ; N)$. Then there exists a sequence $\{g^k\}\subset \Lip ({\mathcal C}_j\cap\mc ; N)$ such that $g^k\to g$ in $W^{s,p}({\mathcal C}_j\cap \mc)$.
\end{lemm}

Two difficulties arise in the proof of Lemma \ref{ua1}. The first one is to show that $\R^m$-valued maps $g$ in $ {\mathcal W}^{s,p}_j$ can be approximated by Lipschitz maps. This is already a non trivial task. An additional difficulty occurs when $g$ is $N$-valued. In this case, we have to prove approximation with $N$-valued Lipschitz maps.

It will be convenient to start by reducing  Lemma \ref{ua1} to a slightly easier to prove statement.

\begin{lemm}
\label{uf1}
Let $0<s<1$, $1\le p<\infty$  and $j\in\N$ be such that $1\le j\le sp< n$.  Let $N$ be a compact manifold without boundary embedded in $\R^m$. Let $\delta>0$ be sufficiently small and define
\beq
\label{uh3}
M=\{ x\in\R^m;\, \dist (x, N)\le\delta\}.
\eeq

Let $g\in {\mathcal W}^{s,p}_j({\mathcal C}_j\cap {\cal C} ; N)$. Then there exists a sequence $\{G^k\}\subset \Lip ({\mathcal C}_j\cap {\cal C} ; M)$ such that $G^k\to g$ in $W^{s,p}({\mathcal C}_j\cap {\cal C})$.
\end{lemm}

\begin{proof}[Lemma \ref{uf1} implies Lemma \ref{ua1}]
Let $\Pi:M\to N$ denote the nearest point projection. Let $g^k=\Pi (G^k)$. We note that $g=\Pi(g)$, and that $g^k$ is clearly Lipschitz. In order to conclude, it suffices to invoke the continuity of the map
\be*
W^{s,p}({\mathcal C}_j\cap {\cal C} ; M)\ni G\mapsto \Pi (G)\in W^{s,p}({\mathcal C}_j\cap {\cal C} ; N).
\ee* 
This  is standard for maps in smooth domains; see e.g. \cite[Proof of (5.43), p. 56]{bbm1} for a slightly more general continuity result. The argument in \cite{bbm1} adapts readily to maps defined on ${\mathcal C}_j\cap {\cal C}$.
\end{proof}

We next turn our attention to the proof of Lemma \ref{uf1}. Since $\mc$ and $j$ are fixed, we will simplify the notation and omit \enquote{$\mc_j\cap\mc$} in the norms and function spaces. With no loss of generality, we may assume that $\ve=1$. For the convenience of the reader, we start by stating the  main technical ingredients required in the proof of Lemma \ref{uf1}. Before proceeding, let us define \enquote{a cube in ${\cal C}_\ell$} (or \enquote{an $\ell$-dimensional cube in ${\cal C}_\ell$}) by backward induction as follows. A cube in ${\cal C}_n$ is any cube of the mesh ${\cal C}_n$. A cube in ${\cal C}_{n-1}$ is any of the $2^n$ faces of a cube in ${\cal C}_n$. For $\ell\le n-2$, a cube in ${\cal C}_\ell$ is any of the $2^{\ell+1}$ faces of any cube in ${\cal C}_{\ell+1}$. 

Let $g:{\cal C}_j\cap{\cal C}\to\R^m$. 
For $\mathfrak C$ a $j$-dimensional cube in ${\cal C}_j\cap{\cal C}$, we let $0_{\mathfrak C}$ be its center. Clearly, if $X_\ast^j\in {\mathfrak C}$, then the projection $X_\ast^{j-1}$ of $X_\ast^j$ on ${\cal C}_{j-1}\cap{\cal C}$ is 
\be*
\d X_\ast^{j-1}=0_{\mathfrak C}+\frac{X_\ast^j-0_{\mathfrak C}}{|X_\ast^j-0_{\mathfrak C}|}.
\ee*

We now define a convenient approximation $g_\mu$ of $g$. 
For $0<\mu<1$ and $X^j\in {\mathfrak C}$, we set
\be*
g_\mu(X_\ast^j)=\begin{cases}
g(X_\ast^{j-1}),&\text{if }|X_\ast^j-0_{\mathfrak C}|\ge 1-\mu\\
\d g\left(0_{\mathfrak C}+\frac{X_\ast^j-0_{\mathfrak C}}{1-\mu}\right),&\text{if }|X_\ast^j-0_{\mathfrak C}|< 1-\mu
\end{cases}.
\ee*

This definition is inspired by the \enquote{filling a hole} technique of Brezis and Li \cite{brezisli}. See also \cite[Lemma 3.1]{hl} and, in the context of fractional spaces, \cite[Appendix D]{bm}. 

We  have the following result, whose proof is postponed to the end of this section.
\begin{lemm}
\label{ug1}
Let $g\in {\cal W}^{s,p}_j$. Then $g_\mu\to g$ in $W^{s,p}$ as $\mu\to 0$.
\end{lemm}
[Here, we do not require $j\le sp$.]

Let $\rho\in C^\infty_c(Q)$ (with $Q$ the unit cube in $\R^j$) be a standard mollifier and set 
\be*
\rho_t(x)=\frac 1{t^j}\, \rho(x/t),\  \forall\, t>0,\ \forall\, x\in\R^j. 
\ee*

 Fix some function $\eta\in C^\infty_c([0,1) ; [0,1])$. We let ${\mathfrak C}$ and $0_{\mathfrak C}$ be as above. Given $g:{\cal C}_j\cap{\cal C}\to\R^m$, we define, with a slight abuse of notation and after identifying the $j$-plane containing ${\mathfrak C}$ with $\R^j$, 
\beq
\label{uh1}
g\ast\rho_t(X_\ast^j)=\int_{\mathfrak C} dY_\ast^j\,  g(Y_\ast^j)\rho_t(X_\ast^j-Y_\ast^j)\ \text{ for }X_\ast^j\in {\mathfrak C}\text{ such that }|X_\ast^j-0_{\mathfrak C}|<1-t.
\eeq

We note that  for small $t$ the quantity 
\be*
g^t(X_\ast^j)=\eta(|X_\ast^j-0_{\mathfrak C}|)\, g\ast\rho_t(X_\ast^j)
\ee*
is well-defined in ${\cal C}_j\cap {\cal C}$.
We also let
\be*
g^0(X_\ast^j)=\eta(|X_\ast^j-0_{\mathfrak C}|)\, g(X_\ast^j).
\ee*

We now state a standard result  on the approximation by smoothing in fractional Sobolev spaces, whose straightforward proof is left to the reader. 
\begin{lemm}
\label{ug2}
Let $g\in W^{s,p}$.
Then $g^t\to g^0$ in $W^{s,p}$ as $t\to 0$.
\end{lemm}
[Here, we do not require $j\le sp$.]

\smallskip
We next present another  auxiliary result, which is a rather easy consequence of Lemma \ref{tw1} (which is fully proved below) and whose proof (granted Lemma \ref{tw1}) is left to the reader. Given $f:{\cal C}_{j-1}\cap{\cal C}$, we consider its homogeneous extension $g$ to ${\cal C}_{j}\cap{\cal C}$. Let $\eta$ be as above. We assume in addition that $\eta=1$ near the origin. This implies that the map 
\be*
{\cal C}_{j}\cap{\cal C}\ni X_\ast^j\mapsto h(X_\ast^j)=\left(1-\eta(|X_\ast^j-0_{\mathfrak C}|)\right)g(X_\ast^j)
\ee*
is well-defined in each point.
\begin{lemm}
\label{ug3}
The mapping $f\mapsto h$ is continuous from $W^{s,p}({\cal C}_{j-1}\cap{\cal C})$ into $W^{s,p}({\cal C}_{j}\cap{\cal C})$. 
\end{lemm}
[Here, we do not require $j\le sp$.]

The final auxiliary result is deeper, and was essentially observed by Schoen and Uhlenbeck \cite{su}. For the fractional version we present below, see \cite[Example 2, p. 210, and eqn (7), p. 206]{bn}. The argument in \cite{bn} (where maps are defined in domains) adapts readily to the case of maps defined on skeletons.
\begin{lemm}
\label{uh2}
Let $0<s<1$ and $1\le p<\infty$ be such that $sp<n$. Let $j\in\N$ be such that $1\le j\le sp$.
Let $g\in W^{s,p}({\cal C}_j\cap{\cal C} ; N)$.  Let $0<t<1$ and let $\delta>0$ be arbitrarily small (but fixed). Let $M$ be as in \eqref{uh3} and $g\ast\rho_t$ be as in \eqref{uh1}. Then, for sufficiently small $t$, we have 
\beq
\label{uh5}
g\ast\rho_t(X_\ast^j)\in M,\ \forall\, X_\ast^j\in{\cal C}_j\cap {\cal C}\text{ such that }\dist\, (X_\ast^j, {\cal C}_{j-1}\cap{\cal C})>t.
\eeq
\end{lemm}
[Here, we do require $j\le sp$.]

\begin{proof}[Proof of Lemma \ref{uf1} using Lemmas \ref{ug1}--\ref{uh2}] The proof relies on two ingredients: approximation of maps as in Lemma \ref{ug1} and induction on $j$. 

\medskip
\noindent
{\bf Step 1.} Proof of the lemma for $j=1$\\
By Lemma \ref{ug1}, it suffices to prove the lemma when $g$ is replaced by $g_\mu$. Since $j=1$ and thus ${\cal C}_0\cap{\cal C}$ is a finite collection of points, this simply means that we may assume that $g$ is constant near each point in ${\cal C}_0\cap{\cal C}$: there exists some $\mu>0$ such that 
\beq
\label{uj1}
g(X_\ast^1)=g(X_\ast^0)\text{ if  }\dist\, (X_\ast^1, {\cal C}_0\cap{\cal C})\le \mu. 
\eeq

Let now $\eta\in C^\infty ([0,1] ; [0,1])$ be such that 
\be*
\eta(x)=\begin{cases}
1,&\text{if }0\le x\le 1-\mu/2\\
0,&\text{if }x>1-\mu/3
\end{cases}.
\ee*

When $0<t<\mu/3$, the map 
\be*
X_\ast^1\mapsto G^t(X_\ast^1)=\eta(|X_\ast^1-0_{\mathfrak C}|)\, g\ast\rho_t(X_\ast^1)+\left( 1-\eta(|X_\ast^1-0_{\mathfrak C}|) \right)\, g(X_\ast^0)
\ee*
is well-defined everywhere, and is clearly Lipschitz. Moreover, by Lemma \ref{ug2} and the choice of $\eta$, we have 
\be*
G^t\to g\text{ in }W^{s,p}\text{ as }t\to 0.
\ee*

It remains to prove that, for small $t$, we have 
\beq
\label{ui1}
G^t(X_\ast^1)\in M,\ \forall\, X_\ast^1\in {\cal C}_1\cap{\cal C}.
\eeq

By Lemma \ref{uh2}, property \eqref{ui1} holds when $|X_\ast^1-0_{\mathfrak C}|\le 1-\mu/2$. Clearly,  \eqref{ui1} holds also when $|X_\ast^1-0_{\mathfrak C}|\ge 1-\mu/3$. Finally, when $1-\mu/2<|X_\ast^1-0_{\mathfrak C}|< 1-\mu/3$ and $t<\mu/3$, we have
\be*
G^t(X_\ast^1)=g\ast\rho_t(X_\ast^1)=g(X_\ast^0)\in N.
\ee*

\medskip
\noindent
{\bf Step 2.} Proof of the lemma for $j\ge 2$\\
Let  $f$ be the restriction of $g$ to ${\cal C}_{j-1}\cap{\cal C}$. By Lemma \ref{ug1}, we may assume that there exists some $\mu\in (0,1)$ such that 
\beq
\label{uj3}
g(X_\ast^j)=f(X_\ast^{j-1}),\ \forall\, {\mathfrak C}\subset {\cal C}_j\cap{\cal C}, \forall\, X^j_\ast\in {\mathfrak C}\text{ such that }|X_\ast^j-0_{\mathfrak C}|>1-\mu.
\eeq

We argue by induction on $j$. By the induction hypothesis and the reduction of Lemma \ref{ua1} to Lemma \ref{uf1},  the map $f$  (which clearly belongs to ${\cal W}^{s,p}_{j-1}({\cal C}_{j-1}\cap{\cal C} ; N)$) is the limit in $W^{s,p}$ of a sequence $\{ F^k\}\subset\Lip ({\cal C}_{j-1}\cap{\cal C} ; N)$. With $\eta$ as in Step 1 and $0<t<\mu/3$, we define
 the Lipschitz maps 
\be*
X_\ast^j\mapsto G^{k,t}(X_\ast^j)=\eta(|X_\ast^j-0_{\mathfrak C}|)\, g\ast\rho_t(X_\ast^j)+\left( 1-\eta(|X_\ast^j-0_{\mathfrak C}|) \right)\, F^k(X_\ast^{j-1}).
\ee*

By Lemmas \ref{ug2} and \ref{ug3}, we have 
\be*
\lim_{k\to\infty}\lim_{t\searrow 0}G^{k,t}=g\text{ in }W^{s,p}.
\ee*

In order to complete Step 2 it remains to prove that, for large $k$ and sufficiently small $t$ (possibly depending on $k$) we have 
\beq
\label{uk1}
G^{k, t}(X_\ast^j)\in M,\ \forall\, X_\ast^j\in {\cal C}_j\cap{\cal C}.
\eeq

As in Step 1, \eqref{uk1} holds when $|X_\ast^j-0_{\mathfrak C}|\le 1-\mu/2$ or $|X_\ast^j-0_{\mathfrak C}|\ge 1-\mu/3$. When $1-\mu/2<|X_\ast^j-0_{\mathfrak C}|< 1-\mu/3$, we argue as follows. Since $j\le sp$, we have $j-1<sp$. By the Sobolev embeddings, $f$ and $F^k$ are continuous and we have $F^k\to f$ uniformly. Let $k_0$ be such that 
\beq
\label{ul1}
\|F^k-f\|_{L^\infty}\le \delta/2,\ \forall\, k\ge k_0.
\eeq 

By \eqref{uj3} and the continuity of $f$, for every fixed $k$ we have
\beq
\label{um1}
\begin{aligned}
&\lim_{t\searrow 0}G^{k,t}(X_\ast^j)=\eta(|X_\ast^j-0_{\mathfrak C}|)\, f(X_\ast^{j-1})+\left( 1-\eta(|X_\ast^j-0_{\mathfrak C}|) \right)\, F^k(X_\ast^{j-1})\\
&\text{ uniformly in the set }\bigcup_{{\mathfrak C}\subset {\cal C}_j\cap{\cal C}}\{ X_\ast^j\in{\mathfrak C};\, 1-\mu/2<|X_\ast^j-0_{\mathfrak C}|< 1-\mu/3\}.
\end{aligned}
\eeq

We complete the proof of \eqref{uk1} using \eqref{ul1} and \eqref{um1}.
\end{proof}

In order to complete the proof of Lemma \ref{uf1}, it remains to proceed to the 
\begin{proof}[Proof of Lemma \ref{ug1}]
We may assume that $\ve=1$ and that $T=0$. We set 
\beq
\label{tu9}
\Omega={\cal C}_j\cap{\cal C},\ E=E_\mu=\{ X^j_\ast\in\Omega;\, |X^j_\ast-X^{j-1}_\ast|>\mu\}, \ F=F_\mu=\{ X^j_\ast\in\Omega;\, |X^j_\ast-X^{j-1}_\ast|<\mu\}.
\eeq

If ${\mathfrak C}$ is a cube in ${\cal C}_j$ and $0<\mu<\mu_0<1$, then we define
\beq
\label{tu2}
{\mathfrak C}_\mu=\{ X^j_\ast\in {\mathfrak C};\, |X^j_\ast-X^{j-1}_\ast|>\mu\},\  {\mathfrak C}_{\mu, \mu_0}=\{ X^j_\ast\in {\mathfrak C};\, \mu<|X^j_\ast-X^{j-1}_\ast|<\mu_0\},\ {\mathfrak C}^c_\mu={\mathfrak C}\setminus {\mathfrak C}_\mu.
\eeq

If ${\mathfrak C}'$ is another cube in ${\cal C}_j$, we define similarly ${\mathfrak C}'_{\mu}$, etc.

We clearly have 
\beq
\label{ha10}
g_\mu\to g\text{ in }L^p\text{ as }\mu\to 0.
\eeq
[For a more general property, see \eqref{tu1} below.]

 It thus remains to prove that 
\beq
\label{ha1}
I=\iint\limits_{\Omega\times\Omega} dX^j_\ast dY^j_\ast\, \frac{|[g_\mu(X^j_\ast)-g(X^j_\ast)]-[g_\mu(Y^j_\ast)-g(Y^j_\ast)]|^p}{|X^j_\ast-Y^j_\ast|^{j+sp}}\to 0\text{ as }\mu\to 0.
\eeq

We split 
\be*
I=I_{E,\, E}+2I_{E,\, F}+I_{F,\, F},\ \text{where }I_{A,\, B}=I_{A,\, B,\, \mu}=\iint\limits_{A\times B}\cdots
\ee*

We have to prove that $I_{E,\, E}\to 0$, $I_{E,\, F}\to 0$ and $I_{F,\, F}\to 0$ as $\mu\to 0$.

\medskip
\noindent
{\bf Step 1.} For every cube ${\mathfrak C}$ in ${\cal C}_j\cap{\cal C}$ we have $I_{{\mathfrak C},\, {\mathfrak C}}\to 0$ as $\mu\to 0$\\
Indeed, we may assume that ${\mathfrak C}$ is open, and then we identify ${\mathfrak C}$ with the unit cube $Q_1\subset\R^j$. We let ${\mathfrak C}^*=Q_2$ denote the double of ${\mathfrak C}$, and set 
\be*
h(X^j_\ast)=\begin{cases}
g(X^j_\ast),&\text{if }X^j_\ast\in {\mathfrak C}\\
g(X^{j-1}_\ast),&\text{if }X^j_\ast\in {\mathfrak C}^\ast\setminus {\mathfrak C}
\end{cases}.
\ee*

\begin{lemm}
\label{tt1}
We have $h\in W^{s,p}({\mathfrak C}^\ast)$.
\end{lemm}
\begin{proof}[Proof of Lemma \ref{tt1}]
Clearly, since $g\in {\cal W}^{s,p}_j$, we have $h\in W^{s,p}({\mathfrak C})$ and $h\in W^{s,p}({\mathfrak C}^\ast\setminus \overline {\mathfrak C})$. It thus suffices to prove that $h\in W^{s,p}$ near each point of $\p {\mathfrak C}$. After a bi-Lipschitz change of variables, and taking the definition of ${\cal W}^{s,p}_j$ into account, we are then reduced to the following lemma, established in \cite[Appendix B, Lemma B.1]{bm}.
\end{proof}
\begin{lemm}
\label{tt2}
Let $0<s<1$, $1\le p<\infty$, $u\in W^{s,p}((0,1)^j)$ and $v\in W^{s,p}((0,1)^{j-1})$ be such that
\be*
\int\limits_{(0,1)^j}dX_1\ldots dX_j\, \frac{|u(X_1,\ldots, X_j)-v(X_1,\ldots, X_{j-1})|^p}{X_j^{sp}}<\infty.
\ee* 
Define
\be*
w(X_1,\ldots, X_j)=\begin{cases}
u(X_1,\ldots, X_j),&\text{if }(X_1,\ldots, X_j)\in (0,1)^j\\
v(X_1,\ldots, X_{j-1}),&\text{if }(X_1,\ldots, X_j)\in (0,1)^{j-1}\times (-1, 0]
\end{cases}.
\ee*

Then $w\in W^{s,p}((0,1)^{j-1}\times (-1,1))$.
\end{lemm}

\noindent
[In the statement of Lemma B.1 in \cite{bm} it is assumed that $1<p<\infty$, but the argument there still holds for $p=1$.]

\medskip
\noindent
{\bf Step 1 completed.} We may extend $h$ to a map, still denoted $h$,  in $W^{s,p}(\R^j)$. Define 
\beq
\label{tt7}
h^t(X^j_\ast)=h(X^j_\ast/t),\ \forall\, X^j_\ast\in\R^j,\ \forall\, t>0.
\eeq

Note that
\beq
\label{tu1}
\text{if }h\in W^{\sigma, p}\text{ for some }\sigma\ge 0,\text{ then the mapping }t\mapsto h^t\in W^{\sigma, p}\text{ defined by \eqref{tt7} is continuous}.
\eeq

Using \eqref{tu1} with $\sigma = s$, we obtain that
\be*
I_{{\mathfrak C}, {\mathfrak C}}=\iint\limits_{{\mathfrak C}\times {\mathfrak C}}dX^j_\ast dY^j_\ast\, \frac{\left|\left[h^{1-\mu}(X^j_\ast)-h^1(X^j_\ast)\right]-\left[h^{1-\mu}(Y^j_\ast)-h^1(Y^j_\ast)\right]\right|^p}{|X^j_\ast-Y^j_\ast|^{j+sp}}\le \left|h^{1-\mu}-h^1\right|_{W^{s,p}(\R^j)}^p\to 0\text{ as }\mu\to 0.
\ee*

\medskip
\noindent
{\bf Step 2.} For every cube ${\mathfrak C}$ in ${\cal C}_j\cap{\cal C}$ and for every fixed $\mu_0\in (0,1)$ we have $I_{{\mathfrak C},\, E_{\mu_0}}\to 0$ as $\mu\to 0$\\
Indeed, by Step 1 it suffices to prove that for every cube ${\mathfrak C}'\neq {\mathfrak C}$ in ${\cal C}_j$ we have, with ${\mathfrak C}_{\mu_0}$ as in \eqref{tu2},  
\beq
\label{tt3}
I_{{\mathfrak C},\, {\mathfrak C}'_{\mu_0}}\to 0\text{ as }\mu\to 0.
\eeq

We note that
\beq
\label{tt4}
|X^j_\ast-Y^j_\ast|\ge \mu_0,\ \forall\, X^j_\ast\in {\mathfrak C},\ \forall\, Y^j_\ast\in {\mathfrak C}'_{\mu_0}. 
\eeq

By \eqref{tt4}, we have
\beq
\label{tt5}
I_{{\mathfrak C},\, {\mathfrak C}'_{\mu_0}}\le C(n, p, \mu_0)\left(\int\limits_{\mathfrak C} dX^j_\ast\, |g_\mu(X^j_\ast)-g(X^j_\ast)|^p+\int\limits_{{\mathfrak C}'_{\mu_0}}dY^j_\ast\, |g_\mu(Y^j_\ast)-g(Y^j_\ast)|^p  \right).
\eeq

We obtain \eqref{tt3} using \eqref{tt5} and \eqref{ha10}. 

\medskip
\noindent
{\bf Step 3.} We have $I_{E,\, E}\to 0$ as $\mu\to 0$\\
By Steps 1 and 2, Step 3 amounts to the following. Let $\xi>0$ be fixed arbitrarily small. Let ${\mathfrak C}\neq {\mathfrak C}'$ be two cubes in ${\cal C}_j$. Then there exists some $0<\mu_0<1$ such that 
\beq
\label{tu22}
I_{{\mathfrak C}_{\mu, \mu_0},\, {\mathfrak C}'_{\mu, \mu_0}}<\xi\ \text{for every }0<\mu<\mu_0.
\eeq

In order to establish \eqref{tu22}, we start from
\beq
\label{tu3}
I_{A, B}\le 2^{p-1}\left(\ \iint\limits_{A\times B}dX^j_\ast dY^j_\ast\, \frac{|g_\mu(X^j_\ast)-g_\mu(Y^j_\ast)|^p}{|X^j_\ast-Y^j_\ast|^{j+sp}}+\iint\limits_{A\times B}dX^j_\ast dY^j_\ast\, \frac{|g(X^j_\ast)-g(Y^j_\ast)|^p}{|X^j_\ast-Y^j_\ast|^{j+sp}}\right).
\eeq

We next establish the following estimate.
\begin{lemm}
\label{tu4}
Let ${\mathfrak C}\neq{\mathfrak C}'$ be two cubes in ${\cal C}_j\cap{\cal C}$. Then, for $0<\mu<1/2$ and for $X^j_\ast\in {\mathfrak C}$ and $Y^j_\ast\in {\mathfrak C}'$ such that
\beq
\label{tv1}
|X_\ast^j-0_{\mathfrak C}|< 1-\mu\text{ and }
 |Y_\ast^j-0_{{\mathfrak C}'}|< 1-\mu, 
\eeq
we have
\beq
\label{tu5}
\left|\left[0_{\mathfrak C}+\frac{X_\ast^j-0_{\mathfrak C}}{1-\mu}\right]- \left[0_{{\mathfrak C}'}+\frac{Y_\ast^j-0_{{\mathfrak C}'}}{1-\mu}\right]\right|\le C\,|X_\ast^j-Y_\ast^j|.
\eeq
\end{lemm}
\begin{proof}[Proof of Lemma \ref{tu4}] Recall that we assume that $\ve=1$ and $T=0$. Write
\beq
\label{ha11}
0_{\mathfrak C}=(C_1,\ldots, C_n),\ 0_{{\mathfrak C}'}=(C'_1,\ldots, C'_n).
\eeq
 
 One may check the following properties of the $C_i$'s:\\
 a) Each $C_i$ is an integer. \\
b) Exactly $n-j$ $C_i$'s are odd. \\
c) The open cube ${\mathfrak C}$ is given by the following system of equations and inequalities:
\be*
X_i=C_i,\ \text{if }C_i\text{ is odd},\ |X_i-C_i|<1,\ \text{if }C_i\text{ is even}.
\ee*
d) Thus every point $X^j_\ast$ as in \eqref{tv1} is of the form 
\be*
X^j_\ast=(C_1+x_1,\ldots, C_n+x_n),\text{ with }x_i=0\text{ if }C_i\text{ is odd and }|x_i|<1-\mu\text{ if }C_i\text{ is even}.
\ee*

Similarly if we write $Y^j_\ast=(C'_1+y_1,\ldots, C'_n+y_n)$.

\smallskip
Estimate \eqref{tu5} will follow from the next estimate, valid for each coordinate:
\beq
\label{tv2}
\left|\left[C_i+\frac{x_i}{1-\mu}\right]-\left[C'_i+\frac{y_i}{1-\mu}\right]  \right|\le C\, |[C_i+x_i]-[C'_i+y_i]|\text{ if }0<\mu<1/2,\ |x_i|<1-\mu,\ |y_i|<1-\mu.
\eeq

In order to establish the validity of \eqref{tv2}, we consider the following cases.

\medskip
\noindent
{\bf Case 1.} $|C_i-C'_i|\ge 3$\\
Then we have
\be*
|[C_i+x_i]-[C'_i+y_i]|\ge |C_i-C'_i|-2\text{ and }\left|\left[C_i+\frac{x_i}{1-\mu}\right]-\left[C'_i+\frac{y_i}{1-\mu}\right]  \right|\le |C_i-C'_i|+2,  
\ee*
and thus \eqref{tv2} holds with $C=5$.

\medskip
\noindent
{\bf Case 2.} $|C_i-C'_i|=2$ and $C_i$ is odd\\
Then $x_i=y_i=0$ and thus \eqref{tv2} holds with $C=1$.

\smallskip
The same argument applies to the next case.

\medskip
\noindent
{\bf Case 3.} $C_i=C'_i$ and $C_i$ is odd

\medskip
\noindent
{\bf Case 4.} $|C_i-C'_i|=2$ and $C_i$ is even\\
We may assume that $C_i=2$, $C'_i=0$, and we have to prove that
\be*
\left|2+\frac{x-y}{1-\mu}\right|\le C\, |2+(x-y)|\text{ when }|x|<1-\mu\text{ and }|y|<1-\mu,
\ee*
which amounts to 
\be*
2+\frac{x-y}{1-\mu}\le C\, [2+(x-y)]\text{ when }|x|<1-\mu\text{ and }|y|<1-\mu.
\ee*

The above inequality holds with $C=2$ (provided $0<\mu<1/2$). 

\medskip
\noindent
{\bf Case 5.} $C_i=C'_i$ and $C_i$ is even\\
Then \eqref{tv2} holds with $C=2$ (provided $0<\mu<1/2$). 

\medskip
\noindent
{\bf Case 6.} $|C_i-C'_i|=1$\\
We may assume that $C_i=1$ and $C'_i=0$. As above, for $0<\mu<1/2$ estimate \eqref{tv2} follows from \be*
\left|1-\frac{y}{1-\mu}\right|\le 2\, |1-y|\text{ when }|y|<1-\mu.\qedhere
\ee*
\end{proof}

\medskip
\noindent
{\bf Step 3 completed.} We estimate $I_{{\mathfrak C}_{\mu, \mu_0},\, {\mathfrak C}'_{\mu, \mu_0}}$ using  \eqref{tu3} with $A={\mathfrak C}_{\mu, \mu_0}$  and $B={\mathfrak C}'_{\mu, \mu_0}$. After the changes of variables 
\beq
\label{td1}
{\mathfrak C}_{\mu, \mu_0}\ni X^j_\ast\mapsto 0_{\mathfrak C}+\frac{X_\ast^j-0_{\mathfrak C}}{1-\mu}\in {\mathfrak C}^c_{(\mu_0-\mu)/(1-\mu)} ,\ {\mathfrak C}'_{\mu, \mu_0}\ni Y^j_\ast\mapsto 0_{{\mathfrak C}'}+\frac{Y_\ast^j-0_{{\mathfrak C}'}}{1-\mu}\in {\mathfrak C}_{(\mu_0-\mu)/(1-\mu)}^{\prime\, c}
\eeq
in the first double integral in \eqref{tu3}, Lemma \ref{tu4} implies that for $0<\mu_0<1/2$ we have
\beq
\label{tu55}
\begin{aligned}
I_{{\mathfrak C}_{\mu, \mu_0},\, {\mathfrak C}'_{\mu, \mu_0}}\le & C(n, p)\iint\limits_{{\mathfrak C}^c_{(\mu_0-\mu)/(1-\mu)}\times {\mathfrak C}^{\prime\, c}_{(\mu_0-\mu)/(1-\mu)}}dX^j_\ast dY^j_\ast\, \frac{|g(X^j_\ast)-g(Y^j_\ast)|^p}{|X^j_\ast-Y^j_\ast|^{j+sp}}\\
&+C(n, p)\iint\limits_{{\mathfrak C}_{\mu, \mu_0}\times {\mathfrak C}'_{\mu, \mu_0}}dX^j_\ast dY^j_\ast\, \frac{|g(X^j_\ast)-g(Y^j_\ast)|^p}{|X^j_\ast-Y^j_\ast|^{j+sp}} \\
\le & C(n, p) \iint\limits_{{\mathfrak C}^c_{2\, \mu_0}\times {\mathfrak C}^{\prime\, c}_{2\, \mu_0}}dX^j_\ast dY^j_\ast\, \frac{|g(X^j_\ast)-g(Y^j_\ast)|^p}{|X^j_\ast-Y^j_\ast|^{j+sp}}.
\end{aligned}
\eeq

We complete Step 3 by noting that the last double integral in \eqref{tu55} goes to $0$ as $\mu_0\to 0$ (since $g\in W^{s,p}({\cal C}_j\cap {\cal C}$)).

\medskip
\noindent
{\bf Step 4.} We have $I_{F,\, F}\to 0$ as $\mu\to 0$\\
In view of Step 1, Step 4 is an immediate consequence of the fact that the restriction of $g$ to ${\cal C}_{j-1}\cap {\cal C}$ belongs to $W^{s,p}$ and of the following
\begin{lemm}
\label{tw1} 
Let $0<s<1$, $1\le p<\infty$, $j\ge 1$ and $h\in W^{s,p}({\cal C}_{j-1}\cap {\cal C})$. Then, with $C=C(j, s, p, {\cal C})$, we have 
\beq
\label{tw22}
\sum_{{\mathfrak C}\neq{\mathfrak C}'}\ \iint\limits_{{\mathfrak C}^c_{1/2}\times {\mathfrak C}^{\prime\, c}_{1/2}}dX^j_\ast dY^j_\ast\, \frac{|h(X^{j-1}_\ast)-h(Y^{j-1}_\ast)|^p}{|X^j_\ast-Y^j_\ast|^{j+sp}}\le C\, |h|_{W^{s,p}({\cal C}_{j-1}\cap{\cal C})}^p.
\eeq

\end{lemm}
\begin{proof}[Proof of Lemma \ref{tw1}]
Estimate \eqref{tw22} is a special case of the following more general inequality, valid for nonegative measurable $f$: 
\beq
\label{tw2}
\iint\limits_{{\mathfrak C}^c_{1/2}\times {\mathfrak C}^{\prime\, c}_{1/2}}dX^j_\ast dY^j_\ast\, \frac{f(X^{j-1}_\ast, Y^{j-1}_\ast)}{|X^j_\ast-Y^j_\ast|^{j+sp}}\le C\, \iint\limits_{(\overline{\mathfrak C}\cap {\cal C}_{j-1})\times (\overline{{\mathfrak C}'}\cap{\cal C}_{j-1})}dX^{j-1}_\ast dY^{j-1}_\ast\, \frac{f(X^{j-1}_\ast, Y^{j-1}_\ast)}{|X^{j-1}_\ast-Y^{j-1}_\ast|^{j-1+sp}}.
\eeq

If we express the left-hand side of \eqref{tw2} using polar coordinates on $\overline{\mathfrak C}\cap {\cal C}_{j-1}$ (respectively on $\overline{{\mathfrak C}'}\cap {\cal C}_{j-1}$), then \eqref{tw2} amounts to the following
\beq
\label{tw3}
\int\limits_{1/2}^1\int\limits_{1/2}^1 dt d\tau\, \frac 1{|[(1-t)\, 0_{{\mathfrak C}}+t\, X^{j-1}_\ast]-[(1-\tau)\, 0_{{\mathfrak C}'}+\tau\, Y^{j-1}_\ast]|^{a}}\le C\, \frac{1}{|X^{j-1}_\ast-Y^{j-1}_\ast|^{a-1}},
\eeq
which  is valid whenever $a>1$, ${\mathfrak C}\neq{\mathfrak C}'$, $X^{j-1}_\ast\in \overline{\mathfrak C}\cap {\cal C}_{j-1}$ and $Y^{j-1}_\ast\in \overline{{\mathfrak C}'}\cap {\cal C}_{j-1}$.

Clearly, estimate \eqref{tw3} holds when $\overline{\mathfrak C}\cap \overline{{\mathfrak C}'}=\emptyset$ (since both sides of \eqref{tw3} are bounded from above and below by finite positive constants). 

We may thus assume that 
\beq
\label{tw4}
\overline{\mathfrak C}\cap \overline{{\mathfrak C}'}\neq\emptyset.
\eeq

 In this case, the idea is to mimic the proof of the estimate \eqref{a13}.

\medskip
\noindent
{\bf Step 1 in the proof of \eqref{tw3}.} We claim that, assuming \eqref{tw4}, there exists some $C=C(n, j)$ such that for $X^{j-1}_\ast\in \overline{\mathfrak C}\cap {\cal C}_{j-1}$ and $Y^{j-1}_\ast\in \overline{{\mathfrak C}'}\cap {\cal C}_{j-1}$ and $1/2\le t, \tau \le 1$ we have
\beq
\label{tw5}
|[(1-t)\, 0_{{\mathfrak C}}+t\, X^{j-1}_\ast]-[(1-\tau)\, 0_{{\mathfrak C}'}+\tau\, Y^{j-1}_\ast]|\ge C\, |[(1-t)\, 0_{{\mathfrak C}}+t\, X^{j-1}_\ast]-Y^{j-1}_\ast|.
\eeq

The proof of \eqref{tw5} relies on the following geometrically clear inequality, whose proof is postponed.

\begin{lemm}
\label{tz1}
Assume that \eqref{tw4} holds. Then there exists some $C=C(n, j)$ such that if $X^j_\ast\in \overline{\mathfrak C}$ and $Y^j_\ast\in\overline{{\mathfrak C}'}$, then there exists some $Z^j_\ast\in \overline{\mathfrak C}\cap {{\mathfrak C}'}$ such that
\beq
\label{tz2}
|X^j_\ast-Z^j_\ast|+|Y^j_\ast-Z^j_\ast|\le C\, |X^j_\ast-Y^j_\ast|.
\eeq

Equivalently, if $P:\overline{\mathfrak C}\cup \overline{{\mathfrak C}'}\to\R^\ell$ is L-Lipschitz on $\overline{\mathfrak C}$ and on $\overline{{\mathfrak C}'}$, then $P$ is $CL$-Lipschitz on $\overline{\mathfrak C}\cup \overline{{\mathfrak C}'}$.
\end{lemm}

Assuming Lemma \ref{tz1} established, we proceed as in the proof of Lemma \ref{lemb3}: we let
\be*
P(X^j_\ast)=X^j_\ast, \ \forall\, X^j_\ast\in \overline{\mathfrak C}, P(Y^j_\ast)=Y^{j-1}_\ast,\ \forall\, Y^j_\ast\in \overline{{\mathfrak C}'}^c_{1/2}.
\ee*
We extend $P$ from ${\overline{{\mathfrak C}^{\prime}}}^c_{1/2}$ to $\overline{{\mathfrak C}'}$ without increasing its Lipschitz constant (which is independent of ${\mathfrak C}'$). For this $P$, estimate \eqref{tw5} reads 
\be*
|P(X^j_\ast)-P(Y^j_\ast)|\le \frac 1C\, |X^j_\ast-Y^j_\ast|,\ \forall\, X^j_\ast\in {\mathfrak C}^c_{1/2},\ \forall\, Y^j_\ast\in {\mathfrak C}^{\prime\, c}_{1/2},
\ee*
which follows from Lemma \ref{tz1}.

\medskip
\noindent
{\bf Step 2 in the proof of \eqref{tw3}.} In view of \eqref{tw5}, we have reduced \eqref{tw3} to 
\beq
\label{tw6}
\int\limits_{1/2}^1 dt\, \frac 1{|[(1-t)\, 0_{{\mathfrak C}}+t\, X^{j-1}_\ast]- Y^{j-1}_\ast|^{a}}\le C\, \frac{1}{|X^{j-1}_\ast-Y^{j-1}_\ast|^{a-1}}.
\eeq

Combining \eqref{tw4} with the fact that ${\mathfrak C}\cap{\mathfrak C}'=\emptyset$, we find that 
\beq
\label{tw7}
1\le |Y^{j-1}_\ast-0_{\mathfrak C}|\le 3 \text{ and }|X^{j-1}_\ast-Y^{j-1}_\ast|\le 24.
\eeq

Using \eqref{tw7}, we obtain that
\beq
\label{tw8}
\begin{aligned}
|[(1-t)\, 0_{{\mathfrak C}}+t\, X^{j-1}_\ast]- Y^{j-1}_\ast|&=|[0_{{\mathfrak C}}-Y^{j-1}_\ast]-t[0_{{\mathfrak C}}-X^{j-1}_\ast]|\\
&\ge 1-t\text{ when }1/2\le t\le 1-|X^{j-1}_\ast-Y^{j-1}_\ast|/100
\end{aligned}
\eeq
and
\beq
\label{tw9}
\begin{aligned}
|[(1-t)\, 0_{{\mathfrak C}}+t\, X^{j-1}_\ast]- Y^{j-1}_\ast|&=|t[X^{j-1}_\ast-Y^{j-1}_\ast]+(1-t)[0_{{\mathfrak C}}-Y^{j-1}_\ast]|\\
&\ge t|X^{j-1}_\ast-Y^{j-1}_\ast|-3(1-t)\\
&\ge C|X^{j-1}_\ast-Y^{j-1}_\ast|\text{ when }1-|X^{j-1}_\ast-Y^{j-1}_\ast|/100< t\le 1.
\end{aligned}
\eeq

Estimate \eqref{tw6} follows from \eqref{tw8} and \eqref{tw9}. \end{proof}

In order to complete Step 4, it remains to proceed to the 

\begin{proof}[Proof of Lemma \ref{tz1}]
Let ${\mathfrak E}=\overline{\mathfrak C}\cap  \overline{{\mathfrak C}'}$, and let $\ell$ be the Hausdorff dimension of ${\mathfrak E}$. Let us note that ${\mathfrak E}$ is a cube in ${\cal C}_\ell$. After translation and permutation of the coordinates, we may identify ${\mathfrak E}$ with a cube in $\R^\ell$, and then we may write 
\be*
\overline{{\mathfrak C}}={\mathfrak D}\times {\mathfrak E},\ \overline{{\mathfrak C}'}={\mathfrak D}'\times {\mathfrak E}
\ee*
with ${\mathfrak D}$, ${\mathfrak D}'$ closed cubes in ${\cal C}_{j-\ell}(\R^{n-\ell})$ such that 
\beq
\label{fha1}
{\mathfrak D}\cap{\mathfrak D}'=\{Z'\}\text{ for some point }Z' \text{ (which has to be a vertex of both }{\mathfrak D}\text{ and }{\mathfrak D}'\text{)}.  
\eeq

We will split a point $X^j_\ast\in \overline{{\mathfrak C}}$ as 
\be*
X^j_\ast=(X', X''),\text{ with }X'\in {\mathfrak D}, \ X''\in {\mathfrak E};\text{ similarly for a point }Y^j_\ast\in \overline{{\mathfrak C}'}.
\ee*

Assume that we have established the estimate
\beq
\label{tz3}
|X'-Z'|+|Y'-Z'|\le C|X'-Y'|,\ \forall\, X'\in {\mathfrak D},\ \forall\, Y'\in {\mathfrak D}'.
\eeq

Then clearly $(Z', X'')\in \overline{\mathfrak C}\cap  \overline{{\mathfrak C}'}$ and
\be*
|X^j_\ast-(Z', X'')|+|Y^j_\ast-(Z', X'')|\le C|X^j_\ast-Y^j_\ast|,
\ee*
i.e. \eqref{tz2} holds.

It thus remains to prove \eqref{tz3}. This is obtained by contradiction. Assume that there are sequences $\{ X^{', k}_\ast\}\subset {\mathfrak D}\setminus \{Z'\}$ and $\{ Y^{', k}_\ast\}\subset {\mathfrak D}'\setminus \{Z'\}$ such that 
\beq
\label{tz4}
|X^{', k}-Z'|+|Y^{', k}-Z'|\ge k|X^{', k}-Y^{', k}|.
\eeq

By symmetry and after passing to a subsequence, we may assume that 
\be*
X^{', k}-Z'=\delta_k W^k, \ Y^{', k}-Z'=\lambda_k T^k,\ |W^k|=1,\ W^k\to W,\ |T^k|=1,\ T^k\to T,\ 0\le\lambda_k\le\delta_k,\ \lambda_k/\delta_k\to \mu.
\ee*

Using \eqref{tz4}, we obtain that $W=T$ (and $\mu=1$). However, this cannot happen. Indeed, since $X^{', k}\in {\mathfrak D}$, we have $Z'+W^k\in {\mathfrak D}$ (check it on a picture using \eqref{fha1}). Thus $Z'+W\in {\mathfrak D}$. Similarly, $Z'+T\in {\mathfrak D}'$. Since $W=T$, we obtain that ${\mathfrak D}\cap {\mathfrak D}'$ contains $Z'+W$, a contradiction.
\end{proof}

Step 4 is complete.

\medskip
\noindent
{\bf Step 5.} We have $I_{E,\, F}\to 0$ as $\mu\to 0$\\

In view of Steps 1 and 2, it suffices to establish the following. Let $\xi>0$ be fixed arbitrarily small. Let ${\mathfrak C}\neq {\mathfrak C}'$ be two cubes in ${\cal C}_j$. Then there exists some $0<\mu_0<1$ such that 
\beq
\label{tz22}
I_{{\mathfrak C}_{\mu, \mu_0},\, {\mathfrak C}^{\prime\, c}_{\mu}}=\iint\limits_{{\mathfrak C}_{\mu, \mu_0}\times {\mathfrak C}^{\prime\, c}_{\mu}}dX^j_\ast dY^j_\ast\, \frac{|g_\mu(X^{j}_\ast)-g(Y^{j-1}_\ast)|^p}{|X^j_\ast-Y^j_\ast|^{j+sp}}<2\xi\ \text{for every }0<\mu<\mu_0.
\eeq

By Step 4,   we have 
\beq
\label{tz23}
\iint\limits_{{\mathfrak C}_{\mu, \mu_0}\times {\mathfrak C}^{\prime\, c}_{\mu}}dX^j_\ast dY^j_\ast\, \frac{|g(X^{j-1}_\ast)-g(Y^{j-1}_\ast)|^p}{|X^j_\ast-Y^j_\ast|^{j+sp}}\to 0\text{ uniformly in }0<\mu<\mu_0\text{ as }\mu_0\to 0.
\eeq

Using \eqref{tz23}, \eqref{tz22} amounts to the existence of some $\mu_0$ such that
\beq
\label{tc1}
J=\iint\limits_{{\mathfrak C}_{\mu, \mu_0}\times {\mathfrak C}^{\prime\, c}_{\mu}}dX^j_\ast dY^j_\ast\, \frac{|g_\mu(X^{j}_\ast)-g(X^{j-1}_\ast)|^p}{|X^j_\ast-Y^j_\ast|^{j+sp}}<\xi\ \text{for every }0<\mu<\mu_0.
\eeq

The key ingredient in the proof of \eqref{tc1} is the following
\begin{lemm}
\label{tc2}
Let $a>j$. Then for ${\mathfrak C}\neq{\mathfrak C}'$ cubes in ${\cal C}\cap {\cal C}_j$ and for $X^j_\ast\in {\mathfrak C}^c_{1/2}$ we have
\beq
\label{tc3}
\int\limits_{{\mathfrak C}^{\prime\, c}_{1/2}}dY^j_\ast\, \frac 1{|X^j_\ast-Y^j_\ast|^{a}}\le C(n, j, a, {\cal C})\frac 1{|X^j_\ast-X^{j-1}_\ast|^{a-j}}.
\eeq
\end{lemm}

\begin{proof}[Proof of Lemma \ref{tc2}]  
When $\overline{\mathfrak C}\cap\overline{{\mathfrak C}'}=\emptyset$, the left-hand side of \eqref{tc3} is bounded from above by a positive constant, and the right-hand side of \eqref{tc3} is bounded from below by a positive constant, so that the conclusion is clear. 
We may thus assume that $\overline{\mathfrak C}\cap\overline{{\mathfrak C}'}\neq\emptyset$. We are then in position to apply estimate \eqref{tw5} (with the roles of $X^j_\ast$ and $Y^j_\ast$ reversed) and infer that 
\beq
\label{te1}
|X^{j-1}_\ast-Y^j_\ast|\le C|X^j_\ast-Y^j_\ast|,\ \forall\, X^j_\ast\in {\mathfrak C}^c_{1/2},\ \forall\, Y^j_\ast\in {\mathfrak C}^{\prime\, c}_{1/2}.
\eeq

On the other hand, since ${\mathfrak C}\cap {\mathfrak C}'=\emptyset$, we clearly have
\beq
\label{te2}
|X^j-X^{j-1}_\ast|\le |X^j_\ast-Y^j_\ast|,\ \forall\, X^j_\ast\in {\mathfrak C},\ \forall\, Y^j_\ast\in {\mathfrak C}^{'}.
\eeq

By \eqref{te1} and \eqref{te2}, we have
\beq
\label{te3}
|X^j_\ast-Y^j_\ast|\ge C(|X^j-X^{j-1}_\ast|+|X^{j-1}_\ast-Y^j_\ast|),\ \forall\, X^j_\ast\in {\mathfrak C}^c_{1/2}, \ \forall\, Y^j_\ast\in {\mathfrak C}'{}^c_{1/2}.
\eeq
If $Z$ is the orthogonal projection of $X^{j-1}_\ast$ on the $j$-dimensional affine plane $\Pi$ spanned by ${\mathfrak C}'$, then \eqref{te3} leads to
\beq
\label{te4}
|X^j_\ast-Y^j_\ast|\ge C(|X^j-X^{j-1}_\ast|+|Z-Y^j_\ast|),\ \forall\, X^j_\ast\in {\mathfrak C}^c_{1/2}, \ \forall\, Y^j_\ast\in {\mathfrak C}^{\prime\, c}_{1/2}.
\eeq

Using \eqref{te4}, we find that
\be*
\int\limits_{{\mathfrak C}^{\prime{}\, c}_{1/2}}dY^j_\ast\, \frac 1{|X^j_\ast-Y^j_\ast|^{a}}\le C\, \int\limits_{\Pi}dY^j_\ast\, \frac 1{[|X^j-X^{j-1}_\ast|+|Z-Y^j_\ast|]^a}=C\frac 1{|X^j-X^{j-1}_\ast|^{a-j}}.
\qedhere
\ee*
\end{proof}
\medskip
\noindent
{\bf Step 5 completed.} Using Lemma \ref{tc2} and a change of variables as in \eqref{td1}, we obtain that the left-hand side $J$ of \eqref{tc1} satisfies, when $0<\mu_0<1/2$, 
\be*
\begin{aligned}
J&\le C \int\limits_{{\mathfrak C}_{\mu, \mu_0}}dX^j_\ast\, \frac{|g_\mu(X^{j}_\ast)-g(X^{j-1}_\ast)|^p}{|X^j_\ast-X^{j-1}_\ast|^{sp}}
=C\, (1-\mu)^{j-sp}\int\limits_{{\mathfrak C}^c_{(\mu_0-\mu)/(1-\mu)}}dX^j_\ast\,\frac{|g(X^{j}_\ast)-g(X^{j-1}_\ast)|^p}{[|X^j_\ast-X^{j-1}_\ast|+\mu/(1-\mu)]^{sp}}\\
&\le C\, \int\limits_{{\mathfrak C}^c_{2\mu_0}}dX^j_\ast\,\frac{|g(X^{j}_\ast)-g(X^{j-1}_\ast)|^p}{|X^j_\ast-X^{j-1}_\ast|^{sp}}\to 0\text{ as }\mu_0\to 0, 
\end{aligned}
\ee*
(here, we use the fact that $g\in {\cal W}^{s,p}_j$) and thus \eqref{tc1} holds.

Step 5 and the proof of Lemma \ref{ug1} are complete.
\end{proof}
\section{ Continuity of the map $g\mapsto h$. Proof of Theorem \ref{thme}}
\l{appe}

Let ${\cal C}$ be a finite submesh of ${\cal C}_n$ and let $j\in\llbracket 0, n-1\rrbracket$. Let $g:{\mathcal C}_j\cap {\cal C}\to \R^m$ and let $h$ be its $j$-piecewise homogeneous extension to $\cal C$.

The main result in this section is the following

\begin{lemm}
\label{cont1} Let $0<s<1$, $1\le p<\infty$ be such that $sp<j+1$. Then we have
\beq
\label{e1}
\|h\|_{L^p(\mc)}^p\le C \|g\|_{L^p(\mc_j\cap\mc)}^p
\eeq
and
\beq
\label{e2}
|h|_{W^{s,p}(\mc)}^p\le C |g|_{W^{s,p}(\mc_j\cap\mc)}^p.
\eeq
[Here, $\|\ \|_{L^p(\mc)}$ and $|\ |_{W^{s,p}(\mc)}$ are naturally defined, and we allow constants depending on $\mc$.]

Equivalently, the map $W^{s,p}(\mc_j\cap\mc)\ni g\mapsto h\in W^{s,p}(\mc)$ is continuous.
\end{lemm}
\begin{proof} 
We may assume that  $T=0$. We use the notation in Section \ref{appb}. In particular, $\omega$ and $\lambda$ will be points in $\R^j$.

\medskip
\noin{\bf Step 1.} Estimate of $\|h\|_{L^p({\cal C})}^p$\\
As in Step 3.1  in  the proof of Lemma \ref{lemb1} (more precisely, by mimicking the derivation of \eqref{equa} with the help of \eqref{defXn}), we have 
\be*
\label{e3}
\|h\|_{L^p({\cal C})}^p=\sum^\circ_{\shortstack{$\scriptstyle \sigma\in S_{n-j,n}$ \\  $\scriptstyle q\in \{-1, 1\}^{n-j}$\\ $\scriptstyle L\in\Z^n$}}\ \ \int\limits_{|\omega|\le 1}d\omega\, k^\sharp(\omega)\,|g(2\ve L+X^j)|^p,
\ee*
where $\d\sum^\circ$ denotes a sum taken only over the $L$'s such that $Q_\ve+2\ve L\in\mc$ and
\be*
\label{e4}
k^\sharp(\omega)=\ve^n\int\limits_{0\leq t\leq 1}
 dt\, t^{n-1}_{1} \ldots  t^j_{n-j} =C(n, j)\, \ve^n.
\ee*
Thus
\be*
\label{e5}
\|h\|_{L^p({\cal C})}^p=C(n, j)\, \ve^n\sum^\circ_{\shortstack{$\scriptstyle \sigma\in S_{n-j,n}$ \\  $\scriptstyle q\in \{-1, 1\}^{n-j}$\\ $L\in\Z^n$}}\ \ \int\limits_{|\omega|\le 1}d\omega\, |g(2\ve L+X^j)|^p  \le C(n, j)\, \ve^{n-j}\int\limits_{\mc_j\cap\mc}|g|^p .
\ee*

\medskip
\noin{\bf Step 2.} Estimate of $|h|_{W^{s,p}({\cal C})}^p$\\
In the spirit of Step 1 above, we have
\be*
\label{e6}
|h|_{W^{s,p}({\cal C})}^p=\sum^\ast_{\shortstack{$\scriptstyle \sigma, \tau\in S_{n-j,n}$ \\  $\scriptstyle q, r\in \{-1, 1\}^{n-j}$\\ $\scriptstyle L, M\in\Z^n$}}\ \ \int\limits_{|\omega|\le 1}d\omega\ \int\limits_{|\lambda|\le 1}d\lambda\, k^\flat_{L, M} (\omega, \lambda)\,|g(2\ve L+X^j)-g(2\ve M+Y^j)|^p,
\ee*
where
$\d\sum^\ast$ denotes a sum taken only over the $L$'s and $M$'s  such that $Q_\ve+2\ve L\in\mc$ and $Q_\ve+2\ve M\in\mc$, 
and we set
\be*
\label{e7}
k^\flat_{L, M}(\omega, \lambda)=\int\limits_{0\leq t\leq 1}dt\ \ \int\limits_{0\leq u\leq 1}du\, 
 t^{n-1}_{1} \ldots  t^j_{n-j} u^{n-1}_{1} \ldots  u^j_{n-j}\, \frac {\ve^{2n}}{|(2\ve L+X^n)-(2\ve M+Y^n)|^{n+sp}}.
\ee*

We rely on the following variant of Lemma \ref{lemb4}.
\begin{lemm} 
\l{leme2} 
Assume that $sp<j+1$. Then
\beq
\label{e8}
k^\flat_{L, M}(\omega, \lambda)\le  \frac{C({\cal C})}{|(2\ve L+X^j) - (2\ve M+ Y^j)|^{j+sp}}.
\eeq
\end{lemm}
\bpr
When $L=M$, the conclusion is given by Lemma \ref{lemb2}. When $|L-M|=1$, this is Lemma \ref{lemb4}. Finaly, when $|L-M|\ge 2$, both sides of \eqref{e8} are bounded from above and from below, with finite positive bounds depending on $\cal C$  (and thus on $\ve$) but independent of $L$, $M$, $X^n$ and $Y^n$.
\epr

\smallskip
\noindent
{\bf Step 2 completed.}
Using Lemma \ref{leme2}, we find that
\be*
\label{e9}
\begin{aligned}
|h|_{W^{s,p}({\cal C})}^p\le & C({\cal C})\sum_{\shortstack{$\scriptstyle \sigma, \tau\in S_{n-j,n}$ \\  $\scriptstyle q, r\in \{-1, 1\}^{n-j}$\\ $\scriptstyle L, M\in\Z^n$}}\ \ \int\limits_{|\omega|\le 1}d\omega\ \int\limits_{|\lambda|\le 1}d\lambda\,  \frac {|g(2\ve L+X^j)-g(2\ve M+Y^j)|^p}{|(2\ve L+X^j)-(2\ve M+Y^j)|^{j+sp}}\\
\le & C({\cal C})\, |g|_{W^{s,p}(\mc_j\cap\mc)}^p.\hskip 115mm\qedhere
\end{aligned}
\ee*
\end{proof} 
We end with the
\begin{proof}[Proof of Theorem \ref{thme}] Theorem \ref{thme} is  a straightforward consequence of Corollary \ref{c1aa}, Lemma \ref{ua1} and Lemma \ref{cont1}.
\end{proof}

\vskip 1cm
\noindent
$^{(1)}$ Department of Mathematics, Rutgers University, 
Hill Center, Busch Campus, 110 Frelinghuysen Rd., Piscataway, NJ 08854, USA\\
brezis@math.rutgers.edu

\bigskip
\noindent
$^{(2)}$ 
Mathematics Department, 
Technion - Israel Institute of Technology, Haifa, 32000, 
Israel
\\
brezis@tx.technion.ac.il

\bigskip
\noindent
$^{(3)}$ 
Universit\'e de Lyon;
Universit\'e Lyon 1;
CNRS UMR 5208 Institut Camille Jordan;
43, boulevard du 11 novembre 1918,
F-69622 Villeurbanne Cedex, France
\\
mironescu@math.univ-lyon1.fr

\end{document}